\documentclass[11pt,a4paper]{article}
\usepackage{amsthm}
\usepackage{amssymb}
\usepackage{amsmath}
\usepackage{amsfonts}
\usepackage{mathrsfs}
\usepackage{enumerate}
\usepackage[top=72pt,bottom=96pt,left=72pt,right=72pt]{geometry}
\newtheorem{lemma}{Lemma}[section]
\newtheorem{remark}[lemma]{Remark}
\newtheorem{example}[lemma]{Example}
\newtheorem{theorem}[lemma]{Theorem}
\newtheorem{corollary}[lemma]{Corollary}
\newtheorem{definition}[lemma]{Definition}
\newtheorem{proposition}[lemma]{Proposition}
\renewenvironment{proof}[1][\unskip]{\par\vskip6pt plus4pt minus3pt%
 \noindent\textbf{Proof #1:\ }}{\hfill\ensuremath\Box}
\newcommand{\Ad}{\mathrm{Ad}}
\newcommand{\ad}{\mathrm{ad}}
\newcommand{\Aff}{\mathrm{Aff}}
\newcommand{\Aut}{\mathrm{Aut}}
\newcommand{\C}{\mathbb{C}}
\newcommand{\CH}{\mathbb{C}\mathrm{H}}
\newcommand{\CP}{\mathbb{C}\mathrm{P}}
\newcommand{\HP}{\mathbb{H}\mathrm{P}}
\renewcommand{\d}{\mathrm{d}}
\newcommand{\diag}{\mathrm{diag}}
\newcommand{\End}{\mathrm{End}}
\newcommand{\frakC}{\mathfrak{C}}
\newcommand{\frakg}{\mathfrak{g}}
\newcommand{\frakh}{\mathfrak{h}}
\newcommand{\frakk}{\mathfrak{k}}
\newcommand{\frakm}{\mathfrak{m}}

\newcommand{\frakp}{\mathfrak{p}}

\newcommand{\bbF}{\mathbb{F}}
\newcommand{\G}{\mathbf{G}}
\newcommand{\GL}{\mathbf{GL}}

\newcommand{\Gr}{\mathbf{Gr}}
\renewcommand{\H}{\mathbb{H}}

\newcommand{\id}{\mathrm{id}}
\renewcommand{\Im}{\mathrm{Im}}
\newcommand{\Isom}{\mathrm{Isom}}
\newcommand{\Jac}{\mathrm{Jac}}

\newcommand{\K}{\mathscr{R}}

\renewcommand{\L}{\Lambda}
\newcommand{\MM}{\mathbf{M}}
\newcommand{\N}{\mathbb{N}}
\newcommand{\NM}{\mathbf{N}}
\renewcommand{\O}{\mathbf{O}}

\newcommand{\Proj}{\mathbf{P}}
\newcommand{\QM}{\mathbf{Q}}
\newcommand{\R}{\mathbb{R}}
\newcommand{\rank}{\mathrm{rank}}
\renewcommand{\Re}{\mathrm{Re}}
\newcommand{\Ric}{\mathrm{ric}}

\newcommand{\rmH}{\mathrm{H}}
\newcommand{\rmK}{\mathrm{K}}

\renewcommand{\S}{\mathbb{S}}
\newcommand{\scal}{\mathbf{s}}
\newcommand{\scrT}{\mathscr{T}}
\newcommand{\SL}{\mathbf{SL}}

\newcommand{\SO}{\mathbf{SO}}
\newcommand{\so}{\mathfrak{so}}
\newcommand{\Sp}{\mathbf{Sp}}
\renewcommand{\sp}{\mathfrak{sp}}
\newcommand{\Spin}{\mathbf{Spin}}
\newcommand{\SU}{\mathbf{SU}}
\newcommand{\su}{\mathfrak{su}}
\newcommand{\Sym}{\mathrm{Sym}}
\renewcommand{\t}{\mathfrak{t}}

\newcommand{\Tr}{\mathrm{Tr}}
\newcommand{\T}{\mathbf{T}}
\newcommand{\trace}{\mathrm{trace}}
\renewcommand{\u}{\mathfrak{u}}
\newcommand{\U}{\mathbf{U}}
\newcommand{\VM}{\mathbf{V}}
\newcommand{\Z}{\mathbb{Z}}

\newcommand{\scrH}{\mathcal{H}}
\newcommand{\scrV}{\mathcal{V}}
\title{Jacobi Relations on Naturally Reductive Spaces}
\author{Tillmann Jentsch%
 \footnote{Lehrstuhl f\"ur Geometrie, Institut f\"ur Geometrie und Topologie,
 Fachbereich Mathematik, Universit\"at Stuttgart, Pfaffenwaldring 57,
 70569 Stuttgart, GERMANY;\ \texttt{tilljentsch@gmail.com}}
 \quad\&\quad Gregor Weingart%
 \footnote{Unidad Cuernavaca del Instituto de Matem\'aticas, Universidad
 Nacional Aut\'onoma de M\'exico, Avenida Universidad s/n, Lomas de Chamilpa,
 62210 Cuernavaca, Morelos, MEXICO;\ \texttt{gw@matcuer.unam.mx}}}
\begin{document}

\maketitle
 \begin{abstract}
  Naturally reductive spaces, in general, can be seen as an adequate
  generalization of Riemannian symmetric spaces.  Nevertheless, there are
  some that are closer to symmetric spaces than others. 
  On the one hand, there is the series of Hopf fibrations over complex space forms,
  including the Heisenberg groups with their metrics of type H. 
  On the other hand, there exist certain naturally reductive spaces in
  dimensions six and seven whose torsion forms have a distinguished algebraic property.
  All these spaces generalize geometric or algebraic properties of $3$--dimensional
  naturally reductive spaces and have the following point in common: 
  along every geodesic the Jacobi operator satisfies an ordinary differential
  equation with constant coefficients which can be chosen independently of
  the given geodesic.

 \begin{center}
  \parbox{300pt}{\textit{MSC 2010:}\ 53C21; 53C25; 53C30.}
  \\[5pt]
  \parbox{300pt}{\textit{Keywords:}\ 
    Naturally reductive spaces, Jacobi relations, generalized twistor equation, symmetric Killing tensors,
    vector cross products, positive sectional curvature, Berger spheres, Heisenberg groups, 
    nearly K\"ahler manifolds, nearly parallel $\G_2$--spaces}
 \end{center}

 \end{abstract}

\section{Introduction}
\label{se:intro}
 A complete Riemannian manifold $(\,M,\,g\,)$ 
 with Levi-Civita connection $\nabla$ and curvature tensor $R$ 
 is called a homogeneous space if some Lie group
 acts transitively on $(\,M,\,g\,)$ by isometries. The first intrinsic
 characterization of simply connected homogeneous
 Riemannian manifolds was given in~\cite{Sin}:

 \begin{theorem}\label{th:Singer}[I.~M.~Singer]
 \hfill\break
  A simply connected complete Riemannian manifold $(\,M,\,g\,)$ is 
  homogeneous if and only if the curvature jets 
  $(\,R|_p,\,\nabla R|_p,\,\nabla^2 R|_p,\,\ldots,
  \,\nabla^{k+1}R|_p\,)$ and $(\,R|_q,\,\nabla R|_q,\,\nabla^2R|_q,\,\ldots,
  \,\nabla^{k+1}R|_q\,)$ are linearly equivalent up to a certain order
  $k+1$ for every two points $p$ and $q$; the necessary order depends
  on the so--called Singer invariant $k\,:=\,k_{\textrm{Singer}}$ of $M$.
 \end{theorem}

 \noindent
 Here, $\nabla^kR$ denotes the $k$--fold iterated covariant derivative of the 
 curvature tensor. The Singer invariant $k_{\textrm{Singer}}\,<\,N$ is bounded
 from above by the maximal length $N\,\leq\,{n\choose 2}$ of a strictly 
 decreasing sequence of Lie subalgebras of $\so(\,n\,)$ where $n$ is the
 dimension of $M$. In particular, $(\,M,\,g\,)$  is homogeneous if and only if 
 the curvature jets of order ${n\choose 2}$ at two arbitrary points are conjugate by a linear isometry.

 \bigskip\noindent
 A seemingly different approach to the homogeneity of   
 Riemannian manifolds was given by Ambrose and Singer in~\cite{AS}. 
 It was shown that a simply connected complete  Riemannian manifold is 
 homogeneous if and only if it admits a (not necessarily
 unique) metric connection $\bar\nabla$ with parallel torsion and curvature
 tensor. Every such a connection $\bar\nabla$ is called an Ambrose--Singer
 connection, and the difference tensor $\bar\nabla\,-\,\nabla$ of
 type $(1,2)$ is called a homogeneous structure (cf. also~\cite{Bal,TV}).
 In~\cite{NTr} it was shown how to derive Theorem~\ref{th:Singer}
 from the Ambrose--Singer Theorem.

 \bigskip
 \noindent
 Maybe a flaw can be seen in the fact that the previously mentioned
 characterizations of homogeneity are not related to explicit equations on
 the Riemannian curvature tensor and its iterated covariant derivatives.
 The parallelism $\nabla R\,=\,0$ of the Riemannian curvature tensor 
 is an example of such an equation since it characterizes simply connected 
 Riemannian symmetric spaces among all simply connected complete 
 Riemannian manifolds. Comparable conditions which produce a broader 
 class of homogeneous spaces are seemingly not known. The most
 obvious generalization of the equation $\nabla R\,=\,0$, the vanishing of
 $\nabla^kR$ for some $k\,\geq\,2$, fails for the following reason:
 If $\nabla^kR\,=\,0$, then the curvature tensor has polynomial growth of
 order $k\,-\,1$ along each geodesic, whereas the norm of $R$ is globally
 bounded on a compact $M$. Therefore, it is immediately clear that on a
 compact Riemannian manifold $\nabla^kR\,=\,0$ for some $k\,\geq\,2$ implies
 that $\nabla R\,=\,0$. In~\cite{NO} it was shown that this problem remains,
 even if the compactness assumption on $M$ is omitted. If one's hope is to
 characterize homogeneous Riemannian manifolds by means of natural equations
 satisfied by the curvature tensor and its iterated higher covariant
 derivatives, then a better idea is clearly in need.

 \bigskip
 \noindent
 An interesting observation made during the last 10 years is the existence
 of linear Jacobi relations on certain non--symmetric homogeneous spaces. 
 Moreover, every such linear Jacobi relation of order $k$ implies that the
 component of $\nabla^{k+1} R$ in the summand with the highest possible weight 
vanishes in a tensor space decomposition with respect to the natural action
 of the Lie algebra $\so(\,n\,)$ of the orthogonal group. Therefore, the trace--free
 part of the symmetrized $k$--fold iterated covariant derivative $\nabla^k R$
 of the curvature tensor satisfies a generalized twistor equation, see
 Section~\ref{se:GT}.

 \bigskip
 \noindent
 Acknowledgments:\ \ \ We would like to thank A.~M.~Naveira for drawing our
 attention to the concept of Jacobi relations. Also we are grateful to R.~Storm 
 for explaining his results to us and to S.~Noshari for helpful comments on 
 an earlier version of this article.
\section{Jacobi Relations on Riemannian Manifolds}
\label{se:JR_RM}
\sloppy
 Let $(\,M,\,g\,)$ be a Riemannian manifold with Levi--Civita connection
 $\nabla$ and curvature tensor $R$. We denote by $\Sym^kT^*M$ and 
 $\End^{\mathrm{sym}}\,TM$ the vector bundles of symmetric $k$--tensors 
 and symmetric endomorphisms of $TM$, respectively. 
 Elements of  $\Sym^kT_p^*M$ are as usual seen as polynomial functions of 
 degree $k$ on $T_pM$. The section $\K_k$ of $\Sym^{k+2}T^*M\otimes 
 \End^{\mathrm{sym}}\,TM$ defined by
 \begin{equation}\label{eq:SCD}
  X\;\;\longmapsto\;\;\Big(\;U\;\longmapsto\;
  \K_k(\,X\,)U\;:=\;(\nabla^k_{X,\,\ldots,\,X}R)_{U,\,X}X\;\Big)
 \end{equation}
 is called the symmetrized $k$--th covariant derivative of $R$. Its significance
 is underlined by the well--known Jet Isomorphism Theorem: Accordingly,
 the Taylor expansion up to order $k\,+\,2$ of the metric tensor in  normal coordinates  
 at some point $p$ is a universal tensorial 
 expression in $(\,\K_0|_p,\,\K_1|_p,\,\K_2|_p,\, \cdots,\, \K_k|_p\,)$. 
 Hence, this sequence carries exactly the same information as the usual 
 $k$--jet $(\,R|_p,\,\nabla R|_p,\,\nabla^2 R|_p,\,\ldots, \,\nabla^kR|_p\,)$ 
 considered in Theorem \ref{th:Singer}. 

 \bigskip
 \noindent
 In terms of the higher covariant derivatives of $R$, the Gau\ss\ Lemma for 
 the exponential map reads $\K_k(X)X\,=\,(\,\nabla^k_{X,\,\cdots,\,X}R\,)_{X,\,X}X\,=\,0$,
 cf. \cite[Chpt. 3]{JW1}. In other words $\K_k$ is actually a section 
 of the $\Z$--graded vector bundle  $\Jac^\bullet TM$ defined by 
  the short exact sequence:
 $$
  0\;\longrightarrow\;
  \Jac^\bullet TM
  \;\stackrel\subset\;\longrightarrow\;
  \Sym^{\bullet+2} T^*M\otimes\End^{\mathrm{sym}}\,TM
  \;\stackrel{\mathrm{cl}}\;\longrightarrow\;
  \Sym^{\bullet+3}T^*M\otimes TM
 $$
 where the closure $\mathrm{cl}\,\K$ of a polynomial map $\K\,:\,T_pM\;\longrightarrow\; 
 \End^{\mathrm{sym}}\,T_pM\,,\, X\;\longmapsto \;\K(\,X\,)$ is given by $\mathrm{cl}\,\K\,:\, 
 T_pM\;\longrightarrow\;T_pM\,,\, X\;\longmapsto \;\K(\,X\,)\, X$. 
 Therefore, the elements of $\Jac^k T_pM$ have the index symmetries described by a Young
 diagram with two rows of lengths $k+2$ and $2$ respectively (cf.~\cite[Chpt. 6]{FulH}). 
 The relevance of this fact will only become clear in Section \ref{se:GT}. 

 \bigskip
 \noindent
 The graded vector bundle $\Jac^\bullet TM$ is evidently a
 graded module over the graded algebra bundle $\Sym^\bullet T^*M$
 of polynomials on $TM$, in turn the graded vector space
 $\Gamma(\,\Jac^\bullet TM\,)$ is a graded module over
 $\Gamma(\,\Sym^\bullet T^*M\,)$.

 \begin{definition}[Jacobi Relations on Riemannian Manifolds]
 \hfill\label{de:JR}\break
  A Jacobi relation of order $d\,\in\,\N$ on a Riemannian manifold $(M,\,g)$
  is a linear relation in the module $\Gamma(\,\Jac^{d+1}TM\,)$ with homogeneous 
  coefficient polynomials $a_1,\,\ldots,\,a_{d+1}$ in the sense 
  $a_k\,\in\,\Gamma(\,\Sym^kT^*M\,)$:
  \begin{equation}\label{eq:JR}
   \K_{d+1}
   \;\;=\;\;
   -\,a_1\,\K_d\;-\;a_2\,\K_{d-1}\;-\;\cdots\;-\;a_{d+1}\,\K_0
  \end{equation}
  where $\K_k$ is the symmetrized $k$--th covariant derivative 
  of $R$ defined in \eqref{eq:SCD}.
 \end{definition}

 \noindent
 For every geodesic $\gamma$, consider the Jacobi operator $\K_0(\,\dot\gamma\,):\,
 U\longmapsto R_{U,\,\dot\gamma}\dot\gamma$ along $\gamma$ as well as its
 iterated covariant derivatives $\K_k(\,\dot\gamma\,)\,=\,\frac{\nabla^k}{dt^k}
 \K_0(\dot\gamma)$ for all $k\,\in\,\N_0$. A Jacobi relation of the form
 (\ref{eq:JR}) exists if and only if for every geodesic 
 $\gamma$ in $M$ it holds true that:
 \begin{equation}\label{eq:JR_original}
  \K_{d+1}(\,\dot\gamma\,)
  \;\;=\;\;
  -\, a_1(\,\dot\gamma\,)\, \K_d(\,\dot\gamma\,)
  \;-\;a_2(\,\dot\gamma\,)\,\K_{d-1}(\,\dot\gamma\,)
  \;-\;\cdots\;-\;a_{d+1}(\,\dot\gamma\,)\,\K_0(\,\dot\gamma\,)
 \end{equation}

 \begin{remark}\label{re:JR}
  Let $(\,M,\,g\,)$ be an arbitrary homogeneous Riemannian space 
  and let $G\,:=\,\Isom(\,M,\,g\,)$. 
  Then, there exists a Jacobi relation with $G$--invariant coefficients for the following 
  reason: Since $\Sym^{\bullet+2} T^*_pM\otimes\End^{\mathrm{sym}}\,T_pM$ is a 
  finitely generated module over $\Sym^\bullet T^*_pM$\,, it is a Noetherian 
  $\Sym^\bullet  T^*_pM$--module by Hilbert's Basis Theorem. This property 
  guaranties that the ascending chain of $\Jac^\bullet T_pM$--submodules 
  \begin{equation}\label{eq:chain}
   \{\K_0|_p\}_{\Sym^\bullet T^*_pM}
   \;\subset\;
   \{\K_0|_p\;,\;\K_1|_p\}_{\Sym^\bullet T^*_pM}
   \;\subset\;
   \{\K_0|_p\;,\;\K_1|_p\;,\;\K_2|_p\}_{\Sym^\bullet T^*_pM}
   \;\subset\;
   \cdots
  \end{equation}
  becomes stationary,  i.e., there exists $k\,\in\,\N$ such that
  $$
   \{\K_0|_p\;,\;\K_1|_p\;,\;\cdots\;,\;\K_k|_p\}_{\Sym^\bullet T^*_pM}
   \;\;=\;\;
   \{\K_0|_p\;,\;\K_1|_p\;,\;\cdots\;,\;\K_{k+1}|_p\}_{\Sym^\bullet T^*_pM}
  $$
  Hence, there exist coefficients $a_0\,,\,\ldots,\,a_k\,\in\,
  \Sym^\bullet T^*_pM$ such that 
  $$
   \K_{k+1}|_p
   \;\;=\;\;
   \sum_{\ell=0}^k a_\ell \;\K_\ell|_p
  $$ 
   Because the isotropy subgroup $H\subset G$ at $p$ is compact,
   the Haar measure of $H$ is a finite bi--invariant Borel measure.
   In particular we can assume that $|H|\,=\,1$. Integrating the previous equation
   with respect to the canonical left action $\star$ of $H$ on tensors gives
  $$
  \int_H h\star \K_{k+1}|_p\;\d\,h
   \;\;=\;\;
    \int_H h\star \sum_{\ell=0}^k a_\ell \;\K_\ell|_p \;\d\,h 
  $$
  From the linearity of the integral and the $H$--invariance
  $h\star \K_\ell\,=\,\K_\ell$ for all $h\in H$, we see that
  $$
   \K_{k+1}|_p\;=\; \sum_{\ell=0}^k \tilde a_\ell \;\K_\ell|_p
  $$
  where $\tilde a_\ell(\,x\,)\,:=\, \int_H a_j(\,h\, x\,)\d\,h$ is
  $H$--invariant. Therefore, we can assume without loss of generality that
  $a_\ell$ shares this property. Since there is a one--one correspondence between
  $H$--invariant tensors at $p$ and $G$--invariant tensors on $M$, we
  obtain a Jacobi relation with $G$--invariant coefficients. 
 \end{remark} 
  
 \noindent
 Let $G$ be a connected Lie group and $H\,\subset\, G$ be a closed subgroup. 
 The pair $(\,G,\,H\,)$ is called an effective homogeneous pair if $G$ acts effectively from the left on the orbit space $
 M\,:=\,G/H$,  i.e., $G$ is a subgroup of the diffeomorphism group of $M$.
 Although one can assume without loss of generality that a homogeneous pair is effective (simply by passing 
 to the images of $G$ and $H$ in the diffeomorphism group of $G/H$), 
 it is often more convenient to have presentations $\tilde G/\tilde H$ where the kernel of the left action 
 of $\tilde G$ on $\tilde G/\tilde H$ is discrete. Then, $(\,\tilde G,\,\tilde H\,)$ is called almost effective. 
 A well--known example is the almost effective presentation $\SU(\,3\,)/\U(\,2\,)$ 
 of the complex projective space $\CP^2$: The center of $\SU(\,3\,)$ is a cyclic group of order 
 three given by the diagonal matrices $\diag(\,e^{\frac{2}{3}k\,\pi\,i\,},\,e^{\frac{2}{3}k\,\pi\,i\,},
 \,e^{\frac{2}{3}k\,\pi\,i\,}\,)$ with $k\,=\,0,\,1,\,2$. 
 This group is contained in $\U(\,2\,)$ and hence acts trivially on $\SU(\,3\,)/\U(\,2\,)$.
 
 \bigskip
 \noindent
 Suppose next that there exists a non--degenerate $\Ad(G)$--invariant bilinear form 
 $B$ on $\frakg$ inducing a $B$-orthogonal direct splitting $\frakg\,=\,\frakh\oplus\frakm$
 such that $B|_{\frakm\times\frakm}$ is positive definite, where $\frakh$ and $\frakg$ denote the 
 Lie algebras of $H$ and $G$, respectively. The splitting is
 automatically reductive, i.e., $\Ad_H(\,\frakm\,)\subset \frakm$, and
 $(\,G,\,H,\,B\,)$ is called a naturally reductive triple. We will call this triple (almost) effective, 
 if $(\,G,\,H\,)$ is (almost) effective. 

 \begin{definition}\label{de:NRS}
 Let $(\,M,\,g\,)$ be a complete Riemannian manifold.
 \begin{enumerate}
  \item 
  A reductive structure is a metric affine connection $\bar\nabla$ such that the $\bar\nabla$--parallel translation 
  $$
   \big(\,\parallel_0^1 \gamma\,\big )^{\bar\nabla}\;:\;T_{\gamma(0)}M\,\longrightarrow\,T_{\gamma(1)}M
  $$  
  along an arbitrary curve $\gamma\colon [\,0,\,1\,]\,\to\, M$ can be realized by an element $a$ 
  of the group $\Aff(\,M,\, \bar\nabla\,)$ of affine transformations: 
  $$
   a(\,\gamma(0)\,)\;=\;\gamma(1)\ \text{and}\ \big(\,\parallel_0^1 \gamma\,\big )^{\bar\nabla}\,=\,\d_{\gamma(0)}a
  $$
  Then, $a$ is uniquely defined by $\gamma$ and an isometry of $(\,M,\,g\,)$. The connected subgroup of
  $\Isom(\,M,\,g\,)\,\cap\,\Aff(\,M,\, \bar\nabla\,)$ 
  generated by the $a$'s is called the transvection group and will be denoted by $\Tr(\,\bar\nabla\,)$.
  \item
  A reductive structure $\bar\nabla$ is called naturally reductive if the $3$--tensor $\tau$ defined by the equation 
  $\bar\nabla\,=\, \nabla\,+\,\frac{\tau}{2}$ is an alternating $3$--form. Then, $\tau$ is called the torsion form.
  \end{enumerate}
 The triple $(\,M,\,g,\,\bar\nabla\,)$ is also called a (naturally) reductive space.
 \end{definition}

 \noindent
 Clearly, every reductive structure of $(\,M,\,g\,)$ is an Ambrose--Singer connection as considered in Section \ref{se:intro}.
 Conversely, if $M$ is simply connected,  then every  Ambrose--Singer connection is a reductive structure, cf. \cite{AS}.
 Given an effective naturally reductive triple $(\,G,\,H,\,B\,)$ we consider the Riemannian manifold
 $(\,M\,:=\,G/H,\,g\,:=\,B|_{\frakm\times\frakm}\,)$ and the canonical connection $\bar\nabla$ 
 associated with $(\,G,\,H,\,B\,)$. It is well known that this establishes a 1--1 correspondence 
 between effective naturally reductive triples $(\,G,\,H,\,B\,)$ and naturally reductive spaces
 $(\,M,\,g\,,\,\bar\nabla\,)$ such that $G\,\cong\,\Tr(\,\bar\nabla\,)$, see \cite[Theorems 1 and 2]{St1}. 
 Accordingly, a naturally reductive space is called normal if $B$ is positive definite,
 and standard normal if the Killing form $\rmK_\frakg$ of the Lie algebra of $G$ is negative definite 
 and there exists $c\,>\,0$ such that $B\,=\,-\,c^2\,\rmK_\frakg$. 
 Simply connected naturally reductive spaces were recently classified up to dimension eight, cf.~\cite{AFF,St3}.

 \bigskip
 \noindent
 The following theorem shows that a naturally reductive
 structure is often already uniquely determined by the underlying
 Riemannian structure:

 \begin{theorem}[\cite{OR1}]\label{th:OR}
  Let  $(M,\,g)$ be a simply connected irreducible compact Riemannian manifold
  which is not isometric neither to a euclidean sphere nor to a simple compact Lie
  group with its bi--invariant metric. Then, there exists at most one 
  naturally reductive structure $\bar\nabla$. In the affirmative case, the isometry group
  $\Isom(\,M,\,g\,)$ is contained in the group $\Aff(\,M,\, \bar\nabla\,)$ of affine transformations.
  Hence, the connected components of these groups coincide 
  $\Isom(\,M,\,g\,)_\circ\,=\,\Aff(\,M,\, \bar\nabla\,)_\circ$. 
  Therefore, the transvection group $\Tr(\,\bar\nabla\,)$ 
  is necessarily a normal subgroup of $\,\Isom(\,M,\,g\,)$. 
  If, moreover, $(M,\,g)$ is non--symmetric, then  the holonomy group of $(M,\,g)$ is generic.
 \end{theorem}
 
 \begin{proof}
  The strong Skew--torsion Holonomy Theorem~\cite[Theorem~1.4]{OR1} implies 
  that $\bar\nabla$ is uniquely determined by $g$, cf.~\cite[Theorem~1.2]{OR1}. 
  In particular, we have $\Isom(\,M,\,g\,)\, \subset\, \Aff(\,M,\,\bar\nabla\,)$ 
  by the uniqueness of $\bar\nabla$, cf.~\cite[Theorem~1.1~(ii)]{OR1}. 
  Conversely, one has $\Aff(\,M,\, \bar\nabla\,)_\circ\,
  \subset\,\Isom(\,M,\,g\,)$ on a compact naturally reductive space,
  see \cite[Remark 6.3]{OR1}. Thus, $\Isom(\,M,\,g\,)_\circ\,
  =\,\Aff(\,M,\,\bar\nabla\,)_0$. Also it is known that $G$ is always a normal 
  subgroup of $\Aff(\,M,\, \bar\nabla\,)$, cf. \cite[Chpt. 7]{OR1}. Moreover, as a further consequence of 
  the strong Skew--torsion Holonomy Theorem, the Riemannian holonomy group of a 
  non--symmetric simply connected irreducible compact naturally 
  reductive space is genericm, cf. \cite[Remark 6.6]{OR1}.
 \end{proof}

\bigskip
\noindent
 Let $(\,M^n,\,g,\,\bar\nabla\,)$ be a naturally reductive space with torsion
 form $\tau$. For each $X\,\in\,T\,:=\,T_pM$, let  $\tau_X$ denote the skew--symmetric endomorphism
 of $T$ characterized by $\langle\, \tau_X(\,Y\,),\,Z\,\rangle \,:=\,\tau(\,X,\,Y,\,Z\,)$ and consider the linear operator 
 $\scrT(\,X\,)\,:\,\End^{\mathrm{sym}}\,T\,\longrightarrow\, \End^{\mathrm{sym}}\,T$ with:
 \begin{equation}
 \label{eq:def_of_derivation}
 \scrT(\,X\,)\, \K
 \;\;:=\;\;
 U\;\longmapsto\;
 \,\frac12\,\Big(\;\K\,\tau_X\,U \,-\,\tau_X\,\K\,U\;\Big)
 \end{equation}
 By definition, $-2\,\scrT(\,X\,)$ is the usual action $\,\tau_X\star \K\,:=\,[\,\tau_X,\,\K\,]$ of 
 $\tau_X$ on $\K\,\in\,\End^{\mathrm{sym}}\,T$ via the commutator bracket. 
 Let $\{X\}^\perp$ denote the orthogonal complement of $\{X\}$ in $T$ and
 note that $\tau_X\,\in\End^{\mathrm{skew}}\,\{X\}^\perp$ 
 due to $\tau_X\,X\,=\,0$. Therefore, we have the induced map $\scrT(\,X\,)|_{\End^{\mathrm{sym}}\,\{X\}^\perp\,}\,:
 \,\End^{\mathrm{sym}}\,\{X\}^\perp\,\longrightarrow\, \End^{\mathrm{sym}}\,\{X\}^\perp$.  
 Then, $P_{\tau}(\,\lambda;X\,)\,:=\,\det\big(\lambda\,\id\,-\,\scrT(\,X\,)|_{\End^{\mathrm{sym}}\,\{X\}^\perp} \big)$
 is the characteristic polynomial of $\scrT(\,X\,)|_{\End^{\mathrm{sym}}\,\{X\}^\perp}$. 
 Note that $\dim\,\End^{\mathrm{sym}}\,\{X\}^\perp\,=\,{n \choose 2}\,=:\,d+1$ if $X\,\neq\,0$. Also
 let $\vartheta \,:=\,2\, \lceil \frac{d}{2}\rceil$; this is the maximal even number less than or equal to $d+1$. 
We will prove the following theorem:

 \begin{theorem}
 \label{th:JR}
  Let $(\,M,\,g,\,\bar\nabla\,)$ be an $n$--dimensional naturally reductive space with torsion form $\tau$.
  There exists a Jacobi relation of order $d\,=\,{n\,\choose 2}\,-\,1$ whose coefficients $a_\ell$ 
  are $\Tr(\,\bar\nabla\,)$--invariant symmetric  Killing tensors of degree $\ell$ which vanish for $\ell$ odd and satisfy
  $P_\tau(\,\lambda;\,X\,)\,=\,\lambda^{d\,+\,1}\,+\, \sum_{k\,=\,1}^{\frac{\vartheta}{2}} a_{2\,k}(\,X\,)\,\lambda^{d\,+\,1\,-\,2\,k}$
  for each $0\,\neq\,X\,\in\,T_pM$.
 \end{theorem}

 \noindent
 We recall that a symmetric tensor field $a\,\in\,\Gamma(\,\Sym^k\,T^*M\,)$
 is called a Killing tensor provided it satisfies the partial differential equation 
 $(\nabla_Xa)(\,X\,)\,=\,0$ for all $X\,\in\,T_pM$ \cite{HMS} (where we see $a_p$ as usual as a 
 polynomial map of degree $k$ on $T_pM$). Killing tensors are as well 
 characterized by the property that for every geodesic $\gamma(\,t\,)$ the function 
 $a(t) := a(\dot\gamma(\,t\,))$ is constant. Also a Riemannian manifold $(\,M,\,g\,)$ 
 is called a  \textsc{g.o.}--space  if every geodesic $\gamma(\,t\,)$ is the orbit
 of the $1$--parameter group of isometries generated  by some Killing vector field,
 cf. \cite[Chpt. 2.2]{BTV}. Clearly, the most prominent examples of \textsc{g.o.}--spaces 
 are naturally reductive spaces. It is straightforward that \textsc{g.o.}--spaces are homogeneous Riemannian 
 manifolds, and that every $\Isom(\,M,\,g\,)_\circ$--invariant symmetric tensor field is Killing.  
 Therefore, it follows from Remark \ref{re:JR} that also \textsc{g.o.}--spaces satisfy a Jacobi relation 
 whose coefficients are  Killing tensors, i.e., the coefficients $a_k(\dot \gamma(\,t\,))$ of Equation
 \eqref{eq:JR_original} are constant in the geodesic parameter. However, in general
 these coefficients, and hence the differential equation for $\K_\gamma$, 
 will depend on the chosen unit speed geodesic. We will now specify the case where this does not happen.

 \subsection{Linear Jacobi Relations}
 \label{se:LJR_RM}
 More precisely we consider Jacobi relations on Riemannian manifolds 
 where the polynomial coefficients are multiples of symmetric powers of 
 the metric tensor:

 \begin{definition}[Linear Jacobi Relations]
 \hfill\label{de:LJR}\break
  A linear Jacobi relation of order $d\,\in\,\N$ on a Riemannian manifold
  $(\,M,\,g\,)$ is a Jacobi relation in the sense of Definition \ref{de:JR}
  in which all coefficients are multiples of the metric tensor $g$:
  \begin{equation}\label{eq:LJR}
   \K_{d+1}
   \;\;=\;\;
   -\,a_2\,\|\,\cdot\,\|^2\,\K_{d-1}
   \;-\;a_4\,\|\,\cdot\,\|^4\,\K_{d-3}
   \;-\;\cdots
   \;-\;a_{\vartheta}\,\|\,\cdot\,\|^\vartheta\,\K_{d+1-\vartheta}
  \end{equation}
  with real coefficients $a_2,\,a_4,\,\ldots\,,a_\vartheta\,\in\,\R$ where $\vartheta \,:=\,2\, \lceil \frac{d}{2}\rceil$. Equivalently,
  \begin{equation}
  \label{eq:LJR2}
   \K_{d+1}(\,\dot\gamma\,)
   \;\;=\;\;
   -\,a_2\,\|\,\dot\gamma\,\|^2\ \K_{d-1}(\,\dot\gamma\,)
    -\,a_4\,\|\,\dot\gamma\,\|^4\ \K_{d-3}(\,\dot\gamma\,)
   \;-\;\cdots
   \;-\;a_\vartheta\,\|\,\dot\gamma\,\|^\vartheta\,\K_{d+1-\vartheta}
   (\,\dot\gamma\,)
  \end{equation}
  along every geodesic $\gamma$ of $(\,M,\,g\,)$. 
 \end{definition}
 \noindent
 If there exists a linear Jacobi relation, then it is obvious that
 there exists a unique minimal linear Jacobi relation, i.e., one whose order $d$ is minimal.
 For example, a symmetric space satisfies $\nabla R\,=\,0$; thus, $\K_1\,=\,0$
 is a linear Jacobi relation, which is minimal unless $R\,=\,0$.
 
\bigskip\noindent
 In order to understand first basic properties of linear Jacobi relations,
 note that \eqref{eq:LJR2} becomes an ordinary differential equation
 with constant coefficients $P(\,\frac d{dt}\,)\K_\gamma^{ij}(\,t\,)\,=\,0$ for the matrix
 coefficients $\K_\gamma^{ij}$ with respect to a $\nabla$--parallel
 frame along $\gamma$ where $P$ is the polynomial 
 \begin{equation}\label{eq:p}
  P(\,\lambda\,)
  \;\;:=\;\;
  \lambda^{d+1}\;+\;a_2\,\lambda^{d-1}\;+\;a_4\,\lambda^{d-3}\;+\;\cdots\;+\;
  \begin{cases}
   a_d\lambda & \textrm{if $d$ is even}
   \\
   a_{d+1} & \textrm{if $d$ is odd}
  \end{cases}
 \end{equation}

 \begin{proposition}\label{p:LJR_STRUK}
   Let $(\, M,\,g\,)$ be a compact Riemannian manifold or a naturally
   reductive space. Suppose there exists a linear Jacobi relation \eqref{eq:LJR}
   and let $P(\,\lambda\,)$ be the polynomial corresponding to the minimal
   linear Jacobi relation.
  \begin{enumerate}
   \item Each root of $P$ is purely
         imaginary and simple. Therefore, there exist strictly positive
         numbers $0 \,< \, \lambda_1 \, < \, \cdots \, < \,\lambda_{\frac{\vartheta}{2}}$
         such that $P(\,\lambda\,)\,=\,Q(\,\lambda\,)$ or $P(\,\lambda\,)\,=\,\lambda\,
         Q(\,\lambda\,)$ where $Q(\,\lambda\,)\,:=\,\prod_{\ell=1}^{\frac{\vartheta}{2}} (\lambda^2 + \lambda_\ell^2)$.
  \item If the Ricci tensor $\Ric(\,X\,)\,:=\, \trace\,\K_0(\,X\,)$ is a non--vanishing Killing tensor, then
         $P(\,\lambda\,)\,=\,\lambda\,Q(\,\lambda\,)$. In particular, then $d$ is even.
  \end{enumerate}
  \end{proposition}
 \begin{proof}
 On a compact Riemannian manifold the norm of the Riemannian curvature
 tensor is bounded. Therefore, the matrix elements $\K_\gamma^{ij}$
 with respect to a $\nabla$--parallel orthonormal frame along a geodesic $\gamma$ are
 bounded functions in the geodesic parameter. The same is true for naturally
 reductive spaces. In fact, here $R$ is parallel with respect to the Ambrose--Singer connection 
 $\bar\nabla$ and geodesics of $\nabla$ are geodesics of $\bar\nabla$. 
 Hence, the matrix coefficients of $\K_\gamma$ are constant
 with respect to a $\bar\nabla$--parallel orthonormal frame along $\gamma$.
 Thus, the $\K_\gamma^{ij}$ are bounded functions, too. 
 Furthermore, the space of solutions of the ODE
 $P(\,\frac d{dt}\,)f(\,t\,)\,=\,0$ is spanned by functions of the form
 $P_\mu(\,t\,)e^{\mu t}$ where $\mu$ is a root of $P$ and $P_\mu(\,t\,)$ is a
 polynomial function in $t$, whose degree is strictly smaller than the multiplicity
 of $\mu$. The space of bounded solutions for $P(\,\frac d{dt}\,)f(\,t\,)\,=\,0$
 is thus spanned by functions of the form $c_\mu e^{\mu t}$ where $\mu$ is a
 purely imaginary root of $P$ and $c_\mu$ is a constant. Since $P$ is a
 real polynomial, there exist real numbers $0\, <\, \lambda_1\, <\, \cdots \,
 < \, \lambda_{\frac{\vartheta}{2}}$ such that the non--zero roots $\mu_i$ of $P$
 are given by $\pm i\lambda_i$. By minimality of $P$, we have  
 $P(\,\lambda\,)\,=\, Q(\,\lambda\,)$ or $P(\,\lambda\,)\, =
 \,\lambda\,Q(\,\lambda\,)$ depending on whether zero is a root of $P$ or not.

 \bigskip
 \noindent
 Assume now that the Ricci tensor is Killing and $P(\,\lambda\,)\,=\,Q(\,\lambda\,)$. We claim that $\Ric\,=\,0$: Let $\gamma$ 
 be a geodesic line. Taking the trace in \eqref{eq:LJR2}, we obtain
 \begin{equation*}
  (\;\frac{d^2}{dt^2}\, + \,\lambda_1^2\;) \cdot\; \cdots\; \cdot
  (\;\frac{d^2}{dt^2}\, + \,\lambda_{\frac{\vartheta}{2}}^2\;)
  \Ric(\;\dot\gamma,\;\dot\gamma\;)
  \;\;\equiv\;\;
  0
 \end{equation*}
 Since $\Ric(\,\dot\gamma,\,\dot\gamma\,)$ is constant by the Killing
 property, this reduces to
 \begin{equation*}
  \lambda_1^2\cdot\;\cdots\;\cdot\lambda_{\frac{\vartheta}{2}}^2\;\cdot\;
  \Ric(\;\dot\gamma,\;\dot\gamma\;)
  \;\;\equiv\;\;
  0
 \end{equation*}
 Because the $\lambda_\ell^2$ are strictly positive, we conclude that $\Ric\,=\,0$.
 \end{proof}

\bigskip\noindent
 For a naturally reductive space we will give in Section \ref{se:th:LJR} an algebraic proof of 
 Proposition \ref{p:LJR_STRUK}. Note that here the Ricci tensor is automatically a Killing tensor.

\bigskip\noindent
 Each symmetric space provides an example of a naturally reductive space with a linear Jacobi relation of
 order zero. In contrast experiments with computational programs like Maple
 or Mathematica indicate that on non--symmetric naturally reductive spaces linear
 Jacobi relations occur only very scarcely. Consequently, the following theorem should be 
 seen as the central result of this article:
 
 \begin{theorem}[Main Theorem]
 \label{th:LJR}
  On the following naturally reductive spaces there exist linear Jacobi
  relations:
  \begin{enumerate}
   \item The Berger metric on the $(\,2\,n\,+\,1\,)$--dimensional
           total space $\hat \MM^n(\,\kappa,\,r\,)$ 
           of the Hopf fibration of circles of radius $r$ over 
           an $n$--dimensional complex space form
           $\MM^n(\,\kappa\,)$ of holomorphic sectional curvature 
           $\kappa\neq 0$ satisfies
           \begin{equation}\label{eq:LJR_BS}
            \K_3
            \;\;=\;\;
            -\;c^2\,\|\,\cdot\,\|^2\,\K_1 
           \end{equation}
           where $c^2\,:=\,\frac14\,r^2\kappa^2$. 
           Also for each $c\,>\,0$ there exists  the same linear Jacobi relation \eqref{eq:LJR_BS} on 
           the $(\,2\,n\,+\,1\,)$--dimensional Heisenberg group $\hat \MM^n(\,0,\,\frac{1}{c^2}\,)$
           equipped with a multiple  $\frac{1}{c^2}$ of the canonical left--invariant 
           metric of type H.
   \item For a $7$--dimensional simply connected standard normal homogeneous space of positive
           sectional curvature and scalar curvature $\scal$ we have
           \begin{equation}\label{eq:LJR_NP}
            \K_3
            \;\;=\;\;
            -\;\frac{2\,\scal}{189}\,\|\,\cdot\,\|^2\,\K_1 
           \end{equation}
           except for $\S^7\, = \,\SU(\,4\,)/\SU(\,3\,)$ and $\S^7\,=\, \Sp(\,2\,)/\Sp(\,1\,)$.
   \item For a $6$--dimensional homogeneous strict nearly K\"ahler
           manifold of scalar curvature $\scal$ we have
           \begin{equation}\label{eq:LJR_NK}
            \K_5
            \;\;=\;\;
            -\,\frac{\scal}{24}\,\|\,\cdot\,\|^2\,\K_3\;-\;\frac{\scal^2}{3600}\,\|\,\cdot\,\|^4\,\K_1
           \end{equation}
  \end{enumerate}
  Moreover, Equations \eqref{eq:LJR_BS}--\eqref{eq:LJR_NK} define minimal linear Jacobi relations
  unless the sectional curvature is constant.
 \end{theorem}

\subsection{Some explanations and ideas for the proofs of the theorems}
\label{se:ingredients}
 The Berger metrics are most easily explained for $\kappa\,>\,0$. 
 At the heart of their construction is the fact that the Hopf  circles are no longer 
 isometrically embedded  in the ambient flat complex  space $\C^{n+1}$ 
 anymore unless $\hat \MM^n(\,\kappa,\,r\,)$ is the round sphere with $r^2\,=\,\frac{4}{\kappa}$.
 Hence, the radius $r$ is only properly defined via the relation  
 $L\,=\,2\,\pi\,r$ where $L$ denotes the intrinsic length of these circles. 
 In contrast, the great circles perpendicular to the Hopf foliation of a  Berger 
 sphere are always isometrically embedded 
 in the complex euclidean space $\C^{n+1}$, and their radius $r_\kappa$ satisfies
 $r_\kappa^2\,=\,\frac{4}{\kappa}$. Consequently, the constant $c$ from 
 \eqref{eq:LJR_BS} complies with the equation $c^2\,=\,\frac{r^2}{r_\kappa^2}\kappa$. 
 Here, $\frac{r^2}{r_\kappa^2}$ is the factor by which we have 
 rescaled the metric of constant sectional curvature in the direction of the Hopf foliation. 
 It is well known that the Berger metrics are naturally reductive, cf. \cite{AFF,Zil1} 
 and the references given there. In Section \ref{se:NR_FB} we recall the construction of the naturally 
 structures on $\hat \MM^n(\,\kappa,\,r\,)$; in particular, we will show how precisely
 the before mentioned constant $c$ enters into the torsion form and the curvature tensor 
 of the naturally reductive structure of $\hat \MM^n(\,\kappa,\,r\,)$, see Proposition \ref{p:BS}. 
 In \cite[Proposition 1]{Kap} it is shown that the metric of $\hat \MM^n(\,0,\,\frac{1}{c^2}\,)$ 
 is naturally reductive. An alternative proof, which is more suitable for our purpose, 
 can be found in \cite[Chpt. 9]{TV}. Here, the linear Jacobi relation \eqref{eq:LJR_BS} 
 is automatically minimal. In \cite{Gor} it was shown that, in fact, every left--invariant 
 metric on a Heisenberg group is naturally reductive, cf. p.~479 of that reference and also \cite{AFF,St1} 
 (However, we do not expect more linear Jacobi relations.)
 
 \bigskip\noindent
  Naturally reductive $3$--symmetric spaces with their canonical nearly K\"ahler structures
  are examples of homogeneous strict nearly K\"ahler manifolds. 
  According to \cite[Theorème 1.1]{But1}, there are no other examples. Also 
  it follows from \cite[Theorème 1.2]{But1} that every $6$--dimensional 
  simply connected naturally reductive $3$--symmetric space is a 
  standard normal homogeneous space from the following list:
  \begin{itemize}
  \item The round sphere $\S^6\,=\,\G_2/\SU(\,3\,)$ realized
           as the purely imaginary octonions of unit length
           with the nearly K\"ahler structure coming from the octonionic
           multiplication,
  \item The complex flag manifold $\bbF^3\,=\,\SU(\,3\,)/\U(\,1\,)\times \U(\,1\,)$ seen as
           the twistor space of $\CP^2$,
  \item The complex projective space $\CP^3\,=\,\SO(\,5\,)/\U(\,2\,)$ seen as the twistor
           space of $\S^4$,
  \item The nearly K\"ahler structure on $\S^3\,\times\,\S^3\,=\,\SU(\,2\,)\times
           \SU(\,2\,)\times \SU(\,2\,)/\SU(\,2\,)$ constructed by Ledger/Obata.
 \end{itemize}
 
 \noindent 
 For the last three mentioned spaces, \eqref{eq:LJR_NK} holds as a minimal linear Jacobi relation. 
 Inhomogeneous strict nearly K\"ahler structures on $\S^6$ and $\S^3\,\times\,\S^3$
 were recently discovered, cf. \cite{FH}. (However, we don't think that here Jacobi
 relations exist.) 

 \bigskip\noindent
 Simply connected  standard normal homogeneous spaces of positive sectional
 curvature in dimension seven are given as follows, cf.~\cite{Ber,Wil,WZ,Zil2}: 
 \begin{itemize}
  \item the round sphere $\S^7$ equally
          written as $\,\SO(\,8\,)\,/\,\SO(\,7\,)$ or $\Spin(\,7\,)\,/\,\G_2$,
  \item the strict Berger sphere $\SU(\,4\,)\,/\,\SU(\,3\,)$ and its
           quaternionic analogue $\Sp(\,2\,)\,/\,\Sp(\,1\,)$,
  \item the so--called squashed sphere $\S^7_{\mathrm{squashed}}\, = \,
           \Sp(\,2\,)\times\Sp(\,1\,)\,/\,\Sp(\,1\,)\,\times\,\Sp(\,1\,)$,
  \item Berger's manifold $V_1\,=\,\SO(\,5\,)\,/\,\SO(\,3\,)$,
  \item the Wilking space $V_3\,=\,\SU(\,3\,)\,\times\,\SO(\,3\,)\,/\,\U(\,2\,)$.
 \end{itemize}

 \noindent
 The last three mentioned  Riemannian manifolds $ \S^7_{\mathrm{squashed}}$,
 $V_1$ and $V_3$ with their standard normal metrics are homogeneous proper nearly 
 parallel $\G_2$--spaces, cf. \cite{AlS,BG1,FKMS}. Here, \eqref{eq:LJR_NP} holds as 
 a minimal linear Jacobi relation.  The normal metric of the $7$--sphere 
 $\Spin(\,7\,)/\G_2\,$ is  the round metric. Furthermore, this is a homogeneous nearly 
 parallel $\G_2$--space which can be seen as the analogue of $\,\G_2/\SU(\,3\,)$.
 The normal metric of the $7$--sphere $\SU(\,4\,)/\SU(\,3\,)$ is a strict Berger metric (not the round metric). 
 There is the linear Jacobi relation \eqref{eq:LJR_BS}, which is, however, different from
 \eqref{eq:LJR_NP}. Its quaternionic analogue $\Sp(\,2\,)/\Sp(\,1\,)$
 is the only example of a simply connected  standard normal homogeneous space 
 of positive sectional curvature lacking of a linear Jacobi relation.

 \bigskip
 \noindent
 The proof of Theorem \ref{th:JR}, given in Section \ref{se:Th:JR}, shows that, in addition to Remark \ref{re:JR},
 the occurrence of Jacobi relations on naturally reductive  spaces is also a direct 
 consequence of the Theorem of Cayley--Hamilton. In terms of linear Jacobi
 relations, however, there is a subtle difference between the naturally reductive structures
 of $\hat \MM^n(\,\kappa,\,r\,)$ and $\hat\MM^n(\,0,\,\frac{1}{c^2}\,)$ with $n\geq 2$, 
 and those of the other spaces mentioned in Theorem \ref{th:LJR}.  For the $6$--dimensional homogeneous strict  nearly K\"ahler manifolds 
 $\S^6$, $\bbF^3$, $\CP^3$ and $\S^3\,\times\,\S^3$, 
 and the homogeneous nearly parallel $\G_2$--spaces $\S^7$, 
 $\S^7_{\mathrm{squashed}}$, $V_1$ and $V_3$, the torsion form  $\tau$ 
 of the standard normal homogeneous structure is pointwise a generalized vector cross product, 
 which means, by definition, that $\tau$ is a nowhere vanishing $3$--form 
 and the eigenvalues of the  skew--symmetric endomorphism 
 $\tau_X\,:\, T_pM\, \longmapsto \,T_pM,\; Y\longmapsto \tau(\,X,\,Y\,)$ do 
 not depend on $X$ as long as the latter vector is chosen from the unit sphere $\S(\,T_pM\,)$,
 see Section \ref{se:GVCP}. In this situation one can show that the Jacobi relation guarantied 
 by Theorem \ref{th:JR} is automatically linear. Differently said, the existence of linear 
 Jacobi relations on the four plus four make eight before mentioned $6$-- and  $7$--dimensional  naturally
 reductive spaces is a natural consequence of a distinguished algebraic  property of
 their torsion forms. Moreover, these are the only non--trivial examples of simply connected naturally 
 reductive spaces whose torsion forms have this property, see Theorem \ref{th:GVCP}. This fact is certainly 
 interesting in itself, but can also be seen  as a first step towards the classification of naturally reductive spaces 
 with linear Jacobi relations.
 
 \bigskip
 \noindent
 The proof of Theorem \ref{th:GVCP} is based on the classification of
 generalized vector cross products given in \cite{BMS}. They exist only in dimensions 
 three, six and seven and are unique in each dimension. Therefore, in  Sections \ref{se:NR_NK} and \ref{se:NR_NP}
 we consider the canonical connections $\nabla^c$ associated with $6$--dimensional strict nearly 
 K\"ahler manifolds and ($7$--dimensional) nearly parallel $\G_2$--spaces. 
 In both cases $\nabla^c$ is an affine metric connection whose torsion form (called intrinsic torsion)
 is a generalized vector cross product. There remains the question when 
 $\nabla^c$ is a naturally reductive structure. Given the results of \cite{But1},\cite{But2} 
 mentioned before, the answer in dimension six is straightforward. In the simply connected case it gives the 
 four naturally reductive $3$--symmetric spaces mentioned before. On the other hand, the classification of simply connected  
 homogeneous nearly parallel $\G_2$--spaces obtained in \cite{FKMS} is much richer, 
 making the answer in dimension seven more difficult. However, by Theorem \ref{th:OR} we can a priori 
 exclude all those spaces which are not proper (except for the round sphere $\S^7$) essentially by the following argument: 
 If a simply connected homogeneous nearly parallel $\G_2$--space fails to be proper, then there
 exists more than one nearly parallel $\G_2$--structure. 
 Now different nearly parallel $\G_2$--structures have different canonical connections;
 if one of them were naturally reductive, the others would be too, which contradicts the 
 uniqueness assertion of Theorem \ref{th:OR}. In addition, we will determine which of the 
 metric tensors underlying simply connected homogeneous nearly parallel $\G_2$--spaces 
 are naturally reductive, cf. Theorem \ref{th:CLASS_G2}.
 It turns out that only the homogeneous nearly parallel $\G_2$--metrics of Aloff--Wallach spaces 
 $\NM(\,k,\,l\,)$ different from $\NM(\,1,\,1\,)$ fail to have this nice property. 
 We conclude that exclusively the four standard normal homogeneous spaces $\Spin(\,7\,)/\G_2\,$, $ \S^7_{\mathrm{squashed}}$,
 $V_1$ and $V_3$ are equipped with nearly parallel $\G_2$--structures which return the naturally reductive structures.

 \bigskip
 \noindent
 In Section \ref{se:SP} we will draw the same conclusion, however,
 this time starting from the list of $7$--dimensional simply connected naturally reductive spaces given in \cite{St3}. 
 Here, we can rule out all total spaces of naturally reductive $\S^1$--fiber bundles over $6$--dimensional  
 compact  base spaces. Moreover, we will show that the naturally reductive quotients $\Sp(\,2\,)\,\times\,\Sp(\,1\,)\,/\,\Sp(\,1\,)\,\times\,\Sp(\,1\,)$ and 
 $\SU(\,3\,)\,\times\,\SO(\,3\,)\,/\,\U(\,2\,)$ do not contain more examples than those already known.
 Thus, we stay again with the  previously mentioned four standard normal homogeneous spaces.
 Principally, it is also possible to decide by direct calculations whether or not the torsion form of a given naturally reductive space
 is a generalized vector cross product. In Appendix \ref{se:N11} we will give these 
 calculations for the Wilking quotient and show that, in fact, only for the standard 
 normal space $V_3$ the torsion form is a generalized vector cross product.

 \bigskip
 \noindent
 Quite differently, the torsion forms of $\hat \MM^n(\,\kappa,\,r\,)$ and $\hat \MM^n(\,0,\,\frac{1}{c^2}\,)$
 are  clearly not generalized vector cross products unless $n \,=\, 1$. Here, the Jacobi relation
 from Theorem \ref{th:JR} is not linear, which makes the existence of a linear Jacobi relation less obvious;
 it is only explained by the fact  that certain components of $\K_0(\,X\,)$ vanish in a 
 decomposition of $\End^{\mathrm{sym}}T_pM$ with respect to $\tau_X\star$, 
 see Proposition \ref{p:EXIST_LJR}. 
 Regarding the minimality of the Jacobi relations from Theorem \ref{th:LJR},
 for \eqref{eq:LJR_BS} and \eqref{eq:LJR_NP} this follows immediately from Proposition \ref{p:LJR_STRUK}.
 In order to see the minimality of \eqref{eq:LJR_NK}, however,
 we need several characterizations of the nearly K\"ahler round sphere $\S^6$, see Proposition \ref{p:CHSC}.
 These loosely gathered arguments are implemented in the proof of Theorem \ref{th:LJR}, which can be found at 
 the end of Section \ref{se:th:LJR}. In Section~\ref{se:EX}  we will compare our results with the sporadic examples  of  linear 
 Jacobi relations observed before, cf.~\cite{Ar,AAN,Gon,GN,MNT,NT} including a discussion of the
 $3$--dimensional  case.
 
 \bigskip
 \noindent
 We know of no other examples of non--symmetric \textsc{g.o.}--spaces  
 with linear Jacobi  relations which were not covered by Theorem \ref{th:LJR}. 
 In \cite{AN} it is claimed that the same Jacobi
 relation \eqref{eq:LJR_NK} also holds on the generalized Heisenberg group
 $\NM^6$ with its canonical metric of type H, cf. \cite{Kap}. 
 However, one easily sees that the calculations from \cite{AN} contain an error. 
 In a forthcoming paper we will show that 
 there does not exist a linear Jacobi relation on  $\NM^6$, cf. \cite{JW2}.
\section{Proof of Theorem~\ref{th:JR}}
\label{se:Th:JR}
 We will first show how to reduce the occurrence of Jacobi relations on naturally reductive spaces 
 to a question of linear algebra. For each $X\,\in\, T_pM$, consider the linear operator 
 $\scrT(\,X\,)\,=\,-\frac12\,\tau_X\star$ on $\End^{\mathrm{sym}}\,T_pM$ see 
 \eqref{eq:def_of_derivation}. Since $\tau_X(\,X\,)\,=\,0$, it induces the partial algebraic derivation
 \begin{equation}\label{eq:Recursion_Operator}
  \scrT:\;\;\Jac^\bullet TM\;\longrightarrow\;\Jac^{\bullet+1}TM,\qquad
  \K\;\longmapsto\;\Big(\;X\;\longmapsto\;\scrT\;\K\,(\,X\,)\;:=\;\scrT(\,X\,)\,\K(\,X\,)\;\Big)
 \end{equation}
 In the following we denote by $\K_k\,$ the symmetrized $k$--th covariant derivative of the 
 Riemannian curvature tensor defined in Equation \eqref{eq:SCD}.  The following lemma shows that 
 naturally reductive spaces are $\frakC_0$--spaces, cf. \cite[Chpt. 2.9]{BTV}:

 \begin{lemma}
 \label{le:Lemma4}
  We have
  \begin{equation}\label{eq:recursion2}
   \K_k
   \;\;=\;\;
   \scrT^k \K_0
  \end{equation}
 \end{lemma}
 
 \begin{proof}
 The subspace of $G\,:=\,\Tr(\,\bar\nabla\,)$--invariant sections of $\Jac^k TM$ will be denoted by
 $\Gamma^G(\,\Jac^kTM\,)$. We have $\K_0\,\in\,\Gamma^G(\,\Jac^0TM\,)$.
 By induction, it suffices to show that $\scrT \K\,\in\,\Gamma^G(\,
 \Jac^{k+1}TM\,)$ and $\scrT(\,X\,)\, \K(\,X\,)\,=\,\nabla_X\, \K\,(\,X\,)$
 for each $\K\,\in\,\Gamma^G(\,\Jac^kTM\,)$: By the very definition of the transvection group,
 every $G$--invariant section of a tensor bundle over $M$ is $\bar\nabla$--parallel. From the formula
 $\bar\nabla\,=\,\nabla\,+\,\frac{1}{2}\tau$, we thus see that $\nabla_X$ acts 
 on $\K\,\in\,\Gamma^G(\,\Jac^kTM\,)$ by the total algebraic derivation
 $-\frac12\,(\,\tau_X\star \K\,)(\,Y\,)\,:=\,\scrT(\,X\,)\,\K(\,Y\,)\, + 
 \,\frac{\,k\,+\,2\,}{2}\,\K(\,\tau_XY\,)$. Since $\tau$ is alternating, 
 we conclude that  $\nabla_X\, \K\, (\,X\,)=\,\scrT(\,X\,)\,\K(\,X\,)$  for each 
 $\K\,\in\,\Gamma^G(\,\Jac^kM\,)$. Furthermore, because $\nabla$ and $\bar\nabla$ 
 both are $G$--invariant, so is their difference $\frac\tau2\,=\,\bar\nabla\,-\,\nabla$. 
 Hence, $\scrT$ is $G$--invariant, too. We conclude that
 \begin{equation}\label{eq:recursion1}
  \scrT\,\Big(\;\Gamma^G(\;\Jac^kTM\;)\;\Big)
  \;\;\subset\;\;
  \Gamma^G(\;\Jac^{k+1}TM\;)
 \end{equation}
 Equation \eqref{eq:recursion2} follows by induction.
 \end{proof}
 
 \bigskip
 \noindent
 For the rest of this section, let $\K\,:=\,\K_0\,:\,X\,\longmapsto\,\K(\,X\,)\,:=\,R(\;\cdot\;,\,X,\,X\,)$ denote the 
 symmetrized Riemannian curvature tensor.

 \begin{corollary}
 \label{co:JR}
 Let $(\,M,\,g,\,\bar\nabla\,)$ be a naturally reductive space with torsion 
 form $\tau$ and associated linear operator   $\scrT$ on $\Jac^{\mathrm{\bullet}}\,TM\,$ 
 defined by \eqref{eq:Recursion_Operator}.
 The Jacobi relation~\eqref{eq:JR} holds if and only if 
 the polynomial $P(\,\lambda\,)\,:=\,\lambda^{d+1}\,+\, \sum_{\ell=1}^{d+1}a_{\ell}\,\lambda^{d+1-\ell}$ 
 satisfies $P(\,\scrT(\,X\,)\,)\K\,(\,X\,)\,Y\,=\,0$  for all $X,\,Y\in T_pM$.
 This Jacobi relation is linear if and only if 
 each $a_\ell$ is constant on the unit sphere bundle $\S(\,T M\,)$.
 \end{corollary}
 \begin{proof}
  The first assertion follows directly from \eqref{eq:Recursion_Operator} and \eqref{eq:recursion2}. 
  Also $a_\ell$ is constant on $\S(\,T M\,)$
  if and only if $a_\ell\,=\,0$, or $\ell$ is even and there exists a real number $c_\ell$ such that 
  $a_\ell\,=\,c_\ell\,\|\,\cdot\,\|^{\frac{\ell}{2}}$. The result follows.
 \end{proof}

 \begin{lemma}
 \label{le:Killing_tensors}
  Let $(\,M,\,g,\,\bar\nabla\,)$ be a naturally reductive space.
  Then, every $\Tr(\,\bar\nabla\,)$--invariant symmetric tensor
  field is Killing.
 \end{lemma}
 \begin{proof}
  The parallel transport along an arbitrary geodesic can be realized
  by a $1$--parameter subgroup of $\Tr(\,\bar\nabla\,)$. Thus, we can use
  the same argument as for \textsc{g.o.}--spaces mentioned after Theorem \ref{th:JR}.
  Alternatively one can copy the proof of \cite[Lemma 2.3]{OR2}.
 \end{proof}

 \begin{corollary}
 \label{co:Killing_tensors} 
  Let $(\,M^n,\,g,\,\bar\nabla\,)$ be a naturally reductive space with torsion
  form $\tau$. Following the notation of Theorem \ref{th:JR}, let 
  $P_{\tau}(\,\lambda\,;\,X\,)\,:=\,\det\big(\lambda\,\id\,-\,\scrT(\,X\,)
  |_{\End^{\mathrm{sym}}\,\{X\}^\perp} \big)$ and $\vartheta \,:=\,2\, \lceil \frac{d}{2}\rceil$
  with $d\,+\,1\,:=\,{n \choose 2}$. There exist $\Tr(\,\bar\nabla\,)$--invariant symmetric  Killing tensors $a_\ell$ 
  of degree $\ell$ for $\ell\,=\,2,\, 4,\,\ldots,\,\vartheta$ such that  $P_\tau(\,\lambda;\,X\,)\,=\,
  \lambda^{d\,+\,1}\,+\, \sum_{k\,=\,1}^{\frac{\vartheta}{2}} a_{2\,k}(\,X\,)\,\lambda^{d\,+\,1\,-\,2\,k}$
  for each $X\,\neq\,0$.
 \end{corollary}
 
 \begin{proof}
  The  torsion form $\tau$ is $g$--invariant $\tau_{\d_pg\,X}\,=\,\d_pg\,\circ\, 
  \tau_X\,\circ\, (\d_pg)^{-1}$ for every $g\,\in\,G\,:=\,\Tr(\,\bar\nabla\,)$. 
  Because the characteristic polynomials of conjugated endomorphisms are the same, 
  we see that the coefficients $\tilde b_i\,=\,\tilde b_i(X)$ of the characteristic polynomial of 
  $\tau_X\,\in\,\End^{\mathrm{skew}}T_pM$ are $G$--invariant and hence
  Killing tensors according to Lemma~\ref{le:Killing_tensors}. The coefficients $b_j(\,X\,)$ of the 
  characteristic polynomial of $\tau_X\,\in\,\End^{\mathrm{skew}}{\{X\}^\perp}$ 
  are given by $b_j(\,X\,)\,=\,\tilde b_{j+1}(\,X\,)$ for $j\,=\,0,\,\ldots,\,n\,-\,1\,$ and $X\,\neq\,0$. 
  Therefore, $b_j$ is a $G$--invariant Killing tensor, too. Due to the fact that every symmetric polynomial
  is a uniquely determined polynomial in the elementary symmetric polynomials, we know 
  that the coefficients $\tilde a_\ell\,=\,\tilde a_\ell(\,X\,)$ of  $P_{\tau}(\,\lambda\,;\,X\,)$ 
  are polynomials in the $b_j$'s. In particular, $\tilde a_\ell$ is a $G$--invariant Killing tensor, too.  
  Furthermore, $\tilde a_\ell$ vanishes if $d\,+\,1\,-\,\ell$ is odd
  since $\scrT(\,X\,)|_{\End^{\mathrm{sym}}\,\{X\}^\perp}$ is skew--symmetric 
  with respect to the canonical scalar product on $\End^{\mathrm{sym}}\, \{X\}^\perp$.
  The assertion of the corollary follows by setting $a_\ell\,:=\,\tilde a_{d\,+\,1\,-\,\ell}$.
 \end{proof}

 \paragraph{Proof of Theorem~\ref{th:JR}}
 By the Theorem of Cayley--Hamilton, we have
 $$
  P_{\tau}(\scrT(\,X\,);\,X\,)|_{\End^{\mathrm{sym}}\,\{X\}^\perp}\;=\;P_{\tau}
 (\scrT(\,X\,)|_{\End^{\mathrm{sym}}\,\{X\}^\perp};\,X\,)\;=\;0
 $$ 
 for all $0\,\neq\,X\,\in\,T_pM$. Evaluating this identity on $\K(\,X\,)\,\in\,\End^{\mathrm{sym}}\{X\}^\perp$ we obtain that 
 $P_{\tau}(\scrT(\,X\,);\,X\,)\,\K(\,X\,)\,Y\,=\,0$ for all $X,\,Y\in T_pM$. 
 The desired Jacobi relation follows in accordance with Corollaries~\ref{co:JR} and \ref{co:Killing_tensors}.\qed

\section{Fiber Bundles of Naturally Reductive Spaces} 
\label{se:NR_FB}
 Let $(\,G,\,K,\,B\,)$ be  an almost effective naturally reductive triple where $K$ contains 
 a proper closed normal subgroup $H\subset K$ whose Lie algebra $\frakh$ is a 
 non--degenerate subspace of $(\,\frakg,\,B\,)$. 
 We will construct a $1$--dimensional family of naturally reductive triples 
 $(\,\hat G,\,\hat K,\,\hat B_s\,)$ such that $\hat G/\hat K$ is the total space 
 of a fiber bundle over $G/K$ and the naturally reductive metric on $\hat G/\hat K$ 
 associated with $\hat B_s$ is bundle--like. 
 From a slightly different point of view, the same construction is also 
 given in  \cite{Zil1}, see also \cite[Chpt. 9\,H]{Bes}. Maybe the most prominent example is the family of 
 non--homothetic  naturally reductive  metrics on the Wilking manifold 
 $\SU(\,3\,)\,\times\,\SO(\,3\,)/\U(\,2\,)$ providing examples of 
 normal homogeneous  spaces of positive sectional curvature. In a similar way one obtains 
 the naturally reductive structures on the total spaces $\hat \MM^n(\,\kappa,\,r\,)$ 
 of the Hopf fibrations over some complex space form $\MM^n(\,\kappa\,)$ 
 with $\kappa\neq 0$, i.e., Berger spheres and their non--compact counterparts
 which fiber over the complex hyperbolic space. 
 Strictly speaking the  Heisenberg group $\hat \MM^n(\,0,\,\frac{1}{c^2}\,)$ does not follow this setup 
 but can be seen  as the limit case $\kappa\to 0$, $r\,\to\,\infty$ such that $c\,=\,\frac12\,r\kappa\,=\,\mathrm{const}$.
 Moreover, we establish the 
 existence of  naturally reductive structures on all simply connected homogeneous 
 nearly  parallel  $\G_2$--spaces of types 2 and 3, see Section~\ref{se:CLASS_G2}.

 \bigskip\noindent
 In our situation, $G/H$ is a homogeneous space and $K/H$ is a Lie group, too. 
 Since $\frakh$ is a non--degenerate subspace of $\frakk$, the Lie algebra of $K/H$ can be identified 
 with $\frakh^\perp$. Then, both $\frakh$ and $\frakh^\perp$ are ideals 
 of $\frakk$. Also there exists a natural  $G$--equivariant fibration 
 $f\,\colon\,G/H\,\longrightarrow\,G/K$.
 Furthermore, set $\hat G\,:=\,G\times K/H$ and consider $K$ as a
 subgroup $\hat K\subset\hat G$ by $\hat K\,:=\,\{\,(\,k,\,k\cdot H\,)
 \,\mid\,k\,\in\,K\,\}$. Thus, also $\hat G$ acts almost effectively on $G/H$
 via $(\,g',\,k\cdot H\,)\,g\cdot H\,:=\,g'g\,k^{-1}\cdot H$,
 $f$ is equivariant with respect to the actions of $\hat G$ and
 $G$, and the inclusion map $G\subset \hat G$ induces an equivariant isomorphism 
 $G/H\cong\,\hat G/\hat K$ of differentiable manifolds. 
 Clearly, $\hat\frakg\,=\,\frakg\,\oplus\,\frakh^\perp$ is the Lie algebra of $\hat G$.

 \begin{lemma} (\cite[Theorem~3]{Zil1})\label{le:Nor_SG}
  In the above situation consider the family of non--degenerate symmetric invariant bilinear forms on $\hat\frakg$ 
  given by $\hat B_s\,:=\,B\,\oplus\,\frac1s\, B|_{\frakh^\perp\times
  \frakh^\perp}$ ($s\,\neq\,0$).
  \begin{enumerate}
   \item If $B|_{\frakh^\perp\times \frakh^\perp}$ is positive definite,
           then $(\,\hat G,\,\hat K,\,\hat B_s\,)$ is a naturally reductive triple
           for $s\,>\,-1\,$. Also the naturally reductive metric induced by $\hat B_s$ converges for $s\,\to\,0$ to the 
           naturally reductive metric induced by $\hat B_0\,:=\,B$ on $G\,/\,H\,\cong\,\hat G\,/\,\hat K$. 
   \item If $B|_{\frakh^\perp\times \frakh^\perp}$ is negative definite,
           then $(\,\hat G,\,\hat K,\,\hat B_s\,)$ is a naturally reductive triple for $s\,<\,-1\,$.
  \end{enumerate}
  The so constructed metrics on $G\,/\,H$ are $\hat G$--invariant and bundle--like,
  i.e., $f$ is a Riemannian submersion. Conversely, if the adjoint 
  representation of $K/H$ is irreducible and does not appear 
  as a component of the isotropy representation of $G/K$, then every
 $\hat G$--invariant  bundle--like metric on $G\,/\,H$ is obtained 
 in this way by a unique parameter $s$.
 \end{lemma}

 \begin{proof}
  Since $B$ is $\Ad(\,G\,)$--invariant, it is clear that  $\hat B\,:=\,\hat B_s$ 
  is  $\Ad(\,\hat G\,)$--invariant. It remains to be determined
  when $\hat B$ is positive definite on $\hat \frakm\times \hat \frakm$, 
  where $\hat \frakm$ denotes the orthogonal complement
  of $\hat\frakk$ in $\hat\frakg$ with respect to $\hat B$.
  By assumption, there exist $B$--orthogonal splittings
  $\frakg\,=\,\frakk\,\oplus\,\frakp\,=\,\frakh\,\oplus\,\frakm$ such that
  $B|_{\frakp\times \frakp}$ is positive definite. Also the orthogonal
  complement of $\frakh$ in $\frakk$ is subjected to an $H$--invariant
  $B$--orthogonal decomposition $\frakm\,=\,\frakp\, \oplus\, \frakh^\perp$.
  Furthermore, it is clear that
  $$
   \hat\frakk
   \;\;=\;\;
   \{(\,X,\,X_{\frakh^\perp}\,)\;\mid\;X\in \frakk\,\}
  $$
  where $X_{\frakh^\perp}$ denotes the orthogonal projection of $X$ onto the 
  orthogonal complement of $\frakh$ in $\frakk$.
  Then, we have $\hat B$--orthogonal splittings 
  $\hat\frakg\,=\,\hat\frakk\,\oplus\,\hat\frakm$ and
  $\hat \frakm\,=\, \frakp\oplus\hat\frakh^\perp\,$ depending on $s$, where 
  $$
   \hat\frakh^\perp
   \;\;:=\;\;
   \{(\,Z,\,-\,s\,Z\,)\;\mid\;Z\in \frakh^\perp\,\}
  $$
  is the orthogonal complement of $\frakp$ in $\hat \frakm$ with respect
  to $\hat B$. Thus, we can canonically  identify $ \hat\frakh^\perp$ 
  with $\frakh^\perp$. Then, the restriction of $\hat B$ to $\frakh^\perp\times \frakh^\perp$  
  is given by $(\,1\,+\,s\,)B|_{\frakh^\perp\times
  \frakh^\perp}$. So that the latter is positive definite,
  we need either that $B(\,Z,\,Z\,)\,>\,0$ and $s\,>\,-1$, or $B(\,Z,\,Z\,)\,<\,0$ 
  and $s\,<\, -1$ for all $Z\in \frakh^\perp$. Also
  $$
   (\,1\,+\,s\,)B|_{\frakh^\perp\times \frakh^\perp}
   \;\;\underset{s \to 0}\,\longrightarrow\;\;
   B|_{\frakh^\perp\times \frakh^\perp}
  $$
  If $B|_{\frakh^\perp\times \frakh^\perp}$ is positive definite, then the previous implies that 
  the induced metric tensor converges for $s\,\to\,0$  to the naturally reductive metric on $G/H$ defined 
  by $B$. 

  \bigskip
  \noindent 
  For the last assertion, it is clear that  $f$ is a Riemannian submersion with respect to the naturally reductive metrics 
  induced by $\hat B$ and $B$. Conversely, given a $\hat G$ invariant bundle--like metric on $G\,/\,H$,
  we consider the orthogonal decomposition  $T G\,/\,H\,=\,\scrH\oplus\scrV$ of the tangent bundle 
  into the horizontal and the vertical subbundle. The vertical bundle is $\hat G$--invariant and so is its 
  orthogonal complement $\scrH$. Then, the isotropy representation of $\hat K$ on 
  horizontal vectors is the isotropy representation of $K$, whereas it is the adjoint representation of 
  $K/H$ on vertical vectors. Since we assume that the adjoint representation of $K/H$ is irreducible 
  and not a component of the isotropy representation of $G/K$, we have $\scrH\bot\scrV$ also for every other 
  $\hat G$--homogeneous metric by Schur's Lemma. Therefore, a bundle--like $\hat G$--invariant metric 
  is uniquely determined on some vertical space $\scrV_p$. This metric 
  has to be equal to $\hat B_s$ for some $s$ again by Schur's Lemma because $\scrV$ 
  is an irreducible  $K/H$--module.
 \end{proof}

\begin{remark}\label{re:scaling}
 Although $\hat B|_{\hat\frakh^\perp\times \hat\frakh^\perp}$ 
 becomes $|\,1+s\,|B|_{\frakh^\perp\times \frakh^\perp}$ under 
 the natural identification $\hat\frakh^\perp\,\cong\,\frakh^\perp$ (see the proof of Lemma \ref{le:Nor_SG}), the metric tensor on $\scrV$ gets 
 effectively scaled by the factor $\frac1{|\,1+s\,|}$ via the correct identification $\scrV\,\cong\,\hat\frakh^\perp$;
 this is due to the fact that the vertical vector fields induced by $X$ and $X\oplus -\,s\,X$ via the infinitesimal left action of
 $\frakg$ and $\hat \frakg$, respectively, on $G/H\,=\,\hat G/\hat K$ are related by $(1\,+\,s)X^*\,=\,(\,X\oplus -\,s\,X\,)^*$ 
 for each $X\in \frakh^\perp$. As a result, in the positive definite case values $s\,>\,0$  correspond to a so called squashing of $G/H$ along the fibers 
 of $f$, whereas for $s < 0$ these fibers get stretched relative to the naturally reductive metric induced by $B$ on $G/H$.
\end{remark}

\begin{example}\label{ex:Wi}
We let $G\,:=\,\SU(\,3\,)$, $K\,:=\,\U(\,2\,)$ and $H\,:=\,\{\diag(\,e^{i\varphi}\,,\,e^{i\varphi}\,)\,\mid\,\varphi\in\R\,\}$.
         Then, $H$ is the center of $K$ and hence a normal subgroup with $K/H\,=\,\SO(\,3\,)$. Let $B$ be the negative of the Killing form
          of $\su(\,3\,)$. We obtain a family of non--homothetic $\SU(\,3\,)\,\times\,\SO(\,3\,)$--invariant naturally reductive metrics on the Wilking manifold 
         $\NM(\,1,\,1\,)$ parameterized by $s>-1$. These metrics are bundle--like over $G/K\,=\,\CP^2$, i.e., the projection map is a Riemannian submersion.
         For $s>0$ we obtain normal metrics of positive sectional curvature, cf. \cite{AW,Wil}.
\end{example}

 \noindent
 Let us now consider specifically the case $\dim\,K/H\,=\,1$. Then, $\frakh^\perp$
 is a $1$--dimensional ideal of $\frakk$ and hence contained in the center
 of $\frakk$. Let $Z_0$ be a unit vector $B(\,Z_0,\,Z_0\,)\,=\,\pm\,1$ of  $\frakh^\perp$.
 Set $\hat B_s\,:=\,B\,\oplus\,\frac1s\, B|_{\frakh^\perp\times \frakh^\perp}$ 
 (where $s$ has to be chosen as in Lemma \ref{le:Nor_SG}).

 \begin{lemma}\label{le:T_and_R}
  There exists a vertical unit vector $V_0$
  such that the torsion form and the curvature tensor of the naturally reductive triples $(\,\hat G,\,\hat K,\,\hat B_s\,)$ ($s\,\neq\,0$) and $(\,G,\,H,\,B\,)$ ($s\,=\,0$) are:
  \begin{equation}\label{eq:HAT_TAU}
   \hat\tau
   \;\;=\;\;
   \tau\;+\;\frac1{\sqrt{|\,1\,+\,s\,|}}\,\rho_*(Z_0)\,\wedge\,V_0^\flat
  \end{equation}
  and:
  \begin{equation}\label{eq:HAT_R}
   \hat R
   \;\;=\;\;
   \bar R\;+\;\frac1{|\,1\,+\,s\,|}\,\rho_*(Z_0)\,\otimes\,\rho_*(Z_0)
  \end{equation}
   Here, $\tau$ and $\bar R$ are the torsion form and the curvature tensor, respectively,
   of the canonical connection $\bar\nabla$ associated with $(\,G,\,K,\,B\,)$, and $\rho_*$ 
   is the linearized isotropy representation of $\frakk$. Also $V_0^\flat$ denotes the metric dual $1$--form. 
 \end{lemma}
 \begin{proof}
 We have $f_*\,\hat \tau(\,X,\,Y\,)
 \,=\,\tau(\,f_*X,\,f_*Y\,)$ for all $X,\,Y\in \frakp$. Also
 $$
  \hat\tau(\,Z,\,Y\,)
  \;\;=\;\;
  -\;[\,Z,\,Y\,]_{\hat\frakm}\;\;=\;\;  -\;[\,Z,\,Y\,]_{\frakm}\;\oplus\;0\;=\;-\;\rho_*(\,Z\,)\,Y\;\oplus\;0
 $$
 for all $Z\,\in\,\hat\frakh^\perp$ and $Y\in \frakp$. Since $\frakh^\perp$ is a
 $1$--dimensional space, we see that $\hat\tau\,=\,\tau\,+\,\frac1{\sqrt{ |\,1\,+\,s\,|}}
 \rho_*(\,Z_0\,)\,\wedge\,V_0^\flat$ where $Z_0$ is a unit vector of
 $\frakh^\perp$, i.e., $B(\,Z_0,\,Z_0\,)\,=\,\pm\,1$, and $V_0\,:=\,
 \mp\, \frac1{\sqrt{|1\,+\,s|}}(\,Z_0,\,-\,s\,Z_0\,)$. Then, $\hat B(\,V_0,\,V_0\,)\,=\,\pm\,1$,
 i.e., $V_0$ is a unit vector. In case that $B(\,Z_0,\,Z_0\,)\,=\,1$, the calculation
 is equally valid also for $s\,=\, 0$, i.e., for the naturally reductive structure associated with $(\,G,\,H,\,B\,)$.
 This proves \eqref{eq:HAT_TAU}.
 For the curvature tensor of $\hat\nabla$, let $X,\,Y\,\in\,\frakp$.
 In order to determine $ [\,X,\,Y\,]_{\hat\frakk}$, we decompose 
 $\,[\,X,\,Y\,]_\frakk$ into the $\frakh^\perp$--part
 $[\,X,\,Y\,]_{\frakh^\perp}$ and its complement
 $[\,X,\,Y\,]_\frakh$. The decomposition of $\,[\,X,\,Y\,]_{\frakh^\perp}$ 
 into $\hat \frakk$-- and $\hat\frakm$--part is given by
 \begin{eqnarray*}
  \big(\;[\,X,\,Y\,]_{\frakh^\perp}
  \oplus 0\;\big)_{\hat \frakk}
  &=&
  \frac s{1\,+\,s}\big(\;[\,X,\,Y\,]_{\frakh^\perp}
  \;\oplus\;\hphantom{-\,s\,}[\,X,\,Y\,]_{\frakh^\perp}\,\;\big)
  \\
  \big(\;[\,X,\,Y\,]_{\frakh^\perp}
  \oplus 0\;\big)_{\hat \frakm}
  &=&
  \frac1{1\,+\,s}\big(\;[\,X,\,Y\,]_{\frakh^\perp}
  \;\oplus\;-\,s\,[\,X,\,Y\,]_{\frakh^\perp}\;\big)
 \end{eqnarray*}
 Therefore,
 \begin{eqnarray*}
  [\,X,\,Y\,]_{\hat\frakk}
  &=&
  [\,X,\,Y\,]_\frakh\,\oplus\,0 
  \;+\;\frac s{1\,+\,s}\big(\;[\,X,\,Y\,]_{\frakh^\perp}
  \;\oplus\;[\,X,\,Y\,]_{\frakh^\perp}\,\;\big)
  \end{eqnarray*}
  We note that the adjoint representation
  of $\hat\frakk$ on the second factor $0\,\oplus\,\hat\frakh^\perp$ of $\hat\frakm$ is trivial since $\frakh^\perp$ is
  an Abelian ideal. Therefore, in order to calculate the curvature tensor, we can keep the 
  canonical identification $\hat\frakk\,\cong\,\frakk$:
  \begin{eqnarray*}
  [\,X,\,Y\,]_{\hat\frakk}
  &=&
  [\,X,\,Y\,]_\frakh\;+\; \frac s{1\,+\,s}\;[\,X,\,Y\,]_{\frakh^\perp}
  \\
  &=&
  [\,X,\,Y\,]_\frakh
  \;+\;[\,X,\,Y\,]_{\frakh^\perp}
  \;-\;[\,X,\,Y\,]_{\frakh^\perp}
  \;+\; \frac s{1\,+\,s}\;[\,X,\,Y\,]_{\frakh^\perp}
  \\
   &=&
  [\,X,\,Y\,]_\frakk\;-\;\frac1{1\,+\,s}\,[\,X,\,Y\,]_{\frakh^\perp}
  \\
  &=&
  [\,X,\,Y\,]_\frakk\;-\;\frac1{|\,1\,+\,s\,|}\,\underbrace{B(\,[\,X,\,Y\,],\,Z_0\,)}_{=\,B(\,Y,\,[\,Z_0,\,X\,]\,)}\,Z_0
 \end{eqnarray*}
 It follows that 
 \begin{eqnarray*}
  \hat R(\,X,\,Y\,)\,
  &=&
  \bar R(\,X,\,Y\,)
  \;+\;\frac1{|\,1\,+\,s\,|}\,B(\,Y,\,[\,Z_0,\,X\,]\,)\,\rho_*(\,Z_0\,)
  \\
  &=&
  \bar R(\,X,\,Y\,)
  \;+\;\frac1{|\,1\,+\,s\,|}\langle\,Y,\,\rho_*(\,Z_0\,)\,X\,\rangle
  \,\rho_*(\,Z_0\,)
 \end{eqnarray*}
 for all horizontal vectors $X,\,Y$. Because $\hat R$ is a horizontal  tensor, this establishes \eqref{eq:HAT_R}.
 \end{proof}

\bigskip
 \noindent
 In the  notation of \cite{St1,St2,St3} the naturally reductive triple
 $(\,\hat g,\,\hat \frakk,\,\hat B_s\,)$ ($s\,\neq\,0$) is a $1$--dimensional
 $(\,\frakh^\perp,\,\hat B|_{\frakh^\perp\times\frakh^\perp}\,)$--extension of
 $(\,\frakg,\,\frakk,\,B\, )$.
  
  \begin{remark}\label{re:ON}
   By \cite[Eq.~(9.24)]{Bes}, the O'Neill tensor $A\,:\,\scrH\to \End^\mathrm{skew}\,T\, G\,/\,H$ 
   of the Riemannian submersion $f\,\colon\,G\,/\,H\,\longrightarrow \,G/K$ 
   satisfies $A_XY\,=\,\frac12 [X\,,\,\,Y]_\scrV\,$ for all horizontal vector fields $X$ and $Y$, 
   where we keep to the notation from the proof of Lemma \ref{le:Nor_SG}.
   Then, we have $A_XY\,=\,\frac12\,\hat\tau(\,X,\,Y\,)_\scrV$  for all horizontal vector fields: 
   Since this is a tensorial equation which is invariant under the action of $G$, it suffices 
   to check it at the origin of $G\,/\,H$ for the vector fields 
   $X^*,\,Y^*$ induced by $X,\,Y\in \frakp$ via the infinitesimal left action of $\frakg$. 
   The claim follows from the well--known relations $\hat\tau(\,X^*,\,Y^*\,)\,=\,-\,[\,X,\,Y\,]^*\,=\, [\,X^*,\,Y^*\,]$, cf. \cite{Zil1}. 
   Therefore,  the additional term $\frac1{\sqrt{|\,1\,+\,s\,|}}\,\rho_*(Z_0)\,\wedge\,V_0^\flat$ 
   from \eqref{eq:HAT_TAU} is twice the O'Neill tensor.
 \end{remark}
  
  \noindent
  As an example, we recall the construction of the Berger spaces $\hat \MM^n(\,\kappa,\,r\,)$, cf. also \cite[p.~587 ff.]{Zil1},
  \cite[9.81]{Bes} and \cite[Chpt. 9.2]{AFF}. For $\kappa\,>\,0$, let $\CP^n(\,\kappa\,)\,=\,\SU(\,n\,+\,1\,)\,/\,\U(\,n\,)\,$ 
  denote the complex projective space equipped with the Fubini--Study metric of constant holomorphic  sectional curvature 
  $\kappa$. The Hopf fibration $f\,:\,\S^{2n+1}(\,r_\kappa\,)\,\longrightarrow\,\CP^n(\,\kappa\,)$ defines a Riemannian submersion 
  from the round sphere whose radius satisfies $r_\kappa^2\,=\,\frac{4}{\kappa}$. Constant rescaling along the 
  vertical distribution gives a family of $\U(\,n\,+\,1\,)$--invariant metrics on  $\hat \MM^n(\,\kappa,\,r\,)\,:=\,\S^{2n\,+\,1}$ 
  parameterized by $r\,>\,0$  such that $f\,:\,\hat \MM^n(\,\kappa,\,r\,)\,\longrightarrow\,\CP^n(\,\kappa\,)$ is still a Riemannian 
  submersion with the same horizontal distribution  and whose fibers are circles of length $2\,\pi\,r$. 
  For $r\,=\,r_\kappa$, we recover the initial round metric.
  
   \bigskip
  \noindent
  For $\kappa\,<\,0$, let $\CH^n(\,\kappa\,)\,=\,\SU(\,n,\,1\,)\,/\,\U(\,n\,)$ denote the complex 
  hyperbolic space with the Fubini--Study metric of constant negative holomorphic sectional curvature $\kappa$. 
  By a similar construction, we obtain an $\U(\,n,\,1\,)$--invariant metric on 
  the non--compact counterpart $\hat \MM^n(\,\kappa,\,r\,)\,:=\,\SU(\,n,\,1\,)\,/\,\SU(\,n\,)$ 
  such that the canonical projection $f\,:\,\hat \MM^n(\,\kappa,\,r\,)\,\longrightarrow\,\CH^n(\,\kappa\,)$
  is a Riemannian submersion whose fibers are circles of length $2\,\pi\,r$.  
  Also let us consider the orthogonal splitting  $T \hat \MM^n(\,\kappa,\,r\,)\,=\,\scrH\oplus\scrV$ into the
  horizontal and the vertical subbundle. The canonical Hermitian structure  $J$ of $\MM^n(\,\kappa\,)\,\in\,\{\CP^n(\,\kappa\,),\,\CH^n(\,\kappa\,)\,\}$ 
  lifts to a linear map $J\,:\,\scrH\,\longrightarrow\,\scrH$. Then, $J\,\oplus\,0$ is the canonical contact structure
  of $\hat \MM^n(\,\kappa,\,r\,)$. 
  
  \bigskip
  \noindent
  For $\kappa\,=\,0$,  we consider the $(\,2\,n\,+\,1\,)$--dimensional Heisenberg group $\MM^n(\,0,\,\frac{1}{c^2}\,)$ 
  with a multiple $\frac{1}{c^2}$ of its canonical left--invariant metric of type H.
  Then, there exists a Riemannian submersion $\MM^n(\,0,\,\frac{1}{c^2}\,)\,\longrightarrow\, \C^n$ 
  and, similar as before, the canonical Hermitian structure $J$ of $\C^n$ induces a contact structure $J\oplus\,0$ on
  $T \hat \MM^n(\,0,\,\frac{1}{c^2}\,)\,=\,\scrH\oplus\scrV$, cf. also \cite[Chpt. 9.3]{AFF}.
  Furthermore, let $V_0$ be the Reeb vector field and denote  its metric dual by $V_0^\flat$.
    
  \bigskip
  \begin{proposition}\label{p:BS}
   \begin{enumerate}
    \item For $\kappa\,\neq\,0$  there exists a naturally reductive structure $\hat\nabla$ 
             on $\hat \MM^n(\,\kappa,\,r\,)$ whose torsion form and 
             curvature tensor are given by  
             \begin{equation}\label{eq:BS_TOR}
               \hat\tau
              \;\;=\;\;
              c\,\,J\,\wedge\,V_0^\flat
              \qquad\qquad
              \hat R
              \;\;=\;\;
              R_{\mathrm{FS}}(\,\kappa\,)\,+\,c^2\,J\,\otimes\,J
             \end{equation}
             with 
            \begin{equation}\label{eq:c2}
              c
              \;\;=\;\;\frac12\;r\,\kappa 
            \end{equation}
      where $R_{\mathrm{FS}}(\,\kappa\,)$ denotes the curvature tensor of $\MM^n(\,\kappa\,)\,\in\,\{\,\CP^n(\,\kappa\,),\,\CH^n(\,\kappa\,)\,\}$.
    \item On $\hat \MM^n(\,0,\,\frac{1}{c^2}\,)$ there exists  a naturally reductive structure $\bar \nabla$ such that  
             \begin{equation}\label{eq:HG_TORCURV}
              \hat\tau \;\;=\;\;
              c\,J\,\wedge\,e_{2n+1}
              \qquad\qquad
              \hat R
              \;\;=\;\;
              c^2\, J\,\otimes\,J
             \end{equation}
  \end{enumerate}
 \end{proposition}
 \begin{proof}
  For simplicity, we start with the case $\kappa>0$. The Fubini--Study metric of $\CP^n(\,\kappa\,)\,=\, G\,/\,K$ is induced by a suitable 
  positive definite invariant symmetric bilinear form $B$ on $\su(\,n\,+\,1\,)$. 
  In order to apply  Lemma \ref{le:Nor_SG}, let $G\,:=\,\SU(\,n\,+\,1\,)$, $
  K\,:=\,\U(\,n\,)\,\subset\, G$ and consider the normal subgroup $H\,:=\,\SU(\,n\,)\,\subset\, K$. 
  Then, $K\,/\,H\,=\,\U(\,1\,)$ and there exists an $(\,n\,+\,1\,)$--fold covering of $G\,\times\,K\,/\,H$ on $\hat G\,:=\,\U(\,n\,+\,1\,)$.
  Since the isotropy representation of $\CP^n(\,\kappa\,)$ is irreducible and the metric tensor 
  of $\hat \MM^n(\,\kappa,\,r\,)$ is both $\hat G$--invariant and bundle--like, 
  it belongs automatically to the family of naturally reductive metrics parameterized by $s\,>\,-\,1$ considered
  in the first part of Lemma \ref{le:Nor_SG}.  Furthermore, the center of $\u(n\,)$ is spanned by
  \begin{equation}\label{eq:Z}
   Z
   \;\;:=\;\;
   \frac1{n+1}\;\diag(\,-\,n\, i,\,i,\,i,\,\cdots,\,i\,)\,\in\su(\,n\,+\,1\,)
  \end{equation}
  and $J\,:=\,\rho_*(Z)\,$ is the standard complex structure of  $\CP^n$. Because of \eqref{eq:HAT_TAU} and \eqref{eq:HAT_R}
  we see that \eqref{eq:BS_TOR} holds for some $c\,>\,0$. It remains to verify \eqref{eq:c2}.  For this, note that both sides of this equation 
  are proportional to $r$: For the r.h.s. this is obvious. For the l.h.s. this follows from Remark \ref{re:scaling} in combination with \eqref{eq:HAT_TAU}.
  Therefore, it suffices to verify \eqref{eq:c2} for one specific value of the parameter $r$, 
  which we choose as $r\,:=\,r_\kappa$ with $r_\kappa^2\,:=\,\frac{4}{\kappa}$:
  It is well known that the curvature tensor of $\CP^n(\,\kappa\,)$ is characterized by
  \begin{equation}\label{eq:FS_CURV}
   R_{\mathrm{FS}}(\,\kappa\,)(\,Y,\,X,\,X,\,Y\,)
   \;\;=\;\;
   \frac{\kappa}4
   \;\left(\;\|X\|^2\|\,Y\|^2\;-\;\langle\,X,\,Y\,\rangle^2\; +\;3\;\langle\,J\,X,\,Y\,\rangle^2\;\right)
  \end{equation} 
  Furthermore, according to \cite[Eq.~(9.29\,c)]{Bes} the sectional curvatures of $\hat \MM^n(\,\kappa,\,r\,)$ 
  and $\CP^n(\,\kappa\,)$ are related by
  $$
   \hat\kappa(\,X,\,Y\,)\,=\,\kappa\,(\,f_*X,\,f_*Y\,)\,-\,3\,\|A_XY\|^2
  $$ for every orthonormal pair of horizontal vectors $X,\,Y$ where $A$ denotes the 
  O'Neill tensor of $f$.  Due to \eqref{eq:BS_TOR} in combination with Remark \ref{re:ON}	
  we have $\|A_XY\|^2\,=\,\frac{c^2}{4}\,\langle\, J\, X\,,\, Y\,\rangle^2$. 
  Inserting \eqref{eq:FS_CURV}, we conclude that
  \begin{equation}\label{eq:ONeill}
   \hat\kappa(\,X,\,Y\,)\;\;=\;\; \frac14\,\kappa \;\|X\|^2\|\,Y\|^2\; + \;\frac34\kappa\left(\;1\, - \,\frac{c^2}{\kappa}\;\right)\,\langle\, J\, X\,,\, Y\,\rangle^2
  \end{equation}
  Therefore, if the sectional curvature of $\hat \MM^n(\,\kappa,\,r\,)$ is constant equal to $\frac{\kappa}{4}$, then necessarily $c^2\,=\,\kappa$ holds.
  Consequently, $c^2\,=\,\kappa\,=\,\frac{1}{4}r_\kappa^2\,\kappa^2$, i.e., \eqref{eq:c2} holds for $r\,=\,r_\kappa$ and thus for all $r\,>\,0$. 
  
  \noindent
  Formula \eqref{eq:c2} can also be verified by a direct calculation: The Fubini--Study metric of constant holomorphic
  sectional curvature $\kappa$ is induced by $B\,:=\,-\,\frac2{\kappa}\,\trace_{\su(\,n\,+\,1\,)}$
  where $\trace_{\su(\,n\,+\,1\,)}(\,A,\,B\,)\,:=\,\trace(A\circ B)$.
  We have $B(\,Z,\,Z\,) \,=\,\frac2{\kappa}\frac{n}{n\,+\,1}$ and hence 
  $c^2\,=\, \frac{\kappa}{2}\frac{n\,+\,1}{n}\frac{1}{1\,+\,s}$ by \eqref{eq:HAT_TAU}. 
  Also recall that $Z$ defined in \eqref{eq:Z} can be interpreted at the same time as the 
  vertical vector $V\,:=\,(\,Z,\, -s\, Z\,)$. Then, $\|V\|^2\,=\,
  \frac1{\kappa}\,\frac{2\,n}{n\,+\,1}(\,1\,+\, s\,)$. The geodesic $\gamma(t)\,=\,
  \exp(\,t\,V)$ in the direction of $V$ is a simply closed curve of period
  $\frac{T}{2\,\pi}\,=\,\frac{n\,+\,1}{n}\frac{1}{1\,+\,s}$, and parameterizes a Hopf
  circle. Therefore, the length $L\,=\,T\,\|V\|$ of this circle satisfies $\frac{L^2}{(\,2\,\pi\,)^2}
  \,=\,  \frac2{\kappa}\frac{n\,+\,1}{n}\frac{1}{1\,+\,s}\,=\,\frac{4\,c^2}{\kappa^2}\,$. 
  Substituting $L\,=\,2\,\pi\,r$ again gives \eqref{eq:c2}.
  
 \noindent
 The case $\kappa\,<\,0$ is treated in the same way with the modifications
  indicated by Lemma~\ref{le:T_and_R} and ends up with the same formulas \eqref{eq:BS_TOR} 
  and \eqref{eq:c2}. For \eqref{eq:HG_TORCURV}, see \cite[Eq. (\,9.24\,)]{TV} and~\cite[p.~546]{St1} 
 \end{proof}

  \bigskip
  \noindent
  From \eqref{eq:ONeill}, we see that the sectional curvature of $\hat \MM^n(\,\kappa,\,r\,)$ 
  is positive if and only if $0\,<\,r^2\,\kappa\,<\,\frac{16}3$. Using
  $r_\kappa^2\,=\,\frac{4}{\kappa}$ we obtain the known estimate 
  $\frac{r^2}{r_\kappa^2}\,<\,\frac{4}3$ for the factor $\frac{r^2}{r_\kappa^2}$ by which we may rescale the 
  metric of the round sphere $\S^{2n+1}(\,r_\kappa\,)$
  in the vertical direction such that the sectional curvature  remains positive, cf. \cite{VZ,Zil1}.

 \bigskip
 \begin{example}(cf. \cite{Jen}, \cite[9.83]{Bes})\label{ex:Je}
  We let $G\,:=\,\Sp(\,n\,+\,1\,)$, $K\,:=\,\Sp(\,1\,)\times \Sp(\,n\,)$ and $H\,:=\, \{\id\}\times\Sp(\,n\,)$. 
  Then, $H$ is a normal subgroup of $K$ and $K/H\,=\,\Sp(\,1\,)$.  Let $B$ be the negative of the Killing form of $\sp(\,n\,+\,1\,)$.  
  For each $\kappa > 0$, we obtain a family of $\hat G\,:=\,\Sp(\,n\,+\,1\,)\,\times\, \Sp(\,1\,)$--invariant 
  naturally reductive metrics on the total space $\S^{4\,n\,+\,3}\,=\,\Sp(\,n\,+\,1\,)\,/\Sp(\,n\,)$ of the quaternionic Hopf bundle
  over the quaternionic projective space $\HP^n(\,\kappa\,)$ of constant quaternionic sectional curvature $\kappa$. 
  The canonical projection map is a Riemannian submersion and each fiber is isometric to a round $3$--sphere of radius $r$. 
  We get the same condition $\frac{r^2}{r_\kappa^2}\,<\,\frac{4}3$ for positiveness of the sectional curvature
  as for the Berger spheres, cf. \cite[Chpt. 2]{VZ}.
 \end{example}

\section{Generalized Vector Cross Products}
\label{se:GVCP}
 Let $\sigma:\,T\times T\times T \longrightarrow \R$ be an alternating $3$--form
 on a euclidean vector space $T$ with positive definite scalar product.
 For each $X\,\in\,T$, we denote by $\sigma_X\,\in\,\so(\,T\,)$ the skew--symmetric
 endomorphism of $T$ characterized by the following equality valid for all
 vectors $Y,\,Z\,\in\,T$:
 \begin{equation}\label{eq:sigmaX}
  \langle\, \,\sigma_XY,\,Z\,\rangle\;\;:=\;\;\sigma(\,X,\,Y,\,Z\,)
 \end{equation}

 \begin{definition}
 \label{de:VCP}
  An alternating $3$--form $\sigma\,\in\,\L^3T^*$ on a euclidean vector
  space $T$ is called a vector cross product in the sense of Gray if it
  satisfies for all $X,\,Y\,\in\,T$ the following identity:
  $$
   \langle\,\sigma_XY,\,\sigma_XY\,\rangle
   \;\;=\;\;
   \langle\,X\,\wedge\,Y,\,X\,\wedge\,Y\,\rangle
   \;\;:=\;\;
   \langle\,X,\,X\,)\,\langle\,Y,\,Y\,\rangle\;-\;\langle\,X,\,Y\,\rangle^2
  $$
 \end{definition}

 \noindent
 Every vector cross product $\sigma$ induces an orthogonal multiplication
 on $\R\oplus T$, hence the classification of orthogonal multiplications
 via $\Z_2$--graded Clifford modules tells us that there are only two examples
 up to isomorphism besides the trivial example $\sigma\,=\,0$ in dimension $1$
 \cite{BMS,FKMS}:
 \begin{itemize}
  \item In three dimensions there exists the well--known vector product
        on $T\,=\,\R^3$ which measures the directed area of two vectors.
        In terms of the standard basis $\{E_1,\,E_2,\,E_3\}$ of $\R^3$ the
        corresponding $3$--form $\sigma\,\in\,\L^3\R^{3*}$ equals the
        standard volume form $\det$:
        $$
         \sigma
         \;\;=\;\;
         dE_1\,\wedge\,dE_2\,\wedge\,dE_3
        $$
  \item The imaginary part of the product of octonions is a vector cross
        product on $T\,=\,\R^7$; its stabilizer in $\GL(\,\R^7\,)$ defines
        the exceptional Lie group $\G_2$ as a subgroup of $\SO(\,7\,)$. In
        terms of the standard basis $\{E_1,\,\ldots,\,E_7\}$ of $\R^7$ the
        corresponding $3$--form $\sigma\,\in\,\L^3\R^{7*}$ reads as:
        \begin{equation}\label{eq:sigma}
         \begin{array}{lcl}
          \sigma
          &\;=\;&
          \Big(\;dE_1\,\wedge\,dE_2\;+\;dE_3\,\wedge\,dE_4
          \;+\;dE_5\,\wedge\,dE_6\;\Big)\,\wedge\,dE_7
          \\[4pt]
          &&
          \;+\;\Re\;\Big(\;(\,dE_1\,+\,i\,dE_2\,)\,\wedge\,
          (\,dE_3\,+\,i\,dE_4\,)\,\wedge\,(\,dE_5\,+\,i\,dE_6\,)\;\Big)
         \end{array}
        \end{equation}
 \end{itemize}
 In order to understand the following definition it is helpful to note that
 $\sigma$ is a vector cross product if and only if $\sigma_X\,\in\,\so(\,T\,)$
 restricted to the orthogonal complement $\{\,X\,\}^\perp\,\subset\,T$ of a unit
 vector $X\,\in\,T$ is always an orthogonal complex structure in the sense
 $(\,\left.\sigma_X\right|_{\{X\}^\perp}\,)^2\,=\,-\id_{\{X\}^\perp}$. In
 particular, $\sigma_X$ and $\sigma_Y$ are conjugated under $\O(\,T\,)$ for
 all unit vectors $X,\,Y$.

 \begin{definition}(\cite[Definition~2.3]{BMS})
 \label{de:GVCP}
  An alternating $3$--form $\tau\,\in\,\L^3T^*$ on a euclidean vector space
  $T$ with positive definite scalar product $g$ is called a generalized vector
  cross product if $\tau\,\neq\,0$ and $\tau_X$ is conjugate to $\tau_Y$ in
  $\O(\,T\,)$ for all unit vectors $X,\,Y\,\in\,T$.
 \end{definition}

 \noindent
 In order to provide an example of a generalized vector cross product in
 dimension six, we consider the subspace $T\,\subset\,\R^7$ of vectors with
 vanishing seventh component $dE_7\,=\,0$ and the $3$--form $\tau\,\in\,
 \L^3T^*$ defined as the restriction of the $3$--form
 $\sigma\,\in\,\L^3\R^{7*}$ of Equation \eqref{eq:sigma} to $T$:
 \begin{equation}\label{eq:tau}
  \begin{array}{lcl}
   \tau
   &:=&
   \Re\;\Big(\;(\,dE_1\,+\,i\,dE_2\,)\,\wedge\,
   (\,dE_3\,+\,i\,dE_4\,)\,\wedge\,(\,dE_5\,+\,i\,dE_6\,)\;\Big)
   \\[4pt]
   &=&
   dE_1\wedge dE_3\wedge dE_5
   \,-\,dE_1\wedge dE_4\wedge dE_6\\
   &&\ \ \ \ \ \ \ \ \ \ \ \ \ \ \ \ \ \ \ \ \ \ \,\;-\,dE_2\wedge dE_3\wedge dE_6
   \,-\,dE_2\wedge dE_4\wedge dE_5
  \end{array}
 \end{equation}
 By definition, $\tau$ equals the real part of the standard complex volume
 form $\det$ on $\C^3\,\cong\,\R^6$. Because $\SU(\,3\,)$ acts transitively on
 the unit sphere of $\C^3$, the $\SU(\,3\,)$--invariance of the determinant
 $\det(AX,AY,AZ)\,=\,\det(X,Y,Z)$ for all $A\,\in\,\SU(\,3\,)$ and $X,\,Y,\,Z
 \,\in\,\C^3$ ensures that $\tau$ is a generalized vector cross product. In
 analogy to a vector cross product the eigenvalues of $\tau_X^2$ are $0$ 
 and $-1$  for a unit vector $X\,\in\,T$; however the eigenspace of $0$ 
 equals the $2$--dimensional subspace $\{\,X,\,iX\,\}\,\subset\,T$.
 It turns out that this is the only example of a generalized vector cross
 product besides the classical vector cross products of Definition
 \ref{de:VCP}:

 \begin{theorem}(\cite[Theorem~2.6]{BMS})
 \hfill\label{th:BMS}\break
  Every generalized vector cross product on an odd dimensional euclidean
  vector space $T$ with positive definite scalar product $g$ is a constant
  rescaling of a classical vector cross product and so $\dim\,T\,=\,3$ or
  $\dim\,T\,=\,7$. The assumption $\sigma\,\neq\,0$ excludes the trivial
  example in dimension $m\,=\,1$. Up to isomorphism and constant rescaling
  every generalized vector cross product on an even dimensional euclidean
  vector space $T$ is isomorphic to the generalized vector cross product
  $\tau$ of Equation \eqref{eq:tau} and consequently $\dim\,T\,=\,6$.
 \end{theorem}

 \noindent
 Recall that the torsion form of a  naturally reductive 
 space $(\,M,\,g,\, \bar\nabla\,)$ is parallel with respect to $\bar\nabla$.
 Its algebraic type is therefore the same at every point and we can ask whether
 it is a generalized vector cross product:

 \begin{theorem}\label{th:GVCP}
  The torsion form $\tau$ of a simply connected naturally reductive space 
  $(\,M,\,g,\, \bar\nabla\,)$ is a generalized vector cross product if and only if
  one of the following cases holds:
  \begin{itemize}
  \item $\dim\,M\,=\,3$ and $\tau\neq 0$;
  \item $\dim\,M\,=\,6$ and there exists a strict nearly K\"ahler structure $J$ 
           such that $\bar\nabla$ is the canonical Hermitian connection: 
           the standard normal homogeneous spaces $\S^6_{\mathrm{round}}\,=\,\G_2/\SU(\,3\,)$, 
           $\bbF^3$, $\CP^3$ and $\S^3\,\times\,\S^3$.
  \item $\dim\,M\,=\,7$ and $(\,M,\,g,\,\bar\nabla)$ is the round sphere $\Spin(\,7\,)/\G_2$, or 
           there exists a proper nearly parallel $\G_2$--structure $\sigma$ whose underlying metric tensor is $g$, 
           and $\bar\nabla$ is the canonical $\G_2$--connection: 
           the standard normal homogeneous spaces $\S^7_{\mathrm{squashed}}$, $V_1$ and $V_3$.
\end{itemize}
 \end{theorem}
\noindent
 The proof of Theorem \ref{th:GVCP} in dimension three is trivial since 
 here every non--vanishing alternating $3$--form is a vector cross product. Its proof in dimensions six and seven 
 follows from the results presented in the next two sections.
 
\section{Six--Dimensional Homogeneous Nearly K\"ahler Manifolds}
\label{se:NR_NK}
 Following \cite{BMS,MNS}, an $\SU(\,3\,)$--structure on an oriented Riemannian manifold 
 $(\,M^6,\,g\,)$ is a reduction of the $\SO(\,6\,)$--principal bundle of positive orthonormal 
 frames to $\SU(\,3\,)$. Since the latter is the intersection $\SL(\,3,\,\C)\cap\SO(\,6\,)$ and hence the stabilizer in 
 $\SO(\,6\,)$ of  the standard complex volume form $(\, dE_1\,+\,i\,dE_2\,)\,\wedge\, (\,dE_3\,+\,i\,dE_4\,)\,
 \wedge\,(\,dE_5\,+\,i\,dE_6\,)$ of $\C^3$, such a reduction is defined by a triple $(\,J,\,\Psi^+,\,\Psi^-\,)$ 
 where $J$ is a Hermitian structure of $(\,M,\,g\,)$, compatible with the given orientation,
 and $\Psi^+$ is the real part of a complex volume form  $\Psi\,:=\, \Psi^+\,+\,i\,\Psi^-$
 of constant length. Then, there exists a real number $c\, >\, 0$ such that 
 locally $\Psi\,=\,c\,(\, dE_1\,+\,i\,dE_2\,)\,\wedge\, (\,dE_3\,+\,i\,dE_4\,)\,\wedge\,(\,dE_5\,+\,i\,dE_6\,)$
 in relation to a suitable positive orthonormal frame field $E_1,\,\ldots,\,E_6$ that fulfills the condition
 $E_2\,=\,JE_1$, $E_4\,=\,JE_3$ and $E_6\,=\,JE_5$, where we stick to the notation 
 $dE_i$ for the dual frame $E_i^\flat\,=\,g(\,E_i,\,\cdot\,)$. 
 Additionally, $\Psi^+$ and  $\Psi^-$ are both $3$--forms of type $(\,3,\,0\,)\,+\,(\,0,\,3\,)$
 of the same constant length related by:
 \begin{eqnarray}\label{eq:Psi_J}
  \Psi^-
   &\;\;=\;\;&
  -\, \Psi^+\circ J \times \id \times \id
 \end{eqnarray}
 Therefore, an $\SU(\,3\,)$--structure can equally be
 defined by a tuple $(\,J,\,\Psi^+\,) $ where $\Psi^+$ is 
 a $(\,3,\,0\,)\,+\,(\,0,\,3\,)$--form of constant length. 
 In particular, $\Psi^+$ is a multiple of the right hand side of  \eqref{eq:tau} 
 and therefore a generalized vector cross product. Then, one can show hat the following 
 constant type equation holds (cf. \cite[Definition 2.1]{BM}, \cite[Theorem 5.2]{G4}):
 \begin{equation}\label{eq:CT}
  \langle\;\Psi^+(\,X,\,Y\,)\;,\;\Psi^+(\,X,\,Y\,)\;\rangle
  \;\;=\;\;
  c^2\,\Big(\;\|\,X\,\|^2\,\;\|\,Y\,\|^2\;-\;\langle\,X,\,Y\,\rangle^2
  \;-\;\langle\,J\,X,\,Y\,\rangle^2\;\Big)
 \end{equation}
 simply by checking it for $X\,:=\,E_1$ and  $Y\,:=\,E_k$ with $k\,=\,1,\,\ldots,\,6$.
 Whereas $\SU(\,n\,)$--structures are defined analogously for arbitrary $n$, 
 it is important to note that \eqref{eq:CT} is reserved to dimension six, 
 cf. also \cite[Chpt. 5.3, Remark 5]{BFGK}.
 It is also straightforward to conclude from the local expression above for $\Psi$ that
 \begin{eqnarray}
  \Psi^-
  &\;\;=\;\;&
  *\,\Psi^+
  \label{eq:Psi_*}
 \end{eqnarray}
 Here, $*$ denotes the Hodge dual on differential forms of $(\,M,\,g\,)$. Thus, $\Psi^+\wedge\Psi^-$ 
 is a non--vanishing $6$--form of constant length and hence a multiple of the Riemannian volume form 
 $\det(\,M,\,g\,)$. We will call an $\SU(\,3\,)$--structure normalized
 if $\frac14\,\Psi^+\wedge\Psi^-$ is equal to $\det(\,M,\,g\,)$, cf. \cite[Chpt. 3]{BMS}. 
 Equivalently \eqref{eq:CT} holds with $c\,=\,1$. 
 Because of \eqref{eq:Psi_J}, \eqref{eq:Psi_*} the $\SU(\,3\,)$--reduction is in fact 
 completely described by the generalized vector cross product $\Psi^+$ alone:

 \begin{lemma}
 \label{le:SU3_STRUC}
  Let $(\,M,\,g\,)$ be an oriented $6$--dimensional Riemannian manifold. Assigning
  to a triple $(\,J,\,\Psi^+,\,\Psi^-\,)$ the $3$--form $\tau\,:=\,\Psi^+$
  defines a 1--1 correspondence between normalized $\SU(\,3 \,)$--structures of $(\,M,\,g\,)$ 
  and $3$--forms $\tau$ which can be represented locally according to \eqref{eq:tau} using
  a suitable positive orthonormal frame field $\{E_1,\,\ldots,\, E_6\}$.
 \end{lemma}

 \begin{proof}
 The fact that  our assignment is well defined was already mentioned before. 
 We claim that it is also injective, i.e., that $(\,J,\,\Psi^+,\,\Psi^-\,)$ is uniquely determined by $\,\Psi^+$:
 We already know that $\Psi^+$ together with the given orientation determines $\Psi^-$ by \eqref{eq:Psi_*}. 
 We claim further that $J$ is determined by the tuple $(\,\Psi^+,\,\Psi^-\,)$. For this, note 
 that the map $\Lambda^2TM\,\to TM$ defined by $X\wedge Y\,\mapsto \Psi^-_X\,Y$ 
 is $\SU(\,3\,)\,=\,\SU(\,T_pM,\,J\,)$--equivariant and non--trivial since $\Psi^-\,\neq\,0$. 
 Schur's Lemma implies that this map is surjective because $\SU(\,3\,)$ acts irreducibly on $T_pM$.
 Hence, $T_pM$ is spanned by tangent vectors $Y$ such that  $Y\,=\, \Psi^-_{Z_1}Z_2$
 for certain $Z_1$ and $Z_2$. Thus,
 $$
  \langle\;JX,\;Y\;\rangle
  \;\;=\;\;
  \Psi^-(\;Z_1,\;Z_2,\;JX\;)
  \;\;=\;\;
  \Psi^-(\;JX,\;Z_1,\;Z_2\;)
  \;\;=\;\;
  \Psi^+(\;X,\;Z_1,\;Z_2\;)
 $$
 according to \eqref{eq:Psi_J} exploiting that $J^2\,=\,-\,\id$. 
 In particular, $J$ is uniquely determined by $\Psi^+$ and $\Psi^-$. 
 It remains to prove that our assignment is surjective: Let $\tau$ be a $3$--form 
 given locally by~\eqref{eq:tau} with respect to a positive orthonormal frame field $\{\,E_1,\,\ldots,\, E_6\,\}$. 
 Set $\Psi^+\,:=\,\tau$, let $J$ be the standard Hermitian
 structure whose K\"ahler form is $dE_1\,\wedge\,dE_2\,+\,dE_3\,\wedge\,dE_4\,+\,dE_5\,\wedge\,dE_6$, and define
 $\Psi^-$ by \eqref{eq:Psi_J}. This way we obtain locally normalized $\SU(\,3\,)$--structures  $(\,J,\,\Psi^+,\,\Psi^-\,)$ with $\Psi^+\,=\,\tau$. 
 Since these structures are uniquely determined by $\tau$, our construction automatically glues resulting in a globally 
 defined normalized $\SU(\,3\,)$--structure such that $\Psi^+\,=\,\tau$. 
 \end{proof}

 \bigskip\noindent
 Examples of $\SU(\,3\,)$--structures are given by $6$--dimensional strict nearly
 K\"ahler manifolds \cite{BM,But2,FH,G4,MNS,Nag}:

 \begin{definition}\label{de:NK}
  An almost Hermitian manifold $(\,M^{2n},\,g,\,J\,)$ with Levi--Civita
  connection $\nabla$ is called nearly K\"ahler if $\nabla J$ is
  skew in the sense
  \begin{equation}
   (\,\nabla_XJ\,)\,Y
   \;\;=\;\;
   -\;(\,\nabla_YJ\,)\,X
  \end{equation}
 A nearly K\"ahler manifold  is called strict provided $\nabla_X J\,\neq\,0$ for each $X\,\neq\, 0$.
 \end{definition}

 \noindent
 Every $6$--dimensional nearly K\"ahler manifold $(\,M,\,g,\,J\,)$ is K\"ahler  (i.e., $\nabla J\,=\,0$) 
 or strict nearly K\"ahler  \cite[Proposition 2.1]{Nag}.  Differentiating the relation $J^2\,=\,-\,\id$ with respect to $\nabla$ one obtains that 
 $\nabla_XJ\circ J\,=\,-J\circ \nabla_XJ$. So we conclude that $\Psi^+(\,X,\,Y,\,Z\,)\,:=\,\langle\, (\,\nabla_X J\,) J\, Y\,,\,Z\,\rangle$  
 is a $3$--form of type $(\,3,\,0\,)\,+\,(\,0,\,3\,)$. This form is also of constant length since it is parallel with respect to the 
 canonical Hermitian connection $\nabla^c$ defined below, cf.~\cite[Chpt. 2]{BM}. 
 Because of \eqref{eq:Psi_J} the corresponding $\SU(\,3\,)$--structure is $(\,J,\,\nabla J\circ J,\,\nabla J\,)$. 
 One can show that every $6$--dimensional 
 strict nearly K\"ahler manifold is an Einstein space of positive scalar 
 curvature $\scal\,>\,0$ and that the underlying $\SU(\,3\,)$--structure is normalized for $\scal\,=\,30$, 
 cf.  \cite[Chpt. 5.3]{BFGK}, \cite[Theorem 5.2]{G4} and \cite[Chpt. 2]{Nag}.

 \bigskip
 \noindent
 For every almost Hermitian manifold $(\,M,\,g,\,J\,)$, the
 canonical Hermitian connection is given by 
 \begin{equation}\label{eq:IHC}
  \nabla^c
  \;\;:=\;\;
  \nabla\;+\;\frac12\;\nabla J\circ J
 \end{equation}
 Alternatively we can write $\nabla^c\,=\,\nabla\,-\,\frac12\,J\circ \nabla J$, cf.~\cite[Definition~1.1]{BM}.	
 It is easy to show that $\nabla^c$ is a metric connection such that $\nabla^c J\,=\,0$. Also note
 that  the canonical Hermitian connections  of $(\,M,\,g,\,J\,)$ and $(\,M,\,g,\,-\,J\,)$ are the same.

 \begin{definition}
  \label{de:NR_NK}
  We will call a $6$--dimensional compact strict nearly K\"ahler manifold $(\,M,\,g,\,J\,)$ homogeneous
  if $\nabla^c$ defined by \eqref{eq:IHC} is a naturally reductive structure of $(\,M,\,g\,)$.
 \end{definition}
  
 \noindent 
 The condition that $(\,M,\,g,\,J\,)$  is strict excludes the Hermitian symmetric spaces from this definition.

 \begin{remark} \label{re:NR_NK}
 In the situation of Definition \ref{de:NR_NK}, the transvection group $\Tr(\,\nabla^c\,)$
 acts transitively by holomorphic isometries on $(\,M,\,g,\,J\,)$, which is the standard definition 
 of a homogeneous nearly K\"ahler manifold. Conversely, by \cite[Theorème 1.1]{But1} every
 six--dimensional homogeneous strict nearly K\"ahler manifold in this classical sense is a 
 $3$--symmetric naturally  reductive space equipped with its canonical nearly K\"ahler 
 structure and hence complies with Definition \ref{de:NR_NK}, cf. \cite[Chpt. 1]{But2}.
 Therefore, Definition \ref{de:NR_NK}  provides the correct "definition'' of  homogeneous 
 strict nearly K\"ahler manifolds in dimension six.
 \end{remark}
 
\noindent
 The following result establishes the classification of $6$--dimensional simply connected 
 naturally reductive spaces  whose torsion form is a generalized vector cross product, 
 as claimed by Theorem \ref{th:GVCP}:

 \begin{proposition}
 \label{p:NK}
  The assignment $(\,M,\,g,\,J\,)\, \longrightarrow \,(\,M,\,g,\,\nabla^c\,:=\,\nabla\,+\,\frac12\,
 \nabla J\circ J\,)$ defines a 1--1 correspondence between:
  \begin{itemize}
   \item $6$--dimensional homogeneous strict nearly K\"ahler
            manifolds $(\,M,\,g,\,J\,)$;
   \item $6$--dimensional oriented naturally reductive spaces
           $(\,M,\,g,\,\bar\nabla\,)$ whose torsion form is a generalized vector cross product.
  \end{itemize}
  The four simply connected examples are mentioned at the beginning of Section \ref{se:ingredients}.
 \end{proposition}

  \begin{proof}
   Let $(\,M,\,g,\,J\,)$ be a  $6$--dimensional homogeneous 
   strict nearly K\"ahler manifold. Then, $\nabla^c$  defined in \eqref{eq:IHC} is a
   naturally reductive structure according to Remark \ref{re:NR_NK}. Therefore, 
   $(\,M,\,g,\,\nabla^c\,)$ is a naturally reductive space which is canonically 
   orientated by $J$.  Also the torsion form $\nabla J\circ J$ is a non--vanishing 
   $(3,0)\,+\,(0,3)$ form of  constant length and hence a generalized vector cross 
   product. Therefore, our assignment is well defined. It is also injective according to 
   Lemma  \ref{le:SU3_STRUC}, and it remains to show that it is surjective: Let 
   $(\,M,\,g\,,\bar\nabla\,)$ be an oriented $6$--dimensional  naturally reductive space 
   with torsion form $\tau$. Suppose also that there are locally positive orthonormal 
   frame fields $\{\,E_1,\,\ldots,\, E_6\,\}$ such that  $\frac1{c}\,\tau$ is given by 
   Equation~\eqref{eq:tau} for a constant $c\,\neq\,0$. As a result, there  exists a 
   normalized $\SU(\,3\,)$--structure $(\,J,\,\Psi^+,\,\Psi^-\,)$ of $(\,M,\,g\,)$ such that 
   $\tau\,=\,c\,\Psi^+$ according  to Lemma \ref{le:SU3_STRUC}. Moreover, 
   $\bar\nabla\tau\,=\,0$ implies that $\bar\nabla$ is an $\SU(\,3\,)$--connection
   because of the uniqueness assertion of Lemma \ref{le:SU3_STRUC}.
   In particular $\bar\nabla J\,=\,0$. Thus, we have $\nabla_XJ\,=\,-\frac12\, \tau_X\,\star\,J$,
   where $\star$ denotes the infinitesimal left action of skew--symmetric 
   endomorphisms on tensors. Furthermore, we claim that $\tau_X\,\star\,J$ 
   is given by $-\,2\,X\,\lrcorner\,*\tau$, where we 
   implicitly identify $J$ with its K\"ahler form $\omega(\,X,\,Y\,)\,:=\,\langle\,J\,X,\,Y\,\rangle$:
   Since the equation in question is linear in $\tau$, we can assume that $c\,=\,1$.
   Because this equation is also invariant under $\SU(\,3\,)$ and the latter group acts
   transitively on the unit sphere $\S^6$, it is sufficient to consider $X\,:=\,E_1$. From:
   \begin{eqnarray*}
    \tau_X
    &\;\;=\;\;&
    E_1\,\lrcorner\,\Psi^+
    \;\;\stackrel{(\ref{eq:tau})}=\;\;
    dE_3\wedge dE_5
    \,-\,dE_4\wedge dE_6
    \\
    X\,\lrcorner\,*\,\tau
    &\;\;\stackrel{(\ref{eq:Psi_*})}{=}\;\;&
    E_1\,\lrcorner\,\Psi^-
    \;\;\stackrel{\eqref{eq:Psi_J}}{=}\;\;
    \,dE_3\wedge dE_6
    \,+\,dE_4\wedge dE_5
   \end{eqnarray*}
   we obtain  by means of the formula $A\star\omega\,=\,\sum_\mu (\,AE_\mu\,)^\flat\,\wedge
   \,E_\mu\,\lrcorner\,\omega$ (valid for every skew--symmetric endomorphism 
  $A$ and every alternating form $\omega$)
   \begin{eqnarray*}
    \tau_X \star \omega
    &\;\;=\;\;&
    \sum_\mu
    E_\mu\,\lrcorner\,E_1\,\lrcorner\,\Psi^+\,\wedge\,E_\mu\,\lrcorner\,\omega \\
    &\;\;=\;\;&
    -\;2\,(\; dE_3\wedge dE_6
     \,+\,dE_4\wedge dE_5\;) \;\;=\;\;
    -\;2\, X\,\lrcorner\,*\,\tau
   \end{eqnarray*}
   Therefore, $\langle\,(\nabla_XJ)Y,\,Z\,\rangle\,=\,-\frac12\,\langle\,(\tau_X\star J)Y,\,Z\,\rangle
   \,=\,*\,\tau(\,X,\,Y,\,Z\,)$, i.e., $\nabla J\,=\,*\,\tau$. 
   In particular $(\,M,\,g,\,J\,)$ is a strict nearly K\"ahler manifold
   since $*\tau$ is a non--degenerate alternating form.
   We also have $*\,\tau(\,X,\,J\,Y,\,Z\,)
   \,=\,\tau(\,X,\,Y,\,Z)$ because of \eqref{eq:Psi_J} and \eqref{eq:Psi_*}. 
   It follows that 
   $$
    \tau(\,X,\,Y,\,Z)\;\;=\;\;*\,\tau(\,X,\,J\,Y,\,Z\,)\;\;
    =\;\;\langle\,\nabla_XJ (\,J\, Y\,),\,Z\,\rangle
   $$
   We conclude that $\tau\,=\,\nabla J\circ J$ and hence $\bar\nabla\,=\,\nabla^c$.
   Thus, $(\,M,\,g\,,J\,)$ is a homogeneous strict nearly K\"ahler manifold. 
 \end{proof}
\subsection{Holomorphic Sectional Curvature of Nearly K\"ahler Manifolds}
\label{se:HSC}
 Let $(\,M,\,g,\,J\,)$ be a nearly K\"ahler manifold. 
 The holomorphic sectional curvature  $ \rmH$ assigns to each unit vector $X$ 
 the sectional curvature $\kappa(\,X,\,J\,X\,)$ of the 
 two--plane $\{X,\,J\,X\}_\R\,\subset\,T_pM$.

 \begin{proposition}
 \label{p:CHSC}
  For a $6$--dimensional strict nearly K\"ahler manifold $(\,M,\,g,\,J\,)$
  the following statements are equivalent:
  \begin{enumerate}[(a)]
   \item $(\,M,\,g\,)$ is a round sphere of constant curvature $\rmH\,>\,0$.
   \item $\rmH$  is a constant function on the unit sphere bundle $\S(\,TM\,)$.
   \item $\K_0(X)JX$ is a multiple of $JX$ for all $p\,\in\,M$ and $X\,\in\,T_pM$.
   \item $\langle\, \K_0(X)\,U,\,\Psi^+_X\, U\,\rangle\,=\,0$
           holds for all $p\,\in\,M$, $X\,\in\,T_pM$ and $U\,\in\,\{\,X,\,JX\,\}^\perp$.
  \end{enumerate}
  Here, $\K_0$ is the symmetrized curvature tensor defined by \eqref{eq:SCD} and 
  $\Psi^+\,:=\,\nabla J\,\circ\,J$.
 \end{proposition}

 \begin{proof}
  $(a)\Longrightarrow(b)$   is obvious. 
  For $(b)\Longrightarrow(a)$, see \cite[Lemma 5.1]{Tan} (or use
   \cite[Proposition 3.4]{G2} in combination with \eqref{eq:CT}).)
The equivalence of $(b)$ and $(c)$ is shown 
 in \cite[Theorem~3.4]{Tan}. Also the implication $(a)\Longrightarrow (d)$ is obvious.
 It remains to show that $(d)$ implies $(b)$: For this we can assume that $\scal\,=\,30$ 
 and hence the nearly K\"ahler constant from \eqref{eq:CT} satisfies $c\,=\,1$. 
 For every tangent unit vector $X$, set $T_1\,:=\,T_1(\,X\,)\,:=\,\{X,JX\}^\perp$. Then, $T_1$ is a 
 Hermitian subspace of $(\,M,\,g,\,J\,)$. In addition, $I_X\,:=\,\,\Psi^+_X$ defines a second Hermitian structure 
 on $T_1$ by virtue of the constant type equation \eqref{eq:CT}. 
 Since $\K_0(\,X\,)$ is a symmetric endomorphism whereas $I_X$ is skew,
 $\langle\, [\,\K_0(\,X\,),\,I_X\,]\,U,\,U\,\rangle\,
 =\, 2\, \langle\, \K_0(X)\,U,\,I_X\, U\,\rangle$.  Thus, assertion $(d)$ implies that  the projection of 
 $\K_0(\,X\,)$ to $T_1$ commutes with $I_X$:  
 $\langle\, [\,\K_0(\,X\,),\,I_X\,]\,U,\,U\,\rangle\,
 =\,0$. We also have $J \circ I_X\,=\,-\,I_{J\,X}\,=\,-I_X \circ J$ 
 because $\Psi^+$ is of type $(3,0)\,+\,(0,3)$. 
 Therefore, $T_1$ is on the one hand an irreducible module over 
 the  quaternionic numbers $\H$. On the other hand, $ \langle\;\K_0(\;X\;)\,J\,U,\;J\,U\;\rangle \,
 =\,  \langle\;\K_0(\;J\,X\;)\;U,\;U\;\rangle$ by a known curvature identity for nearly K\"ahler manifolds, cf. \cite[Corollary~2.2]{G3}. As a result,
 \begin{eqnarray*}
  \langle\;\K_0(\;X\;)\,J\,U,\;J\,U\;\rangle 
  &\;\;=\;\;&
 \langle\;\K_0(\;J\,X\;)\,U,\;U\;\rangle\\
  &\;\;=\;\;&
  \langle\;\K_0(\;J\,X\;)\,I_{J\,X}U,\;I_{J\,X}\,U\;\rangle
\\
   &\;\;=\;\;&
 \langle\;\K_0(\;J\,X\;)\,J\,I_XU,\;J\,I_X\,U\;\rangle
\\
   &\;\;=\;\;&
   \langle\,\K_0(\;X\;)\;I_XU,\;I_XU\;\rangle
  \\
  & \;\;=\;\;&
   \langle\,\K_0(\,X\,)\;U,\;U\;\rangle
 \end{eqnarray*}
 for every $U\,\in\,T_1$. Thus, also $J$ commutes with $\K_0(\,X\,)$ projected to $T_1$. 
 Therefore, $\K_0(\,X\,)$ projected to $T_1$ is an $\H$--homomorphism
 and hence a multiple of the identity due to
 Schur's Lemma. Accordingly, there exists a real constant $\mu\,=\,\mu(\,X\,)$ such
 that $\langle\, R(\,U,\,X,\,X\,),\,U\,\rangle\,=\,\mu(\,X\,)$ for every unit vector
 $U\,\in\,T_1(\,X\,)$. We claim that $\mu$ does not depend on $X$:
 Because $J$ is orthogonal, "$U\,\in\, T_1(\,X\,)$" is logically
 equivalent to "$X\,\in\, T_1(\,U\,)$"; in consequence, we find
 $$
  \mu(\,X\,)\;
  \;\;=\;\;
  \mu(\,U\,)\;
 $$
 for all unit vectors $X,\,U$ such that $X\in T_1(\,U\,)$. Furthermore, given unit vectors $X$ and $Y$
 different from zero, the dimension of $\mathrm{span}_\R\{\,X,\,Y,\,JX,
 \,JY\,\}$ is at most four, and so there exists a unit vector $U$ different from zero
 with $\{\,X,\,Y\,\}\, \subset\, T_1(\,U\,)$. 
 Therefore, $\mu(\,X\,) \,=\, \mu(\,U\,) \,=\, \mu(\,Y\,)$
 i.e., $\mu$ does actually not depend on $X$. Calculating the Ricci tensor, we see that 
 $$
  \Ric(\;X,\;X\;)
  \;\;=\;\;
  4\,\mu\;+\;H(\,X\,)
 $$
 On the other hand, $(\,M,\,g\,)$ is Einstein, hence $\Ric(X,X)\,=\,5\,$. This implies
 that $H(\,X\,)\,\equiv\,5\,-\,4\,\mu$ and so $H\,\equiv\,\mu\,\equiv\,1$.
 \end{proof}
\section{Naturally Reductive Nearly Parallel $\G_2$--Spaces}
\label{se:NR_NP}
 Following \cite{FKMS} and \cite[Chpt. 5.2]{BG1}, a $\G_2$--structure on an oriented Riemannian
 manifold $(\,M^7,\,g\,)$ is a reduction of the $\SO(\,7\,)$--principal 
 bundle of positive orthonormal frames to the subgroup $\G_2\,\subset\,\SO(\,7\,)$.
 This reduction is described by a $3$--form $\sigma$ which locally admits
 a coordinate representation \eqref{eq:sigma} with reference to a positive
 orthonormal frame field $\{E_1,\,\ldots,\,E_7\}$. In particular, $\sigma_p$ is a 
 vector cross product for each $p\in M$. Then, $M$ is equipped with a 
 canonical spin structure since $\G_2$ is simply
 connected. Moreover, the spin--representation of $\G_2$ fixes a non--trivial
 vector. Hence, there exists a distinguished spinor of unit length, cf.~\cite[Chpt. 2]{FKMS}.

 \begin{remark}\label{re:np}
  The group of automorphisms $\Aut(\,M,\,\sigma\,)\,:=\,\{\,f\,
  \colon\,M \widetilde \longrightarrow M,\, f^*\sigma\,=\,\sigma\,\}$ 
  is a subgroup of $\Isom(\,M,\,g\,)$ since $\G_2\,\subset\,\O(\,7\,)$.
  However, we do not automatically obtain a reduction of the 
  holonomy group of $(\,M,\,g\,)$ since $\sigma$ is not necessarily 
  parallel with respect to the Levi--Civita connection.
 \end{remark}

 \noindent
 The  analogue of a $6$--dimensional strict nearly K\"ahler manifold
 is a nearly parallel $\G_2$--space:

 \begin{definition}\label{de:NP}
  A $\G_2$--structure $\sigma$ is called nearly parallel
  if there exists a constant $\tau_\circ\,\neq\,0$ such that for all $X\,\in\,T_pM$:
  \begin{equation}\label{eq:NP}
   \nabla_X\sigma
   \;\;=\;\;
   \frac{\tau_\circ}4\;X\,\lrcorner\,*\sigma
  \end{equation}
 \end{definition}
 \noindent
 For other equivalent definitions of a nearly parallel $\G_2$--structure,
 cf.~\cite[Proposition~2.3]{AlS}. Then, the characteristic spinor
 of the $\G_2$--reduction is a Killing spinor for the Killing number $-\,\frac{\tau_0}{8}$.
 Therefore, every nearly parallel $\G_2$--space is an Einstein manifold with scalar curvature
 $\scal\,=\,\frac{21}{8}\,\tau_\circ^2$, cf. \cite[Chpt. 1.5]{BFGK} and \cite[Chpt. 3]{FKMS}. In the standard
 normalization $\scal\, = \, \frac{21}{8}$ we have $\tau_\circ^2\, =\, 1$.
 In particular, every complete nearly parallel $\G_2$--space is compact with
 finite fundamental group. It is also known that the holonomy group of the Levi Civita connection 
 is generic, cf.~\cite[p.~11]{FKMS}. Therefore, $(\,M,\,g\,)$ is irreducible; also it is a symmetric space if
 and only if it is isometric to the round $7$--sphere. As a result, Theorem~\ref{th:OR} implies 
 that a naturally reductive structure $\bar \nabla$ on a compact nearly parallel 
 $\G_2$--space is unique unless the space is of constant sectional curvature. It turns out that
 the additional assumption $\bar\nabla\sigma\,=\,0$ is even more restrictive.
 Recall that the canonical connection of a nearly parallel $\G_2$--space is defined by 
 \begin{equation}\label{eq:IG2C}
  \nabla^c\,:=\,\nabla\,-\,
  \frac{\tau_\circ}{12}\,\sigma
 \end{equation}
 where $\tau_\circ$ is the nearly parallel constant from~\eqref{eq:NP}.
 It is known that $\nabla^c\sigma \,=\,0$, cf.~\cite{AlS}. 
 Therefore, $\nabla^c$ plays the same role as the canonical 
 connection \eqref{eq:IHC} of an almost Hermitian manifold. 

 \begin{theorem}[\cite{FI}]
 \label{th:FI}
  Let both a nearly parallel $\G_2$--structure $\sigma$ and a naturally reductive
  structure $\bar\nabla$ on a Riemannian manifold $(\,M,\,g\,)$ be given. If
  $\bar\nabla\sigma\,=\,0$, then $\bar\nabla\,=\,\nabla^c$.
 \end{theorem}
 
 \begin{proof}
 According to
 \cite[Corollary~4.3]{FI} there exists at most one affine connection 
 $\bar\nabla$ with skew torsion such that $\bar\nabla\sigma\,=\,0$.
 The canonical connection $\nabla^c$ has these properties, too, and so 
 $\bar\nabla\,=\,\nabla^c$.
 \end{proof}
  
 \bigskip
 \noindent
 The following is the analogue of Definition \ref{de:NR_NK}. 

 \begin{definition}
 \label{de:NR_NP}
  We will call a compact nearly parallel $\G_2$--space $(\,M,\,g,\,\sigma\,)$
  naturally reductive if the canonical  connection $\nabla^c$  defined in \eqref{eq:IG2C} is a 
  naturally reductive structure of $(\,M,\,g\,)$. If $\nabla^c$  is actually (standard) normal homogeneous, 
  then $(\,M,\,g,\,\sigma\,)$ will be called a (standard) normal homogeneous nearly parallel $\G_2$--space.
 \end{definition}

\noindent
 If $(\,M,\,g,\,\sigma\,)$ is naturally reductive, then $\Tr(\,\nabla^c\,)$ 
 is contained in $\Aut(\,M,\,\sigma\,)$ and acts transitively on $M$, 
 i.e $(\,M,\,g,\,\sigma\,)$ is a homogeneous nearly parallel $\G_2$--space  
 in the sense of~\cite{FKMS}. However, the classification of these spaces 
 given in \cite{FKMS} does not follow such a simple pattern as that of 
 six--dimensional homogeneous strict nearly K\"ahler manifolds: First of all, 
 it is essential to assume that the space is simply connected since the classification
 via Killing spinors relies on this fact. Moreover,
 in Section \ref{se:CLASS_G2}, we will see that there are
 simply connected homogeneous nearly parallel $\G_2$--spaces, 
 the metric tensors of which are naturally reductive, but whose nearly 
 parallel $\G_2$--structures are not naturally reductive in the sense of
 Definition  \ref{de:NR_NP}. There exist even simply connected 
 homogeneous nearly parallel $\G_2$--spaces whose metric tensors are not
 naturally reductive. In order to understand such subtleties the 
 following proposition will turn out to be helpful.

 \begin{proposition}
 \label{p:NR_NP}
  \begin{enumerate}
   \item  Let both a nearly parallel $\G_2$--structure $\sigma$ and a naturally reductive
            structure $\bar\nabla$ on a compact Riemannian manifold $(\,M,\,g\,)$ be given.
            If $\Tr(\,\bar\nabla\,)\,\subset\,\Aut(\,M,\,\sigma\,)$, 
            then $\bar\nabla\,=\,\nabla^c$. In particular, then $(\,M,\,g,\,\sigma\,)$ is naturally reductive.
   \item For a simply connected naturally reductive
            nearly parallel $\G_2$--space $(\,M,\,g,\,\sigma\,)$ 
            which is not isometric to $(\,\S^7,\,g_{\mathrm{round}}\,)$,
            we have $\Isom(\,M,\,g\,)\,=\, \Aut(\,M,\,\sigma\,)$.
  \end{enumerate} 
 \end{proposition}
 
 \begin{proof}
  If $\Tr(\,\bar\nabla\,) \,\subset\,\Aut(\,M,\,\sigma\,)$, then $\bar\nabla\sigma\,=\,0$
  because $\Tr(\,\bar\nabla\,)$--invariant tensors are automatically $\bar\nabla$--parallel.
  Therefore, $\bar \nabla\,=\,\nabla^c$ according to Theorem~\ref{th:FI}. 
  This proves the first statement. For the second one, Theorem~\ref{th:OR} implies that
  $\Isom(\,M,\,g\,)\,\subset\,\Aff(\,M,\,\nabla^c\,)$; therefore,
  $\Isom(\,M,\,g\,)\,\subset\,\Aff(\,M,\,\nabla\,)\,\cap\,
  \Aff(\,M,\,\nabla^c\,)\,\subset\,\Aut(\,M,\,\sigma)$, where the 
  last inclusion uses that $\nabla^c\,-\,\nabla\,=\,-\,\frac{\tau_\circ}{12}\,\sigma$.
  In particular, we conclude that $\Isom(\,M,\,g\,)\,\subset\,\Aut(\,M,\,\sigma)$;
  thus, $\Aut(\,M,\,\sigma\,)\,=\,\Isom(\,M,\,g\,)$ 
  since the opposite inclusion is always satisfied.
 \end{proof}

 \bigskip\noindent
 By definition, the canonical connection of a naturally reductive nearly parallel 
 $\G_2$--space $(\,M,\,g,\,\sigma\,)$  is a naturally reductive structure whose 
 torsion form is a non--zero multiple of $\sigma$. Conversely we have:

 \begin{lemma}\cite[Lemma 7.1]{AlS}
 \label{le:AS}
  Let a $7$--dimensional naturally reductive  space  $(\,M,\,g,\,\bar\nabla\,)$ 
  with torsion form $\tau$  be given. If there exists a non--zero constant 
  $c$ such that $\sigma\,:=\, \frac1c\,\tau$ is a $\G_2$--structure, 
  then $\sigma$ is automatically nearly parallel and $\bar\nabla\,=\,\nabla^c$.  
 \end{lemma}

 \noindent
 We obtain the following analogue of Proposition~\ref{p:NK}:

 \begin{corollary}\label{co:NR_NP}
  The assignment $(\,M,\,g,\,\sigma\,)\longmapsto\,(\,M,\,g,\,\nabla^c\,:=
 \,\nabla\,-\,\frac{\tau_\circ}{12}\sigma\,)$ establishes a 1--1 correspondence between:
  \begin{itemize}
   \item  naturally reductive nearly parallel $\G_2$--spaces
            $(\,M,\,g,\,\sigma\,)$ in the sense of Definition \ref{de:NR_NP};
   \item  $7$--dimensional oriented naturally reductive spaces
            $(\,M,\,g,\,\bar\nabla\,)$ whose torsion form $\tau$
            is a non--zero multiple of a vector cross product.
  \end{itemize} 
 \end{corollary}
 \begin{proof}
  It is clear that the assignment is well defined
  and injective. It is also surjective according to Lemma \ref{le:AS}.
 \end{proof}

 \bigskip
 \noindent
  It remains to classify naturally reductive structures on nearly parallel $\G_2$--spaces. 

 \begin{remark}\label{re:NR_NP}
  Recall that ~\cite[Eq.~(7.87\,b)]{Bes} implies that normal homogeneous spaces  
  have non--negative sectional curvature. Refining this argument, one sees that 
  the sectional curvature of a normal homogeneous nearly parallel $\G_2$--space 
  $(\,M,\,g,\,\sigma\,)$ in the sense of Definition \ref{de:NR_NP} is positive. 
  Therefore, $(\,M,\,g,\,\nabla^c\,)$ is a $7$--dimensional normal homogeneous space
  of positive sectional curvature. For sake of simplicity, let us assume that $(\,M,\,g,\,\nabla^c\,)$ is even 
  one of the  $7$--dimensional  simply connected standard normal homogeneous spaces 
  of positive sectional curvature mentioned in Section \ref{se:ingredients}.
  Then, the results from \cite{FKMS} imply at once that $(\,M,\,g,\,\nabla^c\,)$ 
  is equal to $\S^7\,\cong\,\Spin(\,7\,)/\G_2$, $\S^7_{\mathrm{squashed}}$, $V_1$ or $V_3$, cf. \cite[Chpt. 8]{AlS}.
 \end{remark}

\subsection{Discussion of the classification of 
 Homogeneous Nearly Parallel $\G_2$--Spaces from \cite{FKMS}}
\label{se:CLASS_G2}

 In this section, we will decide which compact simply connected nearly parallel 
 $\G_2$--spaces carry naturally reductive metrics and which, by Definition 
 \ref{de:NR_NP}, are even naturally reductive. In Theorem \ref{th:CLASS_G2} we give a 
 summary of our analysis. Together with the results of  Section \ref{se:NR_NK} this provides
 also the proof of Theorem  \ref{th:GVCP}.  
 
 \bigskip
 \noindent
 The classification of compact simply connected nearly parallel $\G_2$--spaces
 which admit a transitive action by  $\G_2$--automorphisms is given in~\cite{FKMS}. 
 More precisely, in this  article the reader may find the list  of all effective homogeneous 
 pairs $(\,G,\,H\,)$ where $G$ acts transitively on some compact simply connected 
 nearly parallel $\G_2$--space $(\,M,\,g, \,\sigma\,)$ such that $M$ results in the 
 quotient $G/H$. Recall that $M$ is automatically spin and that any two spin structures 
 are isomorphic since $M$ is simply connected \cite{LM}.  In fact, there are precisely 
 two isomorphisms between two given spin structures  (similar to isomorphic $2$--fold 
 coverings of one and the same topological space). Let us choose one spin structure as a fixed 
 reference, and denote by $KS(\,M,\,g\,)$ the linear space of Killing spinors. As mentioned before, 
 $\sigma$ defines a second spin structure together with a unit Killing spinor. Because there 
 are two possibilities to identify the two spinor bundles with each other, this Killing spinor is 
 actually only well defined up to a sign in $KS(\,M,\,g\,)$, i.e., $\sigma$ defines a projective class 
 in $KS(\,M,\,g\,)$. Conversely, with the ideas hinted at in Section \ref{se:GVCP}, every projective 
 class in $KS(\,M,\,g\,)$ defines a nearly parallel $\G_2$--structure $\sigma$, see also \cite[Eq. (\,2\,)]{FKMS}. 
 In this way the projective space $\Proj KS(\,M,\,g\,)$  of projective 
 classes of Killing spinors (of a fixed spin structure) can be identified with the set of all nearly parallel 
 $\G_2$--structures $\sigma$ on $(\,M,\,g\,)$. For details of this argument, see \cite[Chpt. 2 \& 3]{FKMS}. 
 As a result, the group of orientation preserving isometries acts naturally by pullback on 
 $\sigma\,\in\, \Proj KS(\,M,\,g\,)$, and the isotropy subgroup of $\sigma$ is given by $\Aut(\,M,\,\sigma\,)$.

\bigskip 
 \noindent 
 The dimension of $KS(\,M,\,g\,)$  is called the type, cf.~\cite[Chpt. 4]{FKMS}. 
 If $(\,M,\,g\,)$ is different from $(\,\S^7,\,g_{\mathrm{round}}\,)$,
 then this  number is at most three. Simply connected compact  nearly parallel $\G_2$--spaces 
 of type $3$ are $3$--Sasakian; conversely,
 every $7$--dimensional simply connected  compact $3$--Sasakian manifold
 different from $\S^7_{\mathrm{round}}$ is nearly parallel $\G_2$ of type $3$,
 cf. \cite[Chpt. 6]{FK} and \cite[Theorem 5.2.6]{BG1}. Similarly,
 simply connected compact nearly parallel  $\G_2$--spaces of type 2 are 
 Sasaki--Einstein, but not $3$--Sasakian, and every $7$--dimensional  simply connected  
 compact strict Sasaki--Einstein manifold $(\,M,\,g\,)$ is nearly parallel $\G_2$ of type 2 
 according to \cite[Chpt. 5]{FK}. Also recall that compact simply connected 
 nearly  parallel $\G_2$--spaces of type $1$ are called proper. 
 Here, we have $\Proj KS(\,M,\,g\,)\,=\,\{\sigma\}$, i.e., $\sigma$ is the unique nearly parallel $\G_2$--structure. 

 \begin{lemma}\label{le:proper}
  For every simply connected homogeneous nearly parallel $\G_2$--space $(\,M,\,g,\,\sigma\,)$,
  the connected component  $\Isom(\,M,\,g\,)_\circ$ of the isometry group acts transitively on $\Proj KS(\,M,\,g\,)$.
 \end{lemma}
 \begin{proof}  
  For the round sphere $\S^7$ this follows from \cite{Fried}. In all other cases, although the assertion of our lemma
  is not explicitly mentioned in \cite{FKMS}, it is nevertheless an obvious consequence of the results proved in that paper:
  The groups $\Isom(\,M,\,g\,)_\circ$ are well known (see \cite[Chpt. 4]{FKMS}), and the possible automorphism groups 
  of dimension at least ten are precisely determined in \cite[Chpt. 7]{FKMS}. It follows that the codimension of 
  $\Aut(\,M,\,\sigma\,)_\circ$ in $\Isom(\,M,\,g\,)_\circ$ is always the maximal possible
  $\dim\,\Proj KS(\,M,\,g\,)\,=\,\dim\, KS(\,M,\,g\,)\,-\,1$ for all $\sigma\,\in\,\Proj KS(\,M,\,g\,)$.
  Therefore, $\Isom(\,M,\,g\,)_\circ$  acts transitively on $\Proj KS(\,M,\,g\,)$. 
  We will explicitly describe the groups $\Aut(\,M,\,\sigma\,)_\circ$ and $\Isom(\,M,\,g\,)_\circ$ in the following paragraphs.
 \end{proof}

 \begin{remark}\label{re:3--Sasakian}
  Let $(\,M,\,g\,)$ be a simply connected compact $3$--Sasakian manifold of dimension seven.
  We also assume that $(\,M,\,g\,)$ is not isometric to the round sphere.
  Let $\{\xi_1,\,\xi_2,\,\xi_3\}$ be a characteristic triple  of Sasakian vector fields 
  with metric duals $\eta_i\,=\,\xi_i^\flat$. According to \cite[Theorem 6.2]{AF} the $3$--form
  \begin{equation}
  \label{eq:AF}
   \sigma\,:=\,\frac12\;\eta_1\,\wedge\,\d\,\eta_1\,-\,\frac12\;\eta_2\,\wedge\,\d\,\eta_2\,-\,\frac12\;\eta_3\,\wedge\,\d\,\eta_3
  \end{equation}
  belongs to $\Proj KS(\,M,\,g\,)$ and every element of $\Proj KS(\,M,\,g\,)$ has this form.
  Since the characteristic factor of $\Isom(\,M,\,g\,)_\circ$ (which is isomorphic to $\SO(\,3\,)$ or $\Sp(\,1\,)$,
  cf. \cite[Chpt. 3.2]{BG1}) acts transitively on the set of characteristic triples, it follows at once that it also 
  acts transitively on $\Proj KS(\,M,\,g\,)$. In contrast, for a simply connected compact strict 
  Einstein Sasakian manifold there does seemingly not exist such an explicit description of the projective line
  $\Proj KS(\,M,\,g\,)$. It is also not known by the  authors whether in this case the characteristic 
  $\T^1$--factor of the isometry group always acts  transitively on $\Proj KS(\,M,\,g\,)$. (In the homogeneous case it does so.) 
 \end{remark}

 \begin{corollary}\label{co:proper}
  Let $(\,M,\,g,\,\sigma\,)$ be a simply connected compact  nearly parallel $\G_2$--space
  which is not isometric to $(\,\S^7,\,g_{\mathrm{round}}\,)$.
  Then, $(\,M,\,g,\,\sigma\,)$ is naturally reductive in the sense of Definition~\ref{de:NR_NP}
  if and only if $\Proj KS(\,M,\,g\,)\,=\,\{\sigma\}$ (i.e., $(\,M,\,g\,)$  is proper) 
  and $g$ is a naturally reductive metric.
 \end{corollary}
 \begin{proof}
  If $(\,M,\,g,\,\sigma\,)$ is naturally reductive, then it is a homogeneous nearly parallel $\G_2$--space 
  and $g$ is a naturally reductive metric. Also the second part of Proposition \ref{p:NR_NP}
  shows that $\Aut(\,M,\,\sigma\,)_\circ\,=\,\Isom(\,M,\,g\,)_\circ$. Therefore, $\Proj KS(\,M,\,g\,)\,=\,\{\sigma\}$ 
  by Lemma \ref{le:proper}. Conversely, if $\Proj KS(\,M,\,g\,)\,=\,\{\sigma\}$, then $\Aut(\,M,\,\sigma\,)$ coincides necessarily with 
  the group of orientation preserving isometries.  In particular, $\Aut(\,M,\,\sigma\,)_\circ\,=\,\Isom(\,M,\,g\,)_\circ$. 
  If there exists furthermore a naturally reductive structure $\bar\nabla$, then 
  $\Tr(\,\bar\nabla\,)\,\subset\, \Isom(\,M,\,g\,)_\circ\,=\,\Aut(\,M,\,\sigma\,)_\circ$. 
  Hence, the first part of Proposition \ref{p:NR_NP} shows that $(\,M,\,g,\,\sigma\,)$ is naturally reductive.
 \end{proof}
 
\paragraph{The round sphere $\S^7$}

 \begin{proposition}\label{p:S7}
 Every nearly parallel $\G_2$--structure on $\S^7_{\mathrm{round}}$ is standard normal homogeneous
 in the sense of Definition~\ref{de:NR_NP}. More precisely, its canonical connection is conjugated by an element of 
 $\SO(\,8\,)$ to the normal homogeneous structure of $\Spin(\,7\,)/\G_2\,$.
\end{proposition}
\begin{proof}
  It is known that the unique $\Spin(\,7\,)$--invariant $\G_2$--structure $\sigma$ on 
  $\S^7_{\mathrm{round}}$
  is nearly parallel, cf. \cite[Remark 5.2.4]{BG1}. 
  The first part of Proposition \ref{p:NR_NP} implies that $(\,\S^7_{\mathrm{round}},\,\sigma\,)$
  is the normal homogeneous nearly parallel $\G_2$--space $\Spin(\,7\,)/\G_2\,$. 
  Furthermore, every other nearly parallel $\G_2$--structure of $\S^7_{\mathrm round}$
  is conjugated to $\sigma$ by an element of $\SO(\,8\,)$ \cite{Fried}. 
  The same element of $\SO(\,8\,)$ also conjugates the canonical connections. The result follows. 
\end{proof}

\paragraph{\bf Homogeneous Nearly Parallel $\G_2$--Spaces of Type 1}
 We consider the three standard normal homogeneous spaces of positive sectional curvature 
 $\S^7_{\mathrm{squashed}}$, $V_1$ and $V_3$ defined after Theorem \ref{th:LJR}. 
 The standard normal metrics of $\S^7_{\mathrm{squashed}}$ and $V_3$ 
 are the second Einstein metrics in the canonical variation of the $3$--Sasakian metrics of
 $\S^7_{\mathrm{round}}\,=\,\Sp(\,2\,)\,/\,\Sp(\,1\,)$ and the Aloff-Wallach space 
 $\NM(\,1,\,1\,)\,=\,\SU(\,3\,)\,/\,\U(\,1\,)$ (not the normal Wilking space), respectively, cf. 
 \cite[Chpt. 2.4]{BG1}. It is known that the second Einstein metric is proper nearly parallel 
 $\G_2$, cf. \cite[Chpt. 5]{FKMS} and \cite[Remark 5.2.4]{BG1}. Also $V_1$ carries a proper 
 nearly parallel $\G_2$--structure, cf. \cite[Chpt. 5]{Bry}.  By properness, all three spaces  
 automatically satisfy $\Aut(\,M,\,\sigma\,)_\circ\,=\, \Isom(\,M,\,g\,)_\circ$ 
 (the connected isometry groups can be found in \cite{WZ}).
 As a consequence of Corollary \ref{co:proper}, we see:

 \begin{proposition}
  \label{p:EX_NR_NP}
  The nearly parallel $\G_2$--structures of $\S^7_{\mathrm{squashed}}$, $V_1$ and
  $V_3$ are standard normal homogeneous in the sense of Definition~\ref{de:NR_NP}.
 \end{proposition}
 
 \noindent
 The fact that the torsion forms of $\S^7_{\mathrm{squashed}}$, $V_1$ and $V_3$ satisfy 
 the conditions of Lemma \ref{le:AS} is also mentioned in \cite[Chpt. 8]{AlS}.

 \bigskip
 \noindent
 The Aloff--Wallach spaces $\NM(\,k,\ell\,)$ are defined as follows (cf. \cite{WZ}):
 Let $G\,:=\,\SU(\,3\,)$, $K\,:=\,\U(\,1\,)\times \U(\,1\,)$. 
 Then, $\bbF^3\,=\,G\,/\,K\,$ is the complex flag manifold. Every pair
 of non--zero integers $k,\,\ell$ defines a connected $1$--dimensional
 (normal) subgroup $H\,\subset\,K$ via the embedding 
 $$
  \iota\;:\qquad
  \R  \;\longrightarrow\;\su(3),\qquad t\;\longmapsto\;
  \diag(\;i\,k\,t\;,\;i\,\ell\,t\;,\;-\,i\,(\;k\;+\;\ell\;)\,t\;) 
 $$
 Without loss of generality one can also assume that $k\,\geq\,\ell\,\geq\,1$ and that
 $\mathrm{gcd}(\,k,\,\ell\,)\,=\,1$. The homogeneous space $\NM(\,k,\ell\,) \,:=\,G/H$ 
 is called an Aloff--Wallach space. For $k\,>\,\ell\,\geq\,1$ 
 there exist two different homogeneous proper nearly parallel 
 $\G_2$--structures $\sigma_i$, see~\cite[Theorem 3.2]{CMS}. 
 Since $\sigma_i$ is proper, it follows that 
 $\Aut(\,\NM(\,k,\ell\,),\,\sigma_i)_\circ\,=\,\Isom(\,\NM(\,k,\ell\,),\,g_i\,)_\circ$
 where $g_i$ denotes the underlying nearly  parallel $\G_2$--metric for $i\,=\,1,\,2$. 
 Furthermore, we have  $\Isom(\,\NM(\,k,\ell\,),\,g_i\,)_\circ\,
 = \,\SU(\,3\,)\times\T^1 / C$ with $\T^1\,:=\,K\,/\,H$ and $C\in \{\{e\},\Z_3\}$, cf. \cite {FKMS,WZ}.
 It is also known that the $g_i$ are the only 
 homogeneous Einstein metrics on  $\NM(\,k,\,\ell\,)$, cf.~\cite{Nik}. 
 In particular, the following proposition shows that there do 
 not exist naturally reductive Einstein metrics on
 Aloff--Wallach spaces different from $\NM(\,1,\,1\,)$.

 \begin{proposition}\label{p:AW}
  The nearly parallel $\G_2$--metrics $g_i$ of $\NM(\,k,\,\ell\,)$ are not naturally reductive 
  for $k\,>\,\ell\,\geq\,1$.
 \end{proposition}

 \begin{proof}
 Let us assume for a contradiction that $g\,:=\,g_i$ is naturally reductive. 
 Since $\Isom(\,\NM(\,k,\ell\,),\,g\,)_\circ$ is covered by  $\SU(\,3\,)\times\T^1$, 
 the corresponding transvection group is covered by $\SU(\,3\,)$ or $\,\SU(\,3\,)\times\T^1$.
 In the first case, $g$ is the naturally reductive metric induced on $G\,/\,H$ by a negative multiple 
 $B\,=\,c\,\rmK_{\su(\,3\,)}$ of the Killing form of $\su(\,3\,)$. In the second case, we have to use the 
 description $\NM(\,k,\,\ell\,)\,=\,\hat G\,/\,\hat K$ where $\hat G\,:=\,\SU(\,3\,)\,\times\,\T^1$ and 
 $\hat K$ is the subgroup of $\hat G$ given by $\{\,(\,k,\,k\,H\,)\,|\,k\,\in\,K\,\}$. 
 Then, there exists some non--vanishing symmetric bilinear form $\tilde B$ 
 on the Lie algebra $\t^1$ of the factor $\T^1$ such that $g$ is the naturally reductive 
 metric induced by $\hat B\,:=\,B\,\oplus\, \tilde B$ on $\hat G\,/\,\hat K$.
 The calculations from~\cite[Chpt. 4.5]{BFGK} show that both cases are not possible.
 We will not carry out the details of this argument since there is also a second one: 
 Recall that the Aloff Wallach spaces are constructed from $\bbF^3$ by the general principle 
 described in Section \ref{se:NR_FB}. Hence
 $B\,=\,\hat B_0$ and $\hat B$ both belong to the family $\hat B_s$ of 
 symmetric invariant bilinear forms on $\su(\,3\,)$ or $\su(\,3\,)\oplus\,\t^1$
 parameterized by $s\,>\,-\,1$ considered in Lemma \ref{le:Nor_SG}. 
 The torsion form $\hat \tau$ associated with $\hat B_s$ 
 is calculated in Lemma \ref{le:T_and_R}, see \eqref{eq:HAT_TAU}. 
 On the other hand, because $(\,\NM(\,k,\,\ell\,),\,g)$ is proper, the latter is 
 at the same time a non--zero multiple of a vector cross product  according to 
 Corollary~\ref{co:proper}. It is straightforward to show that these two conditions are 
 contrary to each other, see Proposition~\ref{p:semi-simple} in Section~\ref{se:SP}.
 Effectively, this provides two independent arguments why neither of the two
 homogeneous nearly parallel $\G_2$--metrics on $\NM(\,k\,,\ell\,)$ is
 naturally reductive.
 \end{proof}

 \bigskip
 \noindent
 We have seen in the proof of Proposition \ref{p:AW} that there exists a $2$--parameter 
 family of naturally reductive metrics on $\NM(\,k\,,\ell\,)$, which are bundle--like 
 over $\bbF^3$, see also Tables 5 and 7 of \cite{St3}. 
 However, the two homogeneous nearly parallel $\G_2$--metrics 
 are not part of this family: they are not bundle--like over $\bbF^3$.

 \paragraph{\bf Homogeneous Nearly Parallel $\G_2$--Spaces of Type 2}
 Here, $(\,M,\,g\,)$ is a $7$--dimensional homogeneous simply connected 
 strict Sasaki--Einstein manifold (not $3$--Sasakian). Consequently, $M$
 is an $\S^1$--fiber bundle over a generalized flag manifold and $g$ is 
 bundle--like \cite [Theorem 3.1.2]{BG1} and \cite[Theorem 3.4]{BG2}. 
 Conversely, every generalized flag manifold is  the base space of a unique 
 $\S^1$--fiber bundle whose total space is a simply connected homogeneous 
 Sasaki--Einstein manifold \cite[Remark 3.1.4]{BG1}.
 The generalized flag manifolds in dimension six are: 
 the complex flag manifold $\bbF^3\,=\,\SU(\,3\,)/\U(\,1\,)\times \U(\,1\,)$ 
 (with the --  unique up to the action of the Weyl group -- homogeneous K\"ahler--Einstein metric \cite[Theorem 8.2]{Bes}, not the nearly K\"ahler one),
 and the following Einstein Hermitian symmetric spaces: $\CP^3$, $\CP^2\,\times\,\CP^1$,  $\CP^1\,\times\,\CP^1\,\times\,\CP^1$ 
 and $\Gr_+(\,5\,,\,2\,)$ (the real Grassmannian of oriented $2$--planes in $\R^5$),  cf. \cite[Corollary 3.1.3 (iii)]{BG1}. 
 From $\bbF^3$ and $\CP^3$, we obtain $3$--Sasakian manifolds via this construction.
 In order to describe the remaining three simply connected homogeneous strict 
 Sasaki--Einstein manifolds in an explicit way, we consider first, more general,
 the following $\S^1$--fiber bundles over  $\Gr_+(\,n\,,\,2\,)$, $\CP^2\,\times\,\CP^1$ and $\CP^1\,\times\,\CP^1\,\times\,\CP^1$,
 respectively (cf. \cite[p. 64]{DNP} and the references mentioned there):

 \begin{enumerate}
  \item Let $G\,:=\,\SO(\,n\,)$, $K\,:=\,\mathbf{S}(\,\O(\,n\,-\,2\,)\,\times\,\O(\,2\,)\,)$ and
        $H\,:=\,\SO(\,n\,-\,2\,)$. The quotients $G/H$ and $G/K$ are the Stiefel
        manifold $\VM(\,n,\,2\,)$ of orthonormal $2$--frames and the
        Grassmannian $\G_+(\,n,\,2\,)$ of oriented $2$--planes in $\R^n$, respectively.
        Hence $\VM(\,n,\,2\,)$ is a simply connected homogeneous $\S^1$--fiber bundle over 
        $\Gr_+(\,n,\,2\,)$ via the standard projection.
  \item Let $G\,:=\,\SU(\,3\,)\times\SU(\,2\,)$ and $K\,:=\,\U(\,2\,)\,\times\,\U(\,1\,)$.
          For any $(\,k\,,\,\ell\,)\,\in\,\Z_+^2$ we define a subalgebra
          $\frakh\,\subset\,\u(\,2\,)\,\oplus\,\u(\,1\,)$ by
          $$
           \frakh
           \;\;:=\;\;
           \left\{\left .
           \begin{pmatrix}
            A\,+\,i\,k\,t\,\id & 0 \cr
            0 & -2\,i\,k\,t
           \end{pmatrix}
           \;\oplus\;
           \begin{pmatrix}
           i\,\ell\,t\, & 0 \cr
           0 & -\,i\,\ell\,t
           \end{pmatrix}
           \;\;\right|\;\;A\,\in\,\su(2),\,t\,\in\,\R
           \right\}
          $$ 
          We denote by $\MM(\,k,\,l\,)$ the homogeneous space $G/H$
          where $G\,:=\,\SU(\,3\,)\,\times\,\SU(\,2\,)$ and $H$ 
          is the subgroup of $K$ whose Lie algebra
          is $\frakh$. Then, a finite covering space of $H$ is isomorphic to 
          $\SU(\,2\,)\,\times\,\U(\,1\,)$. Hence, $\MM(\,k,\,\ell\,)$ is a simply connected  homogeneous
          $\S^1$--fiber bundle over $G/K\,=\,\CP^2\times\CP^1$. If the two non--zero integers $k,\,\ell$
          are relatively prime, then our space $\MM(\,k,\,\ell\,)$ is isomorphic to Witten's $\MM^{\ell,\,k,\,1}$;
          in \cite[p.64]{DNP} and \cite{PP} the latter space is denoted by $\MM(\,\ell,\,k\,)$
          (with the two integers $k$ and $\ell$ interchanged). 
  \item Let $G\,:=\,\SU(\,2\,)\,\times\,\SU(\,2\,)\,\times\,\SU(\,2\,)$ and $K$
          be the standard maximal torus $\U(\,1\,)\,\times\,\U(\,1\,)\,\times\,\U(\,1\,)$. 
          By definition, the Lie algebra of $\u(\,1\,)$ is spanned by $i\,\sigma_3$
          where $\sigma_3$ denotes the third Pauli matrix.
          For every triple $(\,k\,,\,\ell\,,\,m\,)\,\in\,\Z_+^3$, we denote by $\frakh\,
          \subset\, \u(\,1\,)\,\oplus\,\u(\,1\,)\,\oplus\,\u(\,1\,)$ the orthogonal complement of
          $\frakh^\perp\,:=\,\{\,i\,k\,t\,\sigma_3\,\oplus\,i\,\ell\,t\,\sigma_3\,
          \oplus\,i\,m\,t\,\sigma_3\,\mid\,t\,\in\,\R\,\}$ with respect to the standard normal structure.
          We set $\QM(\,k,\,\ell,\,m\,)\,:=\,G/H$ where $H\,\cong\,\U(\,1\,)\,\times\,\U(\,1\,)$ is the subgroup
          of $K$ whose Lie algebra is $\frakh$. By construction,
          $\QM(\,k,\,\ell\,,\,m\,)$ is a simply connected  homogeneous $\S^1$--fiber bundle over
          $\CP^1\times\CP^1\times\CP^1$. If the three non--zero integers $k$, $\ell$ and $m$ are relatively prime,
          then our notation for the space $\QM(\,k,\,\ell\,,\,m\,)$ coincides with the one from \cite[p. 64]{DNP}.
 \end{enumerate}

 \noindent
 Since $(\,G,\,K\,)$ is a Hermitian symmetric pair, a suitable choice of a positive definite invariant bilinear form $B$ on 
 the Lie algebra $\frakg$ of $G$ equips $G\,/\,K$ with the structure of an Einstein Hermitian symmetric space. 
 Also $H$ is a closed normal subgroup of $K$ and $K/H$ is a $1$--dimensional torus $\T^1$.
 According to Lemma~\ref{le:Nor_SG} we obtain a family of naturally reductive structures on $G\,/\,H$ parameterized by $s\,>\,-\,1$ and invariant under 
 $\hat G\,:=\,G\,\times\,\T^1$.

 \bigskip
 \noindent
 In the following, let $M\,:=\,G\,/\,H$ be one of the $7$--dimensional homogeneous spaces
 $\VM(\,5,\,2\,)$, $\MM(\,2,\,3\,)$ and $\QM(\,1,\,1,\,1\,)$ defined before.
 Then, $M$ is simply connected and there exists a homogeneous Sasaki--Einstein
 metric $g$ which is bundle--like over the Hermitian symmetric space $G\,/\,K$.
 Therefore, we have $\dim\, KS (\,M,\,g\,)\,=\,2$, and $(\,M,\,g,\,\sigma\,)$ is a homogeneous nearly parallel $\G_2$--space
 where $\sigma$ is an arbitrary element of the projective line $\Proj KS (\,M,\,g\,)$ of nearly parallel $\G_2$--structures cf.~\cite{FKMS}
 (recall that our space $\MM(\,2,\,3\,)$ is the space $\MM(\,3,\,2\,)$ from \cite{FKMS}) .

 Also there are no other homogeneous examples of type 2. Furthermore, we have  $\Isom(\,M,\,g\,)_\circ\,=\,G\,/\,C\,\times\,\T^1$ where $C$ denotes the center of $G$. 
 In accordance with Lemma \ref{le:proper}, it follows that $\Aut(\,M,\,\sigma\,)_\circ\,=\,G\,/\,C$.

 \begin{proposition}\label{p:NR_T2}
  A simply connected homogeneous nearly parallel $\G_2$--space $(\,M\, := \,G\,/\,H,\,g,\,\sigma\,)$ of type 2
  -- where $\sigma\,\in\,\Proj KS\, (\,M,\,g\,)$ is arbitrarily chosen -- is not naturally reductive in 
  the sense of Definition~\ref{de:NR_NP}. Nevertheless, the underlying nearly parallel $\G_2$--metric 
  $g$ is naturally reductive with transvection group $G\,/\,C\,\times\,\T^1$.
 \end{proposition}
 \begin{proof}
  Here, the canonical $\G_2$--connection is not a naturally reductive structure
  by Corollary~\ref{co:proper}. Since $\Aut(\,M,\,\sigma\,)_0\,=\,G$,  this implies that the underlying nearly 
  parallel $\G_2$--metric $g$ is not the normal metric of $G/H$ induced by $B$\,: otherwise 
  $(\,M,\,g,\,\sigma\,)$ would be a normal homogeneous
  nearly parallel $\G_2$--space according to Proposition \ref{p:NR_NP}. 
  Because $g$ is $\hat G$--invariant and bundle--like, and $K$ acts without trivial factor on $T_pM$, 
  $g$ belongs to the family of naturally reductive metrics parameterized by $s\,>\,-\,1$, $s\,\neq\,0$ 
  constructed in Lemma~\ref{le:Nor_SG}.  Hence, the transvection group is $G\,/\,C\,\times\,\T^1$. 
 \end{proof}
 
 \bigskip
 \noindent
 The three homogeneous nearly parallel $\G_2$--spaces considered above 
 are naturally reductive $\S^1$--fiber bundles over symmetric spaces 
 and their torsion forms are given by \eqref{eq:HAT_TAU}.
 Hence, $\tau$ vanishes on the horizontal distribution in the sense $\tau|_{\scrH\times\scrH\times\scrH}\,=\,0$ 
 according to Lemma~\ref{le:T_and_R}. Since $\dim(\,\scrH\,)\,=\,6$,
 it is therefore also immediately clear that $\tau$ can not be a multiple of the vector cross 
 product  described in \eqref{eq:tau} unless $\tau \,=\,0$.

 \paragraph{\bf Homogeneous Nearly Parallel $\G_2$--Spaces of Type 3}
 Here, $(\,M,\,g\,)$ is a $7$--dimensional homogeneous simply connected $3$--Sasakian manifold.
 The Sasakian foliation is regular and the leafs are isomorphic to $\Sp(\,1\,)$ 
 (for a round sphere) or $\SO(\,3\,)$ (otherwise), cf.~\cite[Proposition~1.2.10]{BG1}. 
 The connected component $G$ of the automorphism group of the $3$--Sasakian structure 
 acts transitively on  $(\,M,\,g\,)$. Furthermore, there exists a 
 connected subgroup $K\,\subset\, G\,$ containing the isotropy group $H$ as a normal subgroup 
 such that $K\,=\,H\cdot \Sp(1)$ and $G/K$ is a Wolf space, i.e., a symmetric space 
 with a homogeneous quaternionic K\"ahler structure, cf. \cite[Theorem 3.2.6]{BG1}.
 Here, the characteristic factor $K\,/\,H$ acts by isometries on $(\,M,\,g\,)$ but not by
 automorphisms of the $3$--Sasakian structure, cf. also Remark \ref{re:3--Sasakian}

 \bigskip
 \noindent
 The only  $7$--dimensional simply connected homogeneous $3$--Sasakian 
 manifolds are the round sphere $\S^7\,=\,G/H$ with $G\,:=\,\Sp(\,2\,)$ and $H\,:=\,\Sp(\,1\,)$,
 and the Aloff-Wallach space $\NM(\,1\,,\,1\,)\,=\,G/H$ with $G\,:=\,\SU(\,3\,)$ and $H\,:=\,\U(\,1\,)$ 
 corresponding to the remaining $6$--dimensional generalized flag manifolds $\CP^3$ and $\bbF^3$, respectively.
 For $M\,:=\,\NM(\,1\,,\,1\,)$ with the $3$--Sasakian metric $g$ we have $\dim\, KS (\,M,\,g\,)\,=\,3$,
 and $(\,M,\,g,\,\sigma\,)$ is a homogeneous nearly parallel $\G_2$--space
 where $\sigma$ is an arbitrary element of the projective plane $\Proj KS (\,M,\,g\,)$ of nearly parallel $\G_2$--structures;
 this is the only homogeneous example of type $3$.  The connected component of the isometry group is given by 
 $\SU(\,3\,)\,/\,\Z_3\,\times\, \SO(\,3\,)$,  cf. \cite{WZ}. 
 Consequently, Lemma \ref{le:proper} implies that $\Aut(\,M,\,\sigma\,)_\circ\,\cong\, \SU(\,3\,)\,/\,\Z_3\,\times\,\T^1$, 
 where the embedding $\T^1\,\subset\,\SO(\,3\,)$ depends on the choice of 
 $\sigma\,\in\,\Proj KS(\,M,\,g\,)$, as is also clear from \eqref{eq:AF}.

 \begin{proposition}\label{p:3Sas}
  The Aloff--Wallach space $M\,:=\,\NM(\,1\,,\,1\,)$ with the $3$--Sasakian metric $g$
  and an arbitrary $\sigma\,\in\,\Proj KS\, (\,M,\,g\,)$ is not naturally 
  reductive in the sense of Definition~\ref{de:NR_NP}. 
  Yet $(\,M,\,g\,)$ is a naturally reductive homogeneous space 
  with transvection group $\SU(\,3\,)\,/\,\Z_3\,\times\, \SO(\,3\,)$.
 \end{proposition}

 \begin{proof} 
  Again, $(\,M,\,g,\,\sigma\,)$ is not a naturally reductive nearly parallel
  $\G_2$--space by  Corollary \ref{co:proper}. 
  Therefore, $g$ is not the normal metric of $\SU(\,3\,)/\U(\,1\,)$:
  otherwise  $(\,M,\,g,\,\sigma\,)$ would be a normal homogeneous nearly parallel
  $\G_2$--space according to the first part of Proposition~\ref{p:NR_NP}. 
  Furthermore, the $3$--Sasakian metric $g$ is $\SU(\,3\,)\,\times\, \SO(\,3\,)$--invariant 
  and bundle--like over $\CP^2$. Since the isotropy representation of $\CP^2$
  is irreducible, $g$ belongs necessarily to the family of naturally reductive metrics
  parameterized by $s\,>\,-\,1$, $s\,\neq\,0$ constructed in Lemma~\ref{le:Nor_SG}. 
  In particular, the transvection group is $\SU(\,3\,)\,/\,\Z_3\,\times\, \SO(\,3\,)$.
  \end{proof}

 \bigskip
 \noindent
  By the same argument, it follows that every homogeneous $3$--Sasakian manifold  $(\,M,\,g\,)$ of dimension at least $7$ 
  is a naturally reductive space. Also it is mentioned in \cite[Theorem~3.2.9]{BG1} 
  that $g$ is not the normal metric of $G/H$. Hence, the transvection group is always covered by $G\,\times\,\T^1$.
  The content of Propositions~\ref{p:S7}--\ref{p:3Sas}  is summarized as follows:

 \begin{theorem}
 \label{th:CLASS_G2}
  The following can be said about naturally reductive structures 
  on simply connected homogeneous nearly parallel $\G_2$--spaces:
  \begin{enumerate}
  \item The round sphere $\S^7\,=\,\Spin(\,7\,)\,/\,\G_2$ (with an arbitrary nearly parallel $\G_2$--structure) 
           and the three proper nearly parallel 
           $\G_2$--spaces $\S^7_{\mathrm{squashed}}$, $V_1$ and $V_3$ are 
           standard normal homogeneous nearly parallel $\G_2$--spaces in the sense of Definition~\ref{de:NR_NP}.
  \item The nearly parallel $\G_2$--metrics underlying  simply connected homogeneous nearly parallel 
          $\G_2$--spaces of types 2 and 3 are naturally reductive.
          However, their nearly parallel $\G_2$--structures are not naturally reductive in the sense of Definition~\ref{de:NR_NP}. 
  \item The homogeneous proper nearly parallel $\G_2$--metrics of the Aloff--Wallach spaces
          $\NM(\,k\,,\,l\,)$ different from $\NM(\,1,\,1\,)$ are not naturally
          reductive.
  \end{enumerate}
 \end{theorem}

 \noindent
  In particular, by Corollary \ref{co:NR_NP}, the torsion form of a $7$--dimensional compact simply connected 
 naturally reductive space $(\,M,\,g,\,\bar\nabla\,)$ is a non--zero multiple of
 a vector cross product if and only if $(\,M,\,g,\,\bar\nabla\,)$ is the round $7$--sphere
 $\Spin(\,7\,)\,/\,\G_2$ or one of the three standard normal homogeneous 
 spaces $\S^7_{\mathrm{squashed}}$,  $V_1$ and $V_3$ in accordance 
 with Theorem \ref{th:GVCP}. By Corollary \ref{co:proper}, the latter 
 three spaces are also the only simply connected compact proper nearly parallel $\G_2$--spaces 
 whose underlying metric tensors are naturally reductive.

\subsection{Comparison with the Results from~\cite{St1,St2,St3}}
\label{se:SP}
 The principal goal of this section is to classify the 
 naturally reductive spaces $(\,M,\,g,\,\bar\nabla\,)$ whose torsion form is 
 a non--zero multiple of a vector cross product. We will indicate how this can be done
 independently of the results from~\cite{FKMS}. For this purpose we start
 from the list of simply connected naturally reductive spaces in dimension seven 
given in \cite{St3}. 

 \begin{lemma}(\cite[Theorem 4.4]{St3})
 \label{le:storm}
  Let $(\,\hat G,\,\hat K,\hat B\,)$ be an effective naturally reductive triple 
  resulting in a $7$--dimensional compact simply connected
  irreducible naturally reductive space $\hat M^7\,=\,\hat G/\hat K$ whose transvection
  algebra $\hat\frakg$ is not  semisimple. Then, there exists an effective naturally 
  reductive triple $(\,G,\, K, B\,)$ whose Lie algebra $\frakg$ is semisimple with $B$ positive definite, 
  and a closed normal subgroup $H\,\subset\,K$ such that
  $\T^1\,:=\,K/H$ is a $1$--dimensional torus, so that $(\,\hat G,\,\hat K,\hat B\,)$ 
  belongs to the family of naturally reductive triples parameterized by $s\,>\,-\,1$, $s\,\neq\,0$ 
  constructed in the first part of Lemma \ref{le:Nor_SG}.
 \end{lemma}
 \begin{proof}
  According to \cite[Table~7]{St3} there are six relevant compact base spaces
  $G/K$ of semisimple type, where always $\rank\;\frakk\,=\,\rank\;\frakg$ holds. 
  Hence, every maximal torus of $\frakk$ is also maximal in $\frakg$.
  Also the corresponding non--degenerate bilinear form $B$ on $\frakg$ is positive definite in each case.
  Therefore, the discussion from \cite[p.~541 ff.]{St1} shows that there exists some
  element in the center of $\frakk$ whose orthogonal complement
  integrates  into a closed normal subgroup $H\subset K$ thus providing
  precisely the family of naturally reductive structures parameterized by $s\,>\,-1$ 
  constructed in the first part of Lemma \ref{le:Nor_SG}.
 \end{proof} 

 \bigskip
 \noindent
  The corresponding six homogeneous quotients $G\,/\,H$ are: $\VM(\,5,\,2\,)$, $\MM(\,k,\,l\,)$, $\QM(\,k,\,\ell,\,m\,)$ 
  (as defined in the previous section), the Aloff--Wallach spaces $\NM(\,k,\,l\,)\,=\,\SU(\,3\,)\,/\,\U(\,1\,)$,
  and the spheres $\S^7\,=\,\SU(\,4\,)\,/\,\SU(\,3\,)$ and $\S^7\,=\,\Sp(\,2\,)\,/\,\Sp(\,1\,)$ each equipped with 
  the family of $\hat G\,:=\,G\,\times\,\T^1$--invariant naturally reductive structures over the $6$--dimensional base space $G\,/\,K$ parameterized by $s\,>\,-1$ constructed in Lemma \ref{le:Nor_SG}.
  This family contains, by definition for $s\,=\,0$, also the normal homogeneous structure induced by $B$ on $G\,/\,H$. Therefore, all six homogeneous
  quotients $G\,/\,H$ appear in both Tables 5 and 7 of \cite{St3}. 

 \begin{proposition}\label{p:semi-simple}
  None of the torsion tensors of the families of naturally reductive structures 
  on the previously mentioned six homogeneous quotients $G\,/\,H$ is a generalized vector cross product.
 \end{proposition}
 \begin{proof}
  As in Lemma \ref{le:T_and_R} let $\hat\tau$ denote
  the torsion form associated with the naturally reductive triples $(\,\hat G,\,\hat K,\,\hat B\,)$ ($s\,\neq\,0$) or $(\,G,\,H,\,B\,)$ ($s\,=\,0$).
  Let us assume for a contradiction that $\hat\tau\,=\,- \,\frac{\tau_0}6\,\sigma$
  where $\sigma$ is a $7$--dimensional vector cross product, cf. Corollary \ref{co:NR_NP}. 
  Also let $\tau$ denote the torsion form  associated with $(\,G,\,K,\,B\,)$ and $\rho_*$ be the linearized isotropy representation of $\frakk$.
  By \eqref{eq:HAT_TAU}, there exists $Z\,\in\,\frakh^\perp\,\subset\,\frakk$ and a vertical unit vector $E_7$ such that 
  $\hat\tau\,=\, \tau\,+\,\rho_*(Z)\,\wedge\,dE_7$. Comparison with~\eqref{eq:sigma} and~\eqref{eq:tau} results in:
  \begin{eqnarray*}
   \tau
   &=&
   -\;\frac{\tau_\circ}{6}\big(\;dE_1\wedge dE_3\wedge dE_5
   \,-\,dE_1\wedge dE_4\wedge dE_6\\
   &&\hphantom \ \ \ \ \,-\,dE_2\wedge dE_3\wedge dE_6
   \,-\,dE_2\wedge dE_4\wedge dE_5\;\big)
  \\
  \rho_*(Z)
  &=&
  -\;\frac{\tau_\circ}{6}\big(\;dE_1\,\wedge\,dE_2\;+\;dE_3\,\wedge\,dE_4
          \;+\;dE_5\,\wedge\,dE_6\;\big)
  \end{eqnarray*}
  Hence,  $J\,:=\,\frac{6}{\tau_\circ}\,\rho_*(Z)$ together with $\Psi^+\,:=\,\tau$ defines an $\SU(\,3\,)$--structure on $G/K$. 
  In particular, $\tau$ is a generalized vector cross product, and hence $G/K$ is a naturally reductive 
  nearly K\"ahler manifold $(\,M,\,g,\,J\,)$ according to Proposition 
  \ref{p:NK}. Moreover, $J$ lies in the image 
  of the linearized isotropy representation of $\frakk$. This implies that $J\,\star\,\tau\,=\,0$
  since the torsion form of a naturally reductive space is invariant under
  the isotropy action.  On the other hand, $\tau$ is a non--vanishing  $(3,0)\,+\,(0,3)$ form; 
  thus $J\,\star\,\tau(X,Y,Z)\, = \,-\,3\,\tau(JX,Y,Z)$. Therefore, $J\,\star\,\tau$ does not vanish, a contradiction.
 \end{proof}
 
 \bigskip
 \noindent
 Lemma \ref{le:storm} and Proposition \ref{p:semi-simple} together imply:

 \begin{corollary}\label{co:semi-simple}
  The transvection algebra of a naturally reductive nearly parallel $\G_2$--space is semisimple.
 \end{corollary}

 \bigskip
 \noindent
 Because of Corollary \ref{co:semi-simple}, we are left with eleven types of 
 $7$--dimensional simply connected naturally reductive quotients $G\,/\,H$ of semisimple type cf. 
\cite[Table 5]{St3}. Six of them can be thrown out again by Proposition \ref{p:semi-simple}.
 The remaining five possibilities for $G\,/\,H$ are $\S^7\,=\,\SO(\,8\,)/\SO(\,7\,)$ and the four 
 homogeneous quotients underlying the four standard normal homogeneous 
 nearly parallel $\G_2$--spaces mentioned in the first part of Theorem \ref{th:CLASS_G2}. 
 The round sphere $\S^7\,=\,\SO(\,8\,)/\SO(\,7\,)$
 is  a symmetric space and can be discarded. It remains to consider the families
 of naturally reductive structures on $\S^7$ and $\NM(\,1,\,1\,)$ with transvection algebras
 $\sp(\,2\,)\,\times\, \sp(\,1\,)$ and $\su(\,3\,)\,\times\,\so(\,3\,)$, respectively.
 The corresponding naturally reductive metrics are part of the canonical variations of the
 $3$--Sasakian metrics on $\S^7$ and $\NM(\,1,\,1\,)$, respectively, obtained by constant rescaling along the vertical bundle
 of the Riemannian submersion $f\,:\,G\,/\,H\,\longrightarrow\, G\,/\,K$ as described in
 Examples \ref{ex:Wi} and \ref{ex:Je}, cf. also \cite[Chpt. 9\,G]{Bes}. Furthermore, it is easy 
 to see that $f$ has totally geodesic fibers. Therefore, it is known that the canonical variation
 contains at most two Einstein metrics, cf. \cite[Proposition 9.72]{Bes}. Moreover, 
 in each case we already know two Einstein  metrics:  Both the $3$--Sasakian round metric 
 and the standard normal "squashed" metric of $\S^7$ are Einstein. Also the 
 $3$--Sasakian metric and the standard normal Wilking metric of  $V_3\,=\,\NM(\,1,\,1\,)$ 
 both are Einstein.  On the other hand, if the torsion tensor of a $7$--dimensional 
 naturally reductive space is a non--zero multiple of a vector cross product, then
 the naturally reductive metric is nearly parallel $\G_2$ according to Lemma \ref{le:AS}
 and hence Einstein. Therefore, the only further candidates are the two $3$--Sasakian metrics. 
 A naturally reductive structure with transvection algebra $\sp(\,2\,)\,\times\, \sp(\,1\,)$ can not be the canonical $\G_2$--connection 
 of a nearly parallel $\G_2$--structure  of the round sphere $\S^7$ by Proposition \ref{p:S7}. 
 It remains to consider the $3$--Sasakian  Aloff--Wallach space $\NM(\,1,\,1\,)$. 
 This is a nearly parallel $\G_2$--space of type 3, and hence  none of its nearly parallel 
 $\G_2$--structures is naturally reductive according to Corollary \ref{co:proper}, cf. also
 Remark \ref{re:3--Sasakian}. In Appendix~\ref{se:N11} we will verify by explicit calculations that, 
 in fact, the torsion form of a naturally reductive structure on $\NM(\,1,\,1\,)$
 with transvection algebra $\su(\,3\,)\,\times\,\so(\,3\,)$ is a non--zero multiple of a 
 vector cross product if and only if it is the standard normal homogeneous structure 
 of $\SU(\,3\,)\,\times\,\SO(\,3\,)\,/\,\U(\,2\,)$.  Thus, we remain again with the four 
 $7$--dimensional  simply connected standard normal 
 homogeneous spaces mentioned in the first part of Theorem \ref{th:CLASS_G2}.

\section{Proof of Theorem~\ref{th:LJR}}
\label{se:th:LJR}
 At the end of this section we will present the proof of Theorem~\ref{th:LJR}. 
 First, we need to recall some well--known facts from linear algebra.
 The standard decomposition of a skew--symmetric endomorphism $\tau$ on some euclidean 
 vector space $V$ is described as follows: There exist real numbers $\lambda_\ell\,>\,0$ ($\ell\,=\,1,\,2,\,\ldots$)
 with $\lambda_k\neq \lambda_\ell$, an  orthogonal decomposition $V\,=\, \bigoplus_{\ell \geq 0} V_\ell$ and Hermitian 
 complex structures $J_\ell$ on $ V_\ell$ for $\ell\, \geq \, 1$ such that $\tau\,=\,
 \bigoplus_{\ell\geq 1} \lambda_\ell\,J_\ell$. Thus, $V_0$ is the kernel of $\tau$, 
 $V_\ell$ is the $\lambda_\ell^2$--eigenspace of $-\tau^2$ for $\ell\,\geq\,1$, and the different
 non--zero  eigenvalues of $\tau$ are the complex conjugated pairs 
 $\pm\,i\,\lambda_\ell$. Also for each $X\in V$ there exists a 
 unique normalized polynomial  $P(\,\lambda\,)$ of smallest degree such 
 that $P(\,\tau\,)\,X\,=\,0$: In fact let $X_\ell$ denote the projection of $X$ onto
 $V_\ell$ for all $\ell\,\geq\,0$. Then, $P(\,\lambda\,)$ is the product of 
 all quadratic polynomials $\lambda^2\,+\,
 \lambda_\ell\,^2$ such that $X_\ell\,\neq \,0$ for $\ell\,\geq\,1$
 multiplied by $\lambda$ if $X_0\,\neq\,0$. In other words $P$ is the minimal 
 polynomial of $\tau$ restricted to the $\tau$--invariant subspace of $V$ 
 generated by $X$. For a generic vector $X$ we have $X_\ell\,\neq \,0$ for all 
 $\ell$ and $P$ coincides with the minimal polynomial of $\tau$.
 
 \begin{lemma}\label{le:U}
  Let $V$ be a euclidean space, $\tau\,:\,S\longrightarrow\so(\,V\,), \ s\longmapsto \tau(\,s\,)$,
  and $X\,:\,S\longrightarrow V,\ s\longmapsto X(\,s\,)$ be real analytic maps defined on
  some connected real analytic space $S$. There exists a polynomial $P(\,\lambda\,)$ such 
  that $P(\,\tau\,)\,X\,\equiv\,0$ if and only if for some non--empty, open and connected 
  subset $U$ of $S$ where the different  non--zero eigenvalues $0\,<\,\lambda_1^2(\,s\,)\, < 
  \,\lambda_2^2(\,s\,)\,<\,\cdots$ of $-\tau^2(\,s\,)$ are well defined continuous functions
  in the variable $s\in U$, the following condition holds for all $\ell\,\geq\,1$:
  $$
  X_\ell
   \;\;\equiv\;\;0
   \qquad\text{\bf or}\qquad
   \lambda_\ell \;\;\equiv\;\;\text{\rm const}
  $$
  In the affirmative case, let $P(\,\lambda\,)$ be the product of the quadratic factors
  $\lambda^2\,+\,\lambda_\ell^2$ such that $X_\ell\,\not\,\equiv\,0$, multiplied 
  by $\lambda$ if $X_0\,\not\equiv\,0$. This defines a polynomial
  $P(\,\lambda\,)$ such that $P(\,\tau\,)\,X \equiv\,0$. In addition, $P$ is the polynomial of smallest degree 
  with this property.
 \end{lemma}

 \begin{proof}
 For sufficiency, let $I\,=\,\{\ell_1,\,\ell_2,\,\ldots,\,\ell_d\}\,\subset\{\,1,\,2,\ldots\,\}$ denote
 those indices such that $\lambda_{\ell_ k}$ is constant for $k\,=\,
 1,\,\ldots,\,d$. By assumption, $X_\ell\,\equiv\,0$ for all $\ell\,\geq\,1$ and
 $\ell\,\notin\,I$. Therefore, $P(\,\tau\,)\,X\,\equiv\,0$ on $U$ where
 $P(\,\lambda\,)\,:=\,\lambda(\,\lambda^2\,+\,\lambda_{\ell_1}^2\,)\,\cdots\,
 (\,\lambda^2\,+\,\lambda_{\ell_d}^2\,)$.
 Since $(\,\tau,\,X\,)$ is real analytic, we obtain that  $P(\tau)\,X\,\equiv\,0$ on $S$. 

 \bigskip
 \noindent
 For necessity, suppose that there exists a polynomial $P(\,\lambda\,)\,$ such
 that $P(\,\tau\,)\,X\,\equiv\,0$. Let $U$ be an open and connected subset of
 $S$ where the different non--zero eigenvalues $\lambda_\ell^2(\,s\,)$ of $-\tau^2(\,s\,)$ are
 continuous functions. Let $I\,=\,\{\ell_1,\,\ell_2,\,\ldots,\,\ell_d\}\,\subset\{\,0,\,1,\,2,\ldots\,\}$ 
 denote those indices with $X_{\ell_k}\,\equiv\,0$ on $U$ for
 $k\,=\,1,\,\ldots,\,d$. Suppose $\ell\,\geq\,1$ and $\ell\notin I$. 
 We will show that  the corresponding eigenvalue $\lambda_\ell$ is constant:
 By choice of $\ell$, we have $X_\ell\,\not\equiv\,0$. Since $X_\ell$ is a continuous map on $S$,  
 we can assume that $X_\ell(\,s\,)\,\neq\,0$ for all $s\in U$. Hence, the quadratic factor
 $\lambda^2\,+\,\lambda^2_{\ell}(\,s\,)$ divides $P(\,\lambda\,)$
 for all $s\,\in\,U$, put differently $\pm\,i\,\lambda_\ell(\,s\,)$ belongs to the
 finite set of zeros of $P$. Because $\lambda_\ell(\,s\,)$ is a continuous
 function, we see that it actually does not depend on $s$, i.e.,
 $\lambda_{\ell}$ is constant. The other assertions are straightforward. 
 \end{proof}

 \bigskip
 \noindent
 In order to obtain the decomposition of the space $\End^{\mathrm{sym}}\,V$
 of symmetric endomorphisms of $V$ with respect to the action
 $\tau\,\star\,\K\,:= \,[\,\tau,\,\K\,]$, recall that on the one hand $V_{\geq 1}\,:=
 \, \bigoplus_{\ell\geq 1} V_\ell$ is a Hermitian vector space with complex structure 
 $J_\tau\,:=\,\bigoplus_{\ell\geq 1}\frac{1}{\lambda_\ell} \tau|_{V_\ell}$.
 Set $\K_{(1,1)}\,:=\,\frac12(\,\K  \,-\,J_\tau\,\circ\,\K \,\circ\, J_\tau\,)$ and
 $\K_{(2,0)+(0,2)}\,:=\,\frac12(\,\K\,+\,  J_\tau\,\circ\K\,\circ\, J_\tau\,)$  for all 
 $\K\in \End^{\mathrm{sym}}\,V_{\geq 1}$.  Then, $\tau\,\circ\, \K_{(2,0)+(0,2)}\, 
 = \,-\,\K_{(2,0)+(0,2)}\circ \tau$   and $\tau\,\circ\,\K_{(1,1)}\,=\,\K_{(1,1)}  \circ \tau$ on $V_{\geq 1}$. 
 On the other hand, let $\K^{k,\ell}$ denote the symmetric endomorphism of $V$ defined by
 \begin{equation}\label{eq:ID1}
  \langle \K^{k,\ell}\,Y,\, Z\,\rangle
  \;\;:=\;\;
  \langle \K\,Y_k,\,Z_\ell\rangle\,
  \;+\; \langle \K\,Y_\ell,\,Z_k\rangle\,
  \end{equation}
 for all $\,Y,\,Z\,\in\,\, V$ and $0\leq k < \ell$ and put
  \begin{equation}\label{eq:ID2}
  \langle \K^{k,k}\,Y,\, Z\,\rangle
  \;\;:=\;\;
  \langle \K\,Y_k,\,Z_k\,\rangle\, 
 \end{equation}
 for all $\,Y,\,Z\,\,\in\,V$ and $k\geq 0$. Then, $\K^{k,\ell}\in \End^{\mathrm{sym}}\,V_{\geq 1}$ for all $1\leq k \leq \ell$ and
 \begin{equation}\label{eq:ID3}
  \K
  \;\;=\;\;
   \underset{0\leq k}\oplus \K^{0,k} \underset{1\leq k\leq \ell}\oplus\K^{k,\ell}_{(1,1)}\oplus  \K^{k,\ell}_{(2,0)+(0,2)}
 \end{equation}
 as well as:
  \begin{eqnarray}
  \label{eq:ID4}
   -\,(\,\tau\,\star\,)^2 \K^{0,k}
   &=&
   \lambda_k^2\, \K^{0,k}
   \\[5pt]
  \label{eq:ID5}
   -\,(\,\tau\,\star\,)^2 \K^{k,\ell}_{(1,1)}
   &=&
   (\,\lambda_\ell\, -\, \lambda_k\,)^2 \K^{k,\ell}_{(1,1)}
   \\[5pt]
   \label{eq:ID6}
   -\,(\,\tau\,\star\,)^2 \K^{k,\ell}_{(2,0)+(0,2)}
   &=&
   (\,\lambda_\ell\, +\, \lambda_k\,)^2 \K^{k,\ell}_{(2,0)+(0,2)}
  \end{eqnarray}
 
 \bigskip
 \noindent
 Let $(\,M,\,g,\,\bar\nabla\,)$ be a naturally reductive space with
 torsion form $\tau$ and symmetrized Riemannian curvature tensor $\K$.
 Recall that the latter is a polynomial $X\,\longmapsto\,\K(\,X\,)$ of degree two 
 with values in $\End^{\mathrm{sym}}\,TM$. Also $\tau$ defines a $1$--form 
 $X\,\longmapsto\,\tau_X$ on $TM$ with values in $\End^\mathrm{skew}\,TM$.
 Then, the previous construction can be transferred in the pointwise sense, i.e.,
 with $\tau\,:=\,\tau_X$ and $\K\,:=\,\K(\,X\,)$. In particular, 
 Equations \eqref{eq:ID1}--\eqref{eq:ID3} remain valid via this interpretation.
 Furthermore, let $p\in M$, set $T\,:=\,T_pM$ and let $\S(T)$ denote the unit--sphere 
 of $T$. Given some non--empty open connected subset $U\,\subset\,\S(T)$
 where the non--zero eigenvalues $0\,<\,\lambda_1^2(\,X\,)\,< 
 \,\lambda_2^2(\,X\,)\,<\,\cdots$  of $-\,\tau_X^2$ 
 are well defined continuous functions in the variable $X\in U$, also Equations 
 \eqref{eq:ID4}--\eqref{eq:ID6} remain valid with a similar meaning.

 \begin{proposition}\label{p:EXIST_LJR}
  There exists a linear Jacobi relation if and only if the following three conditions
  together hold for $1\,\leq\, k\, < \,\ell$:
  \begin{itemize}
   \item $\lambda_k$ is constant,
         or $\K^{k,k}_{(2,0)+(0,2)}\,\equiv\,\K^{0,k}\,\equiv\,0$;
   \item $\lambda_k\,-\,\lambda_\ell$ is constant,
         or $\K^{k,\ell}_{(1,1)}\,\equiv\, 0$;
   \item $\lambda_k\,+\,\lambda_\ell$ is constant,
         or $\K^{k,\ell}_{(2,0)+(0,2)}\,\equiv\, 0$.
  \end{itemize}
  In case these conditions hold, let $Q(\,\lambda\,)$ be the least common multiple of the irreducible
  factors $\lambda^2\,+\, \frac14\, (\lambda_k\, - \, \lambda_\ell)^2$ such
  that $\K^{k,\ell}_{(1,1)}\,\not\equiv\,0$ ($1\leq k\,<\, \ell$),
  $\lambda^2\,+\,\frac14\,(\lambda_k\,+\,\lambda_\ell)^2$ such that
  $\K^{k,\ell}_{(2,0)+(0,2)}\,\not\equiv\, 0$ ($1\,\leq\, k\, \leq\, \ell$)
  and $\lambda^2\,+\,\frac14\,\lambda_k^2$ such that $\K^{0,k}\,\not\equiv\,0$
  ($1\leq k$). Then, $P(\,\lambda\,)\,:=\,\lambda\,Q(\,\lambda\,)$ defines the
  linear Jacobi relation $P(\,\scrT\,)\,\K\,\equiv\,0$ in accordance with
  Corollary~\ref{co:JR}. In addition, the only  possibility for a linear Jacobi relation of strictly 
  smaller degree is $Q(\,\scrT\,)\,\K\,\equiv\,0$.
 \end{proposition}
 \begin{proof}
 Set $S\,:=\,\S(\,T\,)$ and consider the vector space $V\,:=\,\End^{\mathrm{sym}}\,T$ in order to 
 apply Lemma \ref{le:U} to the pair of functions
  $$
     (\,\scrT,\,\K\,):\;\;S\;\longrightarrow\; \so(\,V\,)\;\oplus\;V,\qquad X\;\longmapsto\;(\,\scrT(\,X\,),\,\K(\,X\,)\,)
  $$ 
 For each $X\in U$, the set of non--zero eigenvalues $\{\mu_1^2(\,X\,),\,\mu_2^2(\,X\,),\,\ldots\}$  of $-\,\scrT(\,X\,)^2$ is the 
 (not necessarily disjoint) union of the following two sets:
 $$
  \begin{array}{l}
    \{\;\frac14(\,\lambda_k(\,X\,)\, -\, \lambda_\ell(\,X\,)\,)^2,\;\frac14(\,\lambda_k(\,X\,)\, 
    +\, \lambda_\ell(\,X\,)\,)^2\;|\;1\leq k \,< \,\ell\;\}\\ \{\;\frac14\lambda_k^2(\,X\,),\;\lambda_k^2(\,X\,)\;|\;k\geq 1\;\}
   \end{array}
  $$
  Possibly after passing to a smaller connected non--empty open subset $\tilde U\subset U$ and if necessary 
  removing duplicates, we can assume  that also $0\,<\, \mu_1^2 \,<\,\mu_2^2\,<\, \cdots$ are well defined 
  continuous functions on $U$. The result follows from Lemma \ref{le:U}.
 \end{proof}

 \begin{remark}\label{re:EXIST_LJR}
  Let $(\,M,\,g,\,\bar\nabla\,)$ be a naturally reductive space with
  torsion form $\tau$. The curvature tensors $\bar R$ and $R$ of
  $\bar\nabla$ and $\nabla$ are related by
  \begin{equation}\label{eq:BE_7.87}
   \bar R(\,Y,\,X,\,X,\,Y\,)
   \;\;=\;\;
   R(\,Y,\,X,\,X,\,Y\,)\;+\;\frac14\;\langle\;\tau_X^2Y,\;Y\;\rangle
  \end{equation}
  see~\cite[Prop.~7.87]{Bes}. Moreover, it is clear that $(\tau^2_X)^{k,\ell}
  \,=\,0$ for $k\neq \ell$ and that $(\tau^2_X)_{(2,0)+(0,2)}^{k,k}\,=\,0$
  for all $k\,\geq\,1$. Therefore, the assertions of Proposition~\ref{p:EXIST_LJR}
  remain unaffected if $R$ is replaced by $\bar R$.
 \end{remark}

 \noindent
 However, $P(\,\lambda\,)$ is not always the
 polynomial of smallest degree such that $P(\,\scrT\,)\,\K\,\equiv\,0$. A
 trivial example is the flat euclidean space, where $Q(\,\scrT\,)\,\K\,\equiv\,0$
 with $Q(\,\lambda\,)\, = \, 1$. Nevertheless we can rule out in many
 situations that $Q(\,\scrT\,)\,\K\,\equiv\,0$ holds:
  
 \begin{lemma}
 \label{le:MIN_LJR}
  Let $(\,M,\,g,\,\bar\nabla\,)$ be a naturally reductive space with torsion form $\tau$, associated linear operator 
  $\scrT$ on $\Jac^{\mathrm{\bullet}}\,TM$ defined by \eqref{eq:Recursion_Operator} 
  and symmetrized Riemannian curvature tensor $\K$.
  Let $Q(\,\lambda\,)$ be a normed polynomial that is not divisible 
  by the linear factor $\lambda$. If $Q(\,\scrT\,)\K\,\equiv\, 0$, then $\Ric\,=\,0$.
 \end {lemma}

 \begin{proof}
 Let $X\,\in\,T\,:=\,T_pM$ and let $\Sym^2(\,T\,)\,$ denote the space of symmetric $2$--tensors. 
 The action of $\scrT(\,X\,)$ on $\End^{\mathrm{sym}}\,T$ defined in \eqref{eq:def_of_derivation} 
 can be seen as an action on $\Sym^2(\,T\,)\,$ via the canonical isomorphism $\Sym^2(\,T\,)\,\cong\,
 \End^{\mathrm{sym}}\,T$. Furthermore, by assumption $Q(\,\lambda\,)$ is the product of linear 
 factors $\lambda\,-\,\mu_\ell$  with complex numbers $\mu_\ell\,\neq \,0$, where the index 
 $\ell$ is an integer between $1$ and $d$. Taking the trace of the $\End^{\mathrm{sym}}\,T$--valued 
 equation $Q(\,\scrT\,)\K\,\equiv\, 0$, we have
 $$
  (\;\scrT(\,X\,)\, - \,\mu_1\,\id\;) \cdot \cdots \cdot
  (\;\scrT(\,X\,)\, - \,\mu_d\,\id\;)\;
   \Ric(\,X,\,X\,)
  \;\;=\;\;
  0
 $$
 Since $\scrT(\,X\,)\, \Ric (\,X,\,X\,) \,=\, \Ric(\,\tau_X\,X,\,X\,)\,=\,0$, this reduces to
 $$
  \mu_1\cdot\; \cdots \;\cdot\mu_d\;\cdot\;
   \Ric(\,X,\,X\,)
  \;\;=\;\;
  0
 $$
 Because the $\mu_\ell$ are different from zero, we conclude that
 $\Ric\,\equiv\,0$.
 \end{proof}

 \begin{corollary}\label{co:MIN_LJR}
  In the situation of  Proposition~\ref{p:EXIST_LJR}, suppose that $(\,M,\,g,\,\bar\nabla\,)$ is not Ricci
  flat. If there exists a linear Jacobi relation, then $P(\,\scrT\,)\,\K\,\equiv\,0$ is the minimal linear Jacobi relation. 
  In particular, the order $d$ is even, cf. Proposition \ref{p:LJR_STRUK}.  
  Therefore, a linear Jacobi relation of order $2$ is automatically minimal unless $(\,M,\,g\,)$ is a symmetric space.
 \end{corollary}
 
 \begin{proof}
 Let $Q$ and $P$ be the polynomials from
 Proposition~\ref{p:EXIST_LJR}. Then, we have $P(\,\scrT\,)\,\K\,\equiv\,0$ and 
 the only possibility for a linear Jacobi relation of strictly smaller degree is $Q(\,\scrT\,)\,\K\,\equiv\,0$.
 However, the second alternative can be thrown out because of Lemma \ref{le:MIN_LJR}. 
 In particular, a minimal linear Jacobi relation of order strictly
 smaller than two is necessarily of order zero. Then, $\nabla_{X}R (\,Y,\,X,
 \,X,\,Y\,)\,=\,0$ holds for all vector fields $X$ and $Y$. Since $\nabla R$ has 
 the symmetries of a Young diagram with two rows of length $3$ and $2$, 
 respectively, this implies $\nabla R \,=\,0$, that is $(\,M,\,g\,)$ is a symmetric space.
 \end{proof}

 \paragraph{Proof of Theorem~\ref{th:LJR}}
\begin{enumerate}
\item 
          We will establish the linear Jacobi relation \eqref{eq:LJR_BS}
          on the total space of the Hopf fibration $\hat\MM^n(\,\kappa,\,r\,)\longrightarrow \MM^n(\,\kappa\,)$ 
          of Hopf circles of radius $r$ over some complex space form of constant holomorphic 
          sectional curvature $\kappa$. Let us first assume that $\kappa\neq 0$. We let $\hat\nabla$ 
          be the naturally reductive structure  whose torsion tensor $\hat\tau$ and curvature tensor $\hat R$
          are given by \eqref{eq:BS_TOR}. There is the orthogonal splitting $T\hat\MM^n(\,\kappa,\,r\,)\,=\,\scrH\oplus\scrV$ into 
          the horizontal and the vertical subbundle. Therefore, every unit vector $U$ of $T\,:=\,T_p\hat 
          \MM^n(\,\kappa,\,r\,)$ can be uniquely written as $U\,=\,X\,+\,t\,V_0$ where $V_0$ is a vertical 
          unit vector and $X$ is a horizontal vector $\langle \,X,\,V_0\,\rangle\,=\,0$ satisfying 
          $\|X\|^2\,+\,t^2 \,=\,1$. Then, $\hat\tau_U\,=\,c\,\big  (\, J\,X\,\wedge V_0^\flat\,+\,t\,J\,\big )$
          where $J$ denotes the Hermitian structure of $\MM^n(\,\kappa\,)$ and $c\,\neq\,0$ is given by \eqref{eq:c2}. 
          Hence, the restriction of $\hat\tau_U$ to the $3$--dimensional space $\{\,X,\,J\,X,\,V_0\,\}$ is a skew--symmetric endomorphism
          of length $|\,c\,|$. Set $T_0\,:=\,\R\,(\,X\,+\,t\,V_0\,)$, let $T_1$ be the orthogonal complement of $\{\,X,\,J\,X\,\}$ in
          $\scrH\,=\,\{\,V_0\}^\perp$ and $T_2$ be the orthogonal complement of
          $U\,=\,X\,+\,t\,V_0$ in $\{\,X,\,J\,X,\,V_0\,\}$.  Then,
          \begin{equation}
           \label{eq:BS_DEC}
           T
           \;\;=\;\;
           T_0\;\oplus\;T_1\;\oplus\;T_2
          \end{equation}
          is the eigenspace decomposition of $T$ with respect to $-\, \hat\tau_U^2$.
          The non--zero eigenvalues $\lambda_1^2\,:=\,\,t^2\,c^2$
          and $\lambda_2^2\,:=\,c^2\,$ are well defined 
          continuous functions in $t$ satisfying $\lambda_1^2\,<\,\lambda_2^2$ 
          for $0\,<\,t\,<\,1$. In order to apply Proposition~\ref{p:EXIST_LJR}, we may substitute the Riemannian curvature tensor 
          of $\hat \MM^n(\,\kappa,\,r\,)$ by $\hat R$ as justified by Remark~\ref{re:EXIST_LJR}.
          Furthermore, the symmetrization $\hat \K(\,U\,)\,:\,V\mapsto \hat \K (\,V,\,U,\,U\,)$ of $\hat R$ satisfies 
          $\hat\K(\,U\,) \,=\, \K(\,X\,)\,+\,c^2\,(\,J\,X\,)^\flat\,\odot\,(\,J\,X\,)^\flat$
          where $\K$ is the symmetrized Riemannian curvature tensor \eqref{eq:FS_CURV}
          of the Fubini--Study metric and $\lambda\,\odot\,\lambda$ denotes
          the symmetric square of a $1$--form $\lambda$. Polarization
          of~\eqref{eq:FS_CURV} shows that $\K(\,X\,)|_{T_0}\, =\, 0$ and $\K(\,X\,)(\,T_\ell\,)\subset T_\ell$ for
          $\ell\,=\,1,\,2$.
          Also $\K(\,X\,)|_{T_1}\, =\, \frac{\kappa}{4}\,\|X\|^2\,\id$ and hence $[\,\K(\,X\,),\,\tau_X\,]|_{T_1}\, =\,0$. Thus,
          $\K\,=\,\K^{1,1}_{(1,1)}\,\oplus\, \K^{2,2}$. 
          In addition, $(\,J\,X\,)^\flat\, \odot\,(\,J\,X\,)^\flat$ maps $T_2$ to $T_2$ and vanishes 
          on both $T_0$ and $T_1$. We conclude that  $\hat\K\,=\,\hat\K^{1,1}_{(1,1)}\,\oplus\, 
          \hat\K^{2,2}$. Therefore, the conditions of Proposition~\ref{p:EXIST_LJR} are satisfied, and
          only $\hat\K^{2,2}_{(2,0)+(0,2)}$ contributes a  non--trivial quadratic factor $\lambda^2\, + \,c^2$ 
          to the polynomial $Q(\,\lambda\,)$. Hence, $P(\,\lambda\,)\, = \,\lambda\,Q(\lambda)\,=\,
          \lambda(\,\lambda^2 \,+\, c^2\,)$. We obtain the linear Jacobi relation $P(\,\scrT\,)\K\,\equiv\,0$
          in compliance with Corollary~\ref{co:JR}. This proves \eqref{eq:LJR_BS}. 
          From Equations \eqref{eq:BS_TOR} and  \eqref{eq:BE_7.87}, it follows that the 
          Ricci tensor does not vanish since $\Ric(\,V_0,\,V_0\,)\,=\,\frac{n}{2}\,c^2$. 
          It is also clear that $\hat \MM^n(\,\kappa,\,r\,)$ is not a symmetric space
          unless it is a round sphere $\S^{2n+1}$. Thus, \eqref{eq:LJR_BS} 
          is minimal according to Corollary~\ref{co:MIN_LJR}. 
          Consequently, we conclude from Equation \eqref{eq:HG_TORCURV} 
          that for each $c>0$ there exists the same minimal Jacobi relation $P(\,\scrT\,)\,\K\,\equiv\,0$ 
          on the Heisenberg group $\hat \MM^n(\,0,\,\frac{1}{c^2}\,)$ equipped with a multiple $\frac{1}{c^2}$ 
          of its metric of type H. \qed
          
          \bigskip
          \noindent
          Naturally reductive spaces $(\,M,\,g,\,\bar\nabla\,)$ whose torsion form $\tau$
          is a generalized vector cross product are classified in Theorem \ref{th:GVCP}.
          Here, the eigenvalues of $\tau_X$, and thus also those of the associated operator $\scrT(\,X\,)$ 
          defined in \eqref{eq:def_of_derivation}, are constant in $X\in\S (\,TM\,)$. Therefore, both 
          the characteristic polynomial $P_{\mathrm{char}}(\,\lambda\,)$ and the minimal polynomial 
          $P_{\mathrm{min}}(\,\lambda\,)$ of $\scrT(\,X\,)$ do not depend on $X$. 
          In particular, we obtain the linear Jacobi relations $P_{\mathrm{char}}(\,\scrT\,)\,
          \K\,\equiv\, P_{\mathrm{min}}(\,\scrT\,)\,\K\,\equiv\,0$ according to Corollary \ref{co:JR}. 
          Furthermore, it is easy to see that a rescaling $\tilde g\,= \,\frac{1}{c^2}\,g$ 
          of the metric tensor by a constant factor transforms the coefficients $a_{2k}$
          of the linear Jacobi relation \eqref{eq:LJR} according to $\tilde a_{2k}\,=\,c^{2\,k} \,a_{2k}$ i.e.,
          $a_{2k}$ transforms under a constant rescaling of the metric like $\scal^k$.
          Hence, it suffices to prove the formulas  \eqref{eq:LJR_NP}, \eqref{eq:LJR_NK} 
          as well as their minimality for a specific rescaling of the metric tensor. 
          
         \item
         The four $7$--dimensional spaces mentioned in Theorem \ref{th:GVCP} are 
         standard normal homogeneous nearly parallel $\G_2$--spaces $(\,M,\,g,\,\sigma\,)$
         in the sense of Definition \ref{de:NR_NP}. If the metric  is normalized so that 
         $\scal\,=\,\frac{21}{8}$, then the nearly parallel constant $\tau_0$ is equal to $\pm\,1$.
         Hence, the  torsion form is given by $\tau\,=\,\pm\,\frac{1}{6}\,\sigma$
         according to Corollary \ref{co:NR_NP}. For each unit vector $X\in T_pM$, set 
         $T_0\,:=\,\R\,X$ and $T_1\,:=\,T_0^\perp$. Then, $T_0$ is the kernel
         of $\sigma_X$ and $\,\sigma_X|_{T_1}\,:\,T_1 \longrightarrow T_1$
         is a  complex structure. Thus, the eigenvalues of $-\,\tau_X^2$ are 
         $\{\,0,\,\frac1{36}\,\}$ and those of $-\,\scrT(\,X\,)^2$ are $\{\,0,\,\frac1{144},
         \,\frac1{36}\,\}$. However, $\lambda^2\,+\,\frac1{144}$ is
         not a factor of the polynomial $P(\,\lambda\,)$ defined in
         Proposition~\ref{p:EXIST_LJR}, because $\K^{0,1}$ vanishes
         due to $X\,\lrcorner\, \K(\,X\,)\,=\,0$. Therefore, already 
         $P(\,\lambda\,)\,:=\,\lambda\,(\,\lambda^2\,+\,\frac{1}{36}\,)$
         results in the linear Jacobi relation $P(\scrT\,)\,\K\,\equiv\,0$.
         This proves \eqref{eq:LJR_NP}. Moreover, a simply connected compact
         nearly parallel $\G_2$--space has generic holonomy $\SO(\,7\,)$.
         Therefore, it is a symmetric space if and only it is the round $7$--sphere.
         Otherwise Corollary \ref{le:MIN_LJR} implies that~\eqref{eq:LJR_NP} 
         is minimal. \qed
\item
          \bigskip
          \noindent
          The four $6$--dimensional spaces listed in Theorem \ref{th:GVCP} 
          are simply connected homogeneous strict nearly K\"ahler manifolds.
          As a result of Proposition \ref{p:NK}, the torsion form $\tau$ is given by 
          $\nabla J\circ J$. If the scalar curvature is normalized to
          $\scal\,=\,30$,  then $(\,J,\,\Psi^+\,:=\,\tau\,)$ defines 
          a normalized $\SU(\,3\,)$--structure; especially
          \eqref{eq:CT} holds with $c\,=\,1$. For each unit vector $X\in\,T_pM$, set 
          $T_0\,:=\,\mathrm{span}_{\R}\{\,X,\,J\,X\,\}$ and 
          $T_1\,:=\,T_0^\perp$.  Thus, $T_0$ is the kernel
          of $\tau_X$ and $\tau_X|_{T_1}\,:\,T_1 \longrightarrow T_1$ is a
          complex structure  according to \eqref{eq:CT}. Hence, the eigenvalues 
          of $-\,\tau_X^2$ are $\{\,0,\,1\,\}$. Consequently, the eigenvalues of 
          $-\,\scrT(\,X\,)^2$ are $\{\,0,\,\frac14 \,,\,1\,\}$ and the minimal polynomial
          of $\scrT(\,X\,)$ is
          \begin{equation}\label{eq:POL_NK}
           P(\,\lambda\,)
           \;\;:=\;\;
           \lambda\,
           (\,\lambda^2\,+\,\frac14\,\,)\,(\,\lambda^2\,+\,1\,)
           \;\;=\;\;
           \lambda^5\;+\;\frac54\,\,\lambda^3\;+\;\frac14\,\lambda
          \end{equation} 
          Thus, $P(\scrT\,)\,\K\,\equiv\,0$ holds, which results in \eqref{eq:LJR_NK}. It remains to show
          that this is the minimal linear Jacobi relation unless $(\,M,\,g\,)$ is the round $6$--sphere.
          Following Proposition~\ref{p:EXIST_LJR}, we have to show that 
          $\K^{0,1}$ and $\K^{1,1}_{(2,0)+(0,2)}$
          both do not vanish identically:
          \begin{itemize}
          \item Suppose that $\K^{0,1}\,\equiv\,0$. Then,
                   $\K(\,X\,)\,J\,X \in \R\,J\,X$ for all $X\in T$. Because
                   of Proposition \ref{p:CHSC} $(c)\Rightarrow (a)$, this implies that $(\,M,\,g\,)$ is the round $6$--sphere.
          \item Suppose that $\K^{1,1}_{(2,0)+(0,2)}\,\equiv\,0$. Then, $\langle\,\K(\,X\,)\,U,\, 
                  \tau_XU\,\rangle\,=\,0$ for all $X\,\in\,T$ and $U\,\in\,\{\,X,\,J\,X\,\}^\perp$. Proposition 
                  \ref{p:CHSC} $(d)\Rightarrow (a)$  shows that $(\,M,\,g\,)$ is the round $6$--sphere. \qed
          \end{itemize}
\end{enumerate}
\subsection{Comparison of Theorem~\ref{th:LJR} with known examples} 
\label{se:EX}
 Sporadic examples of linear Jacobi relations were observed before,
 cf.~\cite{GN}. In order to compare these results with Theorem~\ref{th:LJR},
 let $\rmK_\frakg$ denote the Killing form of some semisimple Lie algebra
 $\frakg$. 

 \begin{itemize}
  \item For a $3$--dimensional oriented naturally reductive
          space, the torsion form $\tau$ is a multiple $c$ of the standard
          volume form  $\det\,:=\,\d\,E_1\wedge \d\,E_2\wedge\d E_3$. Then,
          \begin{equation}\label{eq:EX_LJR_0}
           \K_3
           \;\;=\;\;
           -\,c^2\,\|\,\cdot\,\|^2\,\K_1
          \end{equation}
          cf. \cite[Theorem~4.3]{Gon}. 
  \item For the normal metric on Berger's manifold 
          $V_1\,=\,\Sp(\,2\,)/\SU(\,2\,)$ 
          corresponding to $-\frac1{30}\,\rmK_{\sp(\,2\,)}$,
          it was shown in \cite{NT} that: 
         \begin{equation}\label{eq:EX_LJR_1}
          \K_3
          \;\;=\;\;
          -\;\|\,\cdot\,\|^2\,\K_1
         \end{equation}
  \item For the normal metric induced by $-\frac1{12}\,
         \rmK_{\su(\,3\,)\oplus\so(\,3\,)}$  on 
          Wilking's manifold $V_3\,=\,\SU(\,3\,)\,\times\,\SO(\,3\,)\,/\,\U(\,2\,)$,
         we have according to \cite{MNT}:
         \begin{equation}\label{eq:EX_LJR_2}
          \K_3
          \;\;=\;\;
          -\;\frac25\,\|\,\cdot\,\|^2\K_1
         \end{equation}
  \item For the normal metric induced by $-\frac16\,\rmK_{\su(3)}$
          on the complex flag manifold $\bbF^3\,=\,\,\SU(\,3\,)/\T^2$,
          it was proved in \cite{Ar} that:
          \begin{equation}\label{eq:EX_LJR_3}
           \K_5
           \;\;=\;\;
           -\;\frac{5}{8}\,\|\,\cdot\,\|^2\,\K_3\;
           -\;\frac1{16}\,\|\,\cdot\,\|^4\K_1
          \end{equation} 
  \item For the normal metric induced by $-\frac1{12}\,\rmK_{\su(3)}$
          on $\CP^3\,=\,\SO(\,5\,)/\U(\,2\,)$, we have according to \cite{AAN}:
         \begin{equation}\label{eq:EX_LJR_4}
          \K_5
          \;\;=\;\;
          -\;\frac54\,\|\,\cdot\,\|^2\,\K_3
          \;-\;\frac14\,\|\,\cdot\,\|^4\,\K_1
         \end{equation} 
 \end{itemize}

 \noindent
 It remains to compare the above results with Theorem~\ref{th:LJR}. The Jacobi
 relation \eqref{eq:EX_LJR_0} is immediately clear since in dimension three
 the standard volume form $\det$ is a vector cross product. 
 In addition, every simply connected $3$--dimensional naturally reductive space is
 isometric to either $\hat\MM^1(\,\kappa,\,r\,)$ or $\hat\MM^1(\,0,\,\frac{1}{c^2}\,)$,
 cf. \cite[Chpt. 6]{TV}. Thus, $c$ is equal to $\frac12r\kappa$ for $\kappa\neq 0$ 
 or $\frac{1}{c^2}$ is the homothety factor by which we have rescaled the
 metric of type H on the $3$--dimensional Heisenberg group 
 $\hat\MM^1(\,0,\,\frac{1}{c^2}\,)$.
 
 \bigskip
 \noindent
 We still have to show that also Equations \eqref{eq:EX_LJR_1}--\eqref{eq:EX_LJR_4} 
 match with Theorem \ref{th:LJR}. Let $(\,G, \, H,\,B\,)$ be a standard normal 
 naturally reductive triple where $G$ is compact with semisimple Lie algebra $\frakg$.
 Then, the Killing form $\rmK_\frakg$ is negative 
 definite and there exists $b\,>\,0$ such that $B\,=\,-\,b^2\,\rmK_\frakg$. Let $\scal$ 
 denote the scalar curvature of the normal homogeneous space $G/H$.

 \begin{lemma}
  \begin{enumerate}
   \item (cf.~\cite[Lemma~7.1]{AlS})
         If $G/H$ is a standard normal homogeneous nearly parallel $\G_2$--space,
         then $\scal\, = \, \frac{63}{20}\frac{1}{b^2}$.
   \item (cf.~\cite[Lemma 5.4]{MS})
         If $G/H$ is a $6$--dimensional homogeneous strict nearly K\"ahler
         manifold, then $\scal\, = \, \frac{5}{2}\frac{1}{b^2}$.
  \end{enumerate}
 \end{lemma}

 \begin{corollary}
 The results from \cite{Ar,AAN,MNT,NT} are consistent with Theorem \ref{th:LJR}.
 \end{corollary}
 \begin{proof}
    Substituting $\scal\, = \, \frac{63}{20}\frac{1}{b^2}$ with $\frac{1}{b^2}\,=\,30$ and 
  $\frac1{b^2}\,=\,12$ into \eqref{eq:LJR_NP} gives us \eqref{eq:EX_LJR_1} 
  and \eqref{eq:EX_LJR_2}, respectively. Similarely, $\scal\, = \, \frac{5}{2}\frac{1}{b^2}$ 
  together with \eqref{eq:LJR_NK} produces \eqref{eq:EX_LJR_3} and \eqref{eq:EX_LJR_4}.
 \end{proof}

 \bigskip
 \noindent
 Hence, two out of three linear Jacobi relations on $6$--dimensional homogeneous strict nearly K\"ahler manifolds (different from $\S^6$) and 
normal homogeneous proper nearly parallel $\G_2$--spaces, respectively, described in Theorem \ref{th:LJR}
 were already known before. Apparently, the linear Jacobi relations on the nearly K\"ahler $\S^3\,\times\,\S^3$ and 
 the squashed sphere $\S^7_{\mathrm{squashed}}$ were previously overlooked.  Also the existence 
 of linear Jacobi relations on the total  spaces of the Hopf fibrations over complex space forms including 
 the Heisenberg groups is seemingly a new result.

\section{Twistor Equations and Covariant Derivatives of Curvature}
\label{se:GT}
 Consider the Lie algebra $\so(\,n\,)$ of the orthogonal group $\SO(\,n\,)$
 in dimension $n\,\geq\,5$ and denote by $r\,:=\,\lfloor\frac n2\rfloor$ its
 rank. A standard choice of a maximal torus for $\so(\,n\,)$ is the set
 of diagonal matrices $D\,=\,\diag(\,D_1,\,\ldots,\,D_r,\,0\,)$ where each
 submatrix $D_\mu$ ($\,\mu\,=\,1,\,\ldots,\,r$) is a $2 \times 2$--matrix of
 the form $D_\mu\,:=\,\left(\begin{smallmatrix}0 & -\epsilon_\mu\\
 \epsilon_\mu & 0\end{smallmatrix}\right)$ with a complex number
 $\epsilon_\mu\,\in\,\C$ and the trailing $0$ occurs exactly for odd $n$.
 The weights of the standard representation of $\so(\,n\,)$ on $\C^n$ are
 then the linear functionals $\{\,\pm\epsilon_\mu\,\}$ sending $D$ to
 $\epsilon_\mu$. We choose the standard linear ordering on the dual space
 of the torus such that $\epsilon_1\,>\,\cdots\,>\,\epsilon_r$.

 \bigskip
 \noindent
 In this setup irreducible complex representations $V_\lambda$ of
 $\so(\,n\,)$ are labeled by their highest weights $\lambda\,=\,
 \lambda_1\varepsilon_1\,+\,\ldots\,+\,\lambda_r\varepsilon_r$ with
 coefficients $\lambda_1,\,\ldots,\,\lambda_r\,\in\,\frac12\,\Z$ either
 all integers or all half integers satisfying the inequalities
 $$
  \lambda_1\;\;\geq\;\;\lambda_2\;\;\geq\;\;\ldots\;\;\geq\;\;
  \lambda_{r-1}\;\;\geq\;\;\lambda_r\;\;\geq\;\;0
  \qquad\qquad
  \lambda_1\;\;\geq\;\;\lambda_2\;\;\geq\;\;\ldots\;\;\geq\;\;
  \lambda_{r-1}\;\;\geq\;\;|\,\lambda_r\,|
 $$
 for $n$ odd and $n$ even, respectively. The representation $V_\lambda$ of the
 Lie algebra $\so(\,n\,)$ integrates into a representation of $\SO(\,n\,)$ precisely
 if all coefficients $\lambda_1,\,\ldots,\,\lambda_r$ are integers. 
 
 \bigskip
 \noindent
 The corresponding complex vector bundle associated to the
 $\SO(\,n\,)$--principal fiber bundle of positive orthogonal frames over an oriented
 Riemannian manifold $M$ of dimension $n$ will be denoted by $V_\lambda M$.
 The tensor product $V_\lambda\,\otimes\,\C^n$ decomposes into a sum of
 isotypic components $\bigoplus_\varepsilon V_{\lambda+\varepsilon}$ under
 $\so(\,n\,)$ corresponding to a splitting of the vector bundle $V_\lambda M
 \,\otimes\,T M\,=\,\bigoplus_\varepsilon V_{\lambda+\varepsilon}M$ into
 $\nabla$--parallel subbundles indexed by a weight $\varepsilon$
 of the standard representation $\R^n$. For every such $\varepsilon$,
 we consider the first order differential operator
 $$
  P_{\lambda,\varepsilon}:\;\;\Gamma(\;V_\lambda M\;)
  \;\stackrel\nabla\; \longrightarrow \;\Gamma(\;TM\,\otimes\,V_\lambda M\;)
  \;\longrightarrow\;\Gamma(\;V_{\lambda+\varepsilon}M\;)
 $$
 given by the covariant derivative followed by the projection to
 $V_{\lambda+\varepsilon}M$, called a generalized gradient. The highest
 weight $\lambda+\varepsilon_1$ always occurs with multiplicity one
 \cite[p.~6]{SW} and the corresponding gradient $T_{\lambda}\,:=\,P_{\lambda,\varepsilon_1}$
 is called a generalized twistor operator. The generalized twistor equation
 $T_{\lambda}v\,=\,0$ for sections $v\,\in\,\Gamma(\,V_\lambda M\,)$
 is a partial differential equation of finite type \cite{Pil}. If $T_{\lambda}v\,=\,0$,
 then $v$ is called a generalized twistor. Well--known
 examples are the conformal Killing equation for alternating $k$--forms
 corresponding to the highest weight $\lambda\,:=\,\varepsilon_1\,+\,
 \ldots\,+\,\varepsilon_k$ \cite{Sem}, the conformal Killing equation for
 symmetric Killing tensors \cite{HMS} and the Penrose twistor equation
 for positive and negative half spinors corresponding to $\lambda\,=\,
 \frac12\,(\,\varepsilon_1\,+\,\ldots\,+\,\varepsilon_{r-1}\,\pm\,
 \varepsilon_r\,)$. 

 \bigskip
 \noindent
 Recall that the symmetrized $k$--th covariant derivative $\K_k$ of the
 curvature tensor $R$ defined in \eqref{eq:SCD} is a section of $\Jac^k TM$ 
 and hence has the index symmetries of a  Young diagram with two rows of lengths 
 $k+2$ and $2$, respectively. Also there is a unique decomposition 
 \begin{equation}
  \K_k
  \;\;=\;\;
  \K_k^\circ\;+\;\K_k^{\textrm{\tiny rest}}
 \end{equation}
 were $\K_k^{\textrm{\tiny rest}}$ is an expression involving the metric
 tensor at least linearly and all traces of $\K_k^\circ$ vanish. For this
 reason $\K_k^\circ$ is called the completely trace--free part of $\K_k$. 
 Thus, $\K_k^\circ$ is a section of 
 $V_{(k + 2)\epsilon_1+2\epsilon_2}M$ according to Weyl's construction 
 of the irreducible representations of the orthogonal group
 (cf.~\cite[Chpt.~19.5]{FulH}). The parallelity of the metric
 tensor furthermore implies that
 \begin{equation}\label{eq:identify_the_twistor}
  T_{(k+2)\,\epsilon_1\,+\,2\epsilon_2}\K_k^\circ
  \;\;=\;\;
  \K_{k+1}^\circ
 \end{equation}
 Therefore, $\K_k^\circ$ is a generalized twistor if and only if 
 \begin{equation}\label{eq:T_k} 
  \K_{k+1}^\circ
  \;\;=\;\;
  0
 \end{equation}

 \begin{corollary}
 \label{co:generalized_twistor}
  \begin{enumerate}
    \item Let $(\,M,\,\,g\,)$ be the $(\,2\,n\,+\,1\,)$--dimensional ($n\,\geq\,2$)
            total space of the Hopf fibration
            over an $n$--dimensional complex space form
            or a $7$--dimensional simply connected standard 
            normal homogeneous space of positive sectional curvature 
            (possibly) except for $\Sp(\,2\,)/\Sp(\,1\,)$.
            Then, $\K_2^\circ$ is a generalized twistor.
   \item  Let $(\,M,\,g\,)$ be a $6$--dimensional homogeneous strict nearly K\"ahler manifold.
            Then, $\K_4^\circ$ is a generalized twistor.
  \end{enumerate}
 \end{corollary}

 \begin{proof}
 Let us assume that a linear Jacobi relation \eqref{eq:LJR}
 of degree $d$ holds. Then,
 \begin{equation*}
  T_{(d+2)\varepsilon_1\,+\,2\varepsilon_2}\,\K_d^\circ
  \;\;\stackrel{\eqref{eq:identify_the_twistor}}=\;\;
  \K_{d+1}^\circ
  \;\;\stackrel{\eqref{eq:LJR}}=\;\;
  \Big(\;g\,(\,a_1\,\K_{d-1}\;+\;a_2\,g\,\K_{d-3}\;+\;\cdots\;)
  \;\Big)^\circ
  \;\;\stackrel{\textrm{def.}}=\;\;
  0
 \end{equation*}
 where the last equality holds because the trace--free part of a term which
 involves the metric tensor $g$ at least linearly vanishes by definition.
 Therefore, Equation \eqref{eq:LJR} implies Equation \eqref{eq:T_k} for $k\,=\,d$. 
 Applying this argument to Equations \eqref{eq:LJR_BS}, \eqref{eq:LJR_NP} and \eqref{eq:LJR_NK} proves the corollary.
 \end{proof}
 
 \bigskip
 \noindent
 We don't know whether Equation \eqref{eq:T_k} also holds on $\Sp(\,2\,)/\Sp(\,1\,)$ for some $k$.

 \bigskip
 \noindent
 It is easy to see that \eqref{eq:T_k} holds with $k\,=\,0$ on some Einstein manifold
 if and only if $\nabla R\,=\,0$. More generally, the Jet Isomorphism Theorem 
 of Riemannian geometry implies that the Taylor expansion of the metric tensor in normal coordinates
 is already encoded in the sequence $\K_0,\,\K_1,\,\K_2,\,\ldots$. For an Einstein 
 manifold it even suffices to consider the sequence of the completely traceless parts $\K_0^\circ,\,
 \K_1^\circ,\,\K_2^\circ,\,\ldots$. The latter assertion can be argued as
 follows: It is straightforward to show that the curvature tensor $R$ of an Einstein
 manifold and its covariant derivative $\nabla R$ both are completely
 trace--free. The symmetries of the Riemannian curvature tensor can be used
 to show that all traces of $\K_{k+2}$ are determined by $\K_0,\,\K_1,\,
 \ldots,\,\K_k$. Our assertion follows by induction. 

 \bigskip
 \noindent
 Therefore, the generalized twistor equation \eqref{eq:T_k} seems to be
 in particular interesting in connection with Einstein manifolds. We 
 don't know whether such an equation holds for more Einstein manifolds than those mentioned in 
 Corollary \ref{co:generalized_twistor}, e.g. for the other homogeneous nearly parallel $\G_2$--spaces.

\appendix
\section{An explicit calculation for $\SU(\,3\,)\,\times\,\SO(\,3\,)\,/\,\U(\,2\,)$ }
\label{se:N11}
 We consider the invariant bilinear form on $\su(\,3\,)\,\oplus\, \su(\,2\,)$ defined by 
 $\hat B_s\,:=\,-\frac12\,\trace_{\su(3)}\,-\,\frac1{2\,s}\,\trace_{\su(2)}$ with $s\neq 0$. 
 Here, $\trace_{\su(n)}$ denotes the usual trace form $\trace(A,B)\,:=\, \trace(A\circ B)$ 
 on skew--Hermitian matrices.  Recall that $\rmK_{\su(2)}\,
 =\,4\,\trace_{\su(2)}$ and $\rmK_{\su(3)}\,=\,6\,\trace_{\su(3)}$ where $\rmK_\frakg$ denotes the
 Killing form. Therefore, $\hat B_{\frac32}$ is a multiple of the Killing form of $\su(\,3\,)\oplus\su(\,2\,)$. 
 We obtain a positive definite invariant bilinear form precisely for $s\,>\,0$. 
 The corresponding family $g_s$ of normal metrics are Wilking's metrics of positive sectional curvature 
 on $\NM(\,1\,,\,1\,)\,=\,\SU(\,3\,)\,\times\, \SO(\,3\,) /\U(\,2\,)$. For $ -1\,< s\,<\, 0$ we obtain naturally 
 reductive structures on $\NM(\,1\,,\,1\,)$ which are not normal.

 \begin{proposition}\label{p:AWS_11}
  The torsion tensor of the canonical connection on $\NM(\,1\,,\,1\,)$ associated with $\hat B_s$ is a multiple of a vector
  cross product if and only if $s\,=\,\frac32$.
 \end{proposition} 

 \begin{remark}\label{re:CH}
  For every $A\,\in\,\su(2)$ and every complex $2\times 2$-matrix $X$, we have
  \begin{equation}\label{eq:2x2-Matrizen-Identitaet}
   A\,X\,+\, X\, A\,=\,\trace(A\,X)\,\id+\,\trace(X)\,A
  \end{equation}
 \end{remark}

 \begin{proof}
 Every complex $2\times 2$-matrix $X$ is a complex 
 linear combination $a\, \id\,+\,b\, B$ with $B\in\su(\,2\,)$. 
 Hence, its suffices to verify \eqref{eq:2x2-Matrizen-Identitaet} for $X\,:=\,\id$ and $X\,\in\su(\,2\,)$.
 If $X$ is the identity of $\C^2$, then \eqref{eq:2x2-Matrizen-Identitaet} is obvious.
 Furthermore, the Theorem of Cayley--Hamilton reads $A^2\,+\,\det\,A\,\id\,=\,0$
 for $A\,\in\,\su(2)$ due to $\trace\,A\,=\,0$; in particular $\|\,A\,\|^2\,
 :=\,\det\,A\,=\,-\frac12\,\trace(\,A^2\,)$. By polarization, this identity
 becomes
 $$
  A\,B\;+\;B\,A\;-\;\trace(\,AB\,)\,\id
  \;\;=\;\;
  0
 $$
 valid for all $A,\,B\,\in\,\su(2)$. Thus, \eqref{eq:2x2-Matrizen-Identitaet} also holds for all $X\,\in \su(\,2\,)$.
 \end{proof}

 \bigskip
 \noindent
 {\bf Proof of Proposition~\ref{p:AWS_11}}
 Every element of $\su(3)$ is of the form
 $$
  \begin{pmatrix}A&- a \cr a^H &-\trace(A)\end{pmatrix}
 $$
 where $A\in\u(2)$ and $a^H$ denotes the Hermitian transpose of some
 column vector $a\in\C^2$. For each $s\,\neq\,0,\,-1$, there exists an
 $\hat B\,:=\,\hat B_s$--orthogonal direct splitting $\su(3)\oplus\su(2)\,=\,
 \frakh\oplus\frakm$ with
 $$
  \frakh 
  \;\;:=\;\;
  \left\{\left . \begin{pmatrix}A\,& 0 \cr 0 &-\trace(A)
  \end{pmatrix} \oplus  A\, -\,\frac12\, \trace(A)\,\id
  \;\;\right|\;\; A\in \u(2)\right\} 
  \;\;\subset\;\; \su(3)\; \oplus\; \su(2)
 $$ 
 and
 $$
 \frakm
  \;\;:=\;\;
  \left\{\left.\;
  \begin{pmatrix}A&- a \cr a^H & 0\end{pmatrix}
  \;\oplus\;-s\;A
  \;\;\right|\;A\,\oplus\,a\;\in\;\su(2)\,\oplus\,\C^2\;\;\right\}
  \;\;\subset\;\;
  \su(\,3\,)\;\oplus\;\su(\,2\,)
 $$
 in accordance with Lemma \ref{le:Nor_SG}. Then, we have the obvious isomorphisms
 $\frakh\,\cong\,\u(2)$ and $\frakm\,\cong\,\su(2)\,\oplus\,\C^2$. 
 Accordingly, we write elements of  $\frakm$ uniquely as formal direct sums
 $A\,\oplus\,a$ with $A\in \su(2)$ and $a\in \C^2$. Via this identification the scalar product
 on $\frakm$ is given by 
 $$
  \langle\, A\,+\,a,\,B\,\oplus\,b\,\rangle
  \;\;=\;\;
  (\,s\,+\,1\,)\langle\,A,\,B\,\rangle_{\su(2)}\,+ \,
  \langle\, a,\,b\,\rangle_{\C^2}.
 $$ 
 Since $\su(2)$ and $\C^2$ are $\hat B$-orthogonal,
 we see from the polarization formula that the residual bracket
 $[\,\cdot,\,\cdot\,]_\frakm\,:\,\frakm\,\times\,\frakm\, \longrightarrow\, \frakm$
 is a multiple $\frac{1}{c}$ of a vector cross product if and only if
 the alternating 3--tensor $T(\,x,\,y,\,z\,)\,:=\,\langle \,[\,x,\,y\,]_\frakm,\,z\,\rangle$
 satisfies the following three equations:
 \begin{eqnarray}
  \label{eq:VCP1}
  \frac{c^2}{s\,+\,1}\,T_A^2\,(\,B\,\oplus\,b\,)
  &=&
  \Big(\;\langle\,A,\,B\,\rangle\,A\,-\,\|A\|^2\,B\Big )
  \;\oplus\;
  \Big (\;-\,\|A\|^2\,b\;)\Big )
  \\
  \label{eq:VCP2}
  c^2\,T_a^2\,(\,B\,\oplus\,b\,)
  &=&
  \Big(\;-\,\|a\|^2\,B\;\Big)
  \;\oplus\;
  \Big(\;\langle\,a,\,b\,\rangle\,a\,-\,\|a\|^2\;b\;\Big)
  \\
  \label{eq:VCP3}
  c^2\,\big(\, T_A\circ T_a\,+\,T_a\circ T_A\,\big)
  \,(\,B\,\oplus\,b\,)
  &=&
  \Big(\;\langle\,a,\,b\,\rangle\,A\;\Big)
  \;\oplus\;
  \Big(\; (\,s\,+\,1\,)\,\langle\,A,\,B\,\rangle\,a\;\Big)
 \end{eqnarray}
 (for the notation, cf. \eqref{eq:sigmaX}. The commutator 
 of two elements of $\frakm$ is calculated straightforwardly:
 $$
  [\,A\,\oplus\,a\,,\,B\,\oplus\,b\,]
  \;\;=\;\;
  \begin{pmatrix}
   [\,A,\,B\,]\,+\,ba^H\,-\,ab^H & -(Ab-Ba) \cr
   (Ab-Ba)^H                       & -\,2\,i\,\Im(a^Hb)
  \end{pmatrix}
  \;\oplus\;
  s^2[\,A,\,B\,]
  \;\;\in\;\;
  \su(3)\oplus\su(2)
 $$
 In order to calculate the projection to $\frakm$ we consider
 the following map $\varphi\,\colon\,\su(3)\,\oplus\,\su(2)\, 
 \longrightarrow\, \u(2)$ defined by
 $$
  \varphi \left (\;A\;\oplus\;a\;\oplus\;B \;\right ) :=  s\,A\;+\;B
 $$
 with $A\,\oplus\,a\,\oplus\,B\,\in\,\u(2)\,\oplus\,\C^2\,\oplus\,\su(2)\,=\,\su(\,3\,)\,\oplus\,\su(2)$.
 By construction, $\varphi|_{\frakm}\,=\,0$ and $\varphi|_{\frakh}$ defines
 an isomorphism of $\frakh$ with $\u(2)$. We have
 $$
  \varphi(\,[\,A\,\oplus\,a,\,B\,\oplus\,b\,]\,)
  \;\;=\;\;
  s(s\,+\,1) [A,\,B] - s(a\,b^H\,-\,b\,a^H).
 $$
 On the other hand, every element $X\in\frakh\,\cong\,\u(2)$ has a
 decomposition $X\,=\,U\,+\,\alpha\,i\, \id$ with $U\,\in\,\su(2)$ and
 $\alpha\,=\,\frac{1}{2\,i}\,\trace(X)\in \R$. Hence,
 $$
  \varphi(X\;\oplus\;0\;\oplus\;U)
  \;\;=\;\;
  s\,X\,+\,U
  \;\;=\;\;
  (s\,+\,1) U\,+\,\frac s2\,\trace(X)\,\id
 $$
 Since $\trace(a\,b^H\,-\,b\,a^H)\,=\,-\,(a^H\,b\,-\,b^H\,a)
 \,=\,-\,2\,i\, \Im(\,a^H\,b\, )$, the trace--free part of
 $a\,b^H\,-\,b\,a^H\in\u(\,2\,)$ is given by
 $$
  a\,*\,b
  \;\;:=\;\;
  a\,b^H\,-\,b\,a^H\,+\,i\,\Im(\,a^H\,b\,)\,\id
  \;\;\in\;\;
  \su(2)
 $$
 Thus, $\varphi(\,[\,A\,\oplus\,a,\,B\,\oplus\,b\,]\,)\, \,\stackrel!=\,
 \varphi(X\,\oplus\,0\,\oplus\,U)$ if and only if
 \begin{eqnarray*}
  s\,(s\,+\,1) [A,\,B] -s\, a\,*\,b\,
  &\stackrel!=&
  \, (s\,+\,1)U
  \\
  2\,i\,\Im(\,a^H\,b\, )\,
  &\stackrel!=&
  \,\trace(X)
 \end{eqnarray*}
 We conclude that 
 \begin{eqnarray*}
  U
  &=&
  s \,[A,\,B]\,-\,\frac{s}{s\,+\,1}a\,*\,b
  \\
  \alpha
  &=&
  \Im(\,a^H\,b\, )
 \end{eqnarray*}
 Therefore, the residual bracket $[\,\cdot,\,\cdot\,]_\frakm\,:\,\frakm
 \,\times\,\frakm\, \longrightarrow \,\frakm$ is given by 
 $$
  [A\;\oplus\;a\;,\;B\;\oplus\;b\;]_\frakm
  \;\;=\;\;
  \Big(\;(1-s)\,[\,A,\,B\,]\,-\,\frac{1}{s\,+\,1}a\,*\,b\;\Big)
  \;\oplus\;\Big(\;Ab\;-\;Ba \;\Big )
 $$
 Due to the identities $\ad(\,A \,)^2\,B\,=\,4\,(\langle\,A,\,B\,\rangle
 \,-\,\|A\|^2\,)$ and $A^2b\,=\,-\|A\|^2\,b$ for $A\,\oplus\,B\,\oplus\,b
 \,\in\,\su(2)\,\oplus\,\su(2)\,\oplus\,\C^2$ we obtain that the square
 of the residual bracket with $A\,\in\, \frakm$ equals the endomorphism
 $$
  B\;\oplus\;b\;\longmapsto\;
  \Big(\;4\,(1-s)^2(\,\langle A,\,B\;\rangle\,-\,\|A\|^2\,)B\;\Big)
  \;\oplus\;\Big(\;-\,\|\,A\,\|^2\,b\;\Big)
 $$
 Thus, \eqref{eq:VCP1} holds if only if $s=\frac32$ or $s = \frac12$.
 Then, we have $c^2\,=\,\frac52$ and $c^2\,=\,\frac{3}{2}$, respectively.

 \bigskip
 \noindent
 For these values of $s$ and $c^2$ we can also verify~\eqref{eq:VCP2}
 as follows. We calculate that  the square of the residual bracket
 with $(\,0,\, a\,)\in\,\frakm$ equals the endomorphism
 \begin{eqnarray*}
  B\;\oplus\;b
  &\longmapsto&
  -\,\frac1{1+s}\;
  \Big(\;B\,a\,a^H\,+\,a\,a^HB\,-\,i\,\Im\,\trace(B\,a\,a^H)\,\id\,)\;\Big)
  \\[-4pt]
  &&
  \qquad\qquad\oplus\;\Big(\;-\,\frac{1}{1+s}\,(\,b\,a^H\,-\,a\,b^H
  \,-\,i\,\Im(a^H\,b)\,\id\,)\,a\;\Big)
 \end{eqnarray*}
 Using the identity derived above for the Hermitian matrix
 $X\,=\,aa^H$ with $\trace\,X\,=\,a^Ha\,=:\,\|\,a\,\|^2$,
 we see that $[a,\,[\,a,\,\cdot]]_\frakm$
 is the following expression:
 $$
  B\,\oplus\,b\;\longmapsto\;\Big(\;-\,\frac{1}{1+s}\,\|\,a\,\|^2\,B\;\Big)
  \;\oplus\;\Big(\;-\,\frac{1}{1+s}\,\|\,a\,\|^2\,b
  \,+\,\frac{1}{1+s}\,\langle\,  a,\,b\,\rangle\,)\,a\;\Big)
 $$
 This shows that also~\eqref{eq:VCP2} holds for both values of $s\,=\,\frac32$
 and $s\,=\,\frac12$ with the same constants $c^2$ as before. For~\eqref{eq:VCP3} we calculate
 \begin{eqnarray*}
  T_a\circ T_A
  \;\;=\;\;
  B\,\oplus\,b\;
  &\longmapsto&
  -\,\frac1{s\,+\,1}\;a\,*\,A\,b\;\oplus\;+\;(\,s\,-\,1\,)\;[\,A,\,B\,]a
  \\
  T_A\circ T_a
  \;\;=\;\;
  B\,\oplus\,b
  &\longmapsto&
  \frac{s\,-\,1}{s\,+\,1}\;[\,A\,,a\,*\,b\,]\;\oplus\;-\;A\,B\,a
 \end{eqnarray*}
 In particular:
 \begin{eqnarray*}
  \big(\;T_a\circ T_A + T_A\circ T_a\;\big)\;b
  &=&
  \Big(\;-\,\frac{1}{s\,+\,1}\;a\,*\,A\,b
  \;+\;\frac{s\,-\,1}{s\,+\,1}[\,A,\,a\,*\,b\,]\;\Big)\;\oplus\;0
  \\
  \big(\;T_a\circ T_A +T_A\circ T_a\;\big)\;B
  &=&
  0\;\oplus\;\;\Big(\,s\,-\,1\,)\;[\,A,\,B\,]a\,-\,A\,B\,a\;\Big)
 \end{eqnarray*}
 Setting $s\,:= \,\frac12$ in the second line, we obtain
 $$
 \big (\;T_a\circ T_A +T_A\circ T_a\;\big )\;B\;= \; 0\;\oplus\; \;\Big
 ( \;-\,\frac12\,[\,A,\,B\,]\,a\,-\;A\,B\,a\;\Big ) .
 $$
 If $\{\,A,\,B\,\}$ is an orthonormal system of $\su(2)$, then 
 $$
  -\,\frac12\,[\,A,\,B\,]\,a\,-\;A\,B\,a
  \;\;=\;\;
  -2\,A\,B\,a\;\;\neq\;\;0
 $$
 We conclude that the identity $(\,T_a\circ T_A\, +\,T_A\circ T_a\,)\,B
 \,=\,\langle\,A,\,B\,\rangle\,a$ fails.
 Let us verify that~\eqref{eq:VCP3} holds for $s\,:= \,\frac32$ and
 $c^2\,:=\,\frac{5}{2}$. Thus, we have to show that
 \begin{equation}\label{eq:Unsere_letzte_Gleichung}
  \big(\,T_a\circ T_A\,+\,T_A\circ T_a\,\big)\,
  (\,B\;\oplus\;b\,)
  \;\;=\;\;
  \frac25
  \,\langle\,a,\,b\,\rangle\,A\;\oplus\;\langle\,A,\,B\,\rangle\,a
 \end{equation}
 For this: With our values for $s$ and $c^2$ we have
 \begin{eqnarray*}
  \big(\;T_a\circ T_A\,+\,T_A\circ T_a\;\big)\;b
  &=&
  \Big(\;-\,\frac25\;a\,*\,A\,b\;+\;\frac15
   [\,A,\,a\,*\,b\,]\;\Big) \;\oplus\;0
  \\
  \big(\;T_a\circ T_A\,+\,T_A\circ T_a\;\big)\;B
  &=&
  0\;\oplus\;\Big(\;-\,\frac12\;\big(\;A\,B\,+\,\,B\,A\;\big)\,a\;\Big)
  \\
  &=&
  0\;\oplus\;-\frac12\,\trace\big (A\,B\;\big)\,a
 \end{eqnarray*}
 In order to simplify the first equation, we note that $*$ is an $\SU(\,2\,)$-equivariant map in the
 sense that $(\,A\,a)\,*\,(A\,b)\,=\,a\,*\,b$ for every $A\,\in\,\SU(\,2\,)$.
 Hence, the Leibniz rule $[\,A,\,a\,*\,b\,]\,=\,(\,A\,a\,)\,*\,b\,+\,
 a\,*\,A\,b$ holds for all $A\,\in\,\su(2)$. We conclude that
 \begin{eqnarray*}
  \big(\;T_a\circ T_A\,+\,T_A\circ T_a\;\big)\;b
  &=&
  \Big(\,\frac{1}{5}\;(\;(\,A\,a\,)\,*\,b
  \,-\,\;a\,*\,(\,A\,b\,)\;\Big )\;\oplus\;0
  \\
  &=&
  \frac25\;\langle\,a\,,b\,\rangle\, A\;\oplus\;0
 \end{eqnarray*}
 where the last equality follows from:
 \begin{eqnarray*}
  \lefteqn{(\,A\,a\,)*\,b\,-\,a\,*(\,A\,b\,)}
  &&
  \\
  &=&
  (\,(Aa)\,b^H\,-\,b\,(Aa)^H)\,+\,
  i\,\Im(\,(Aa)^H\,b\,)\,\id\,-\,\left((a\,(Ab)^H\,-\,(Ab)\,a^H)
  \,+\,i\,\Im(\,(a^HAb\,)\,\id\right)
  \\
  &=&
  A\,(a\,b^H\,+\,b\,a^H)\,+\,(b\,a^H\,+\,a\,b^H)\,A
  \,-\,\trace\,(\,a^H\,A\,b\, + \,b^H\,A\,a \,)\,\id\,
  \\
  &=&
  2\,\Re(a^H\,b)\,A
  \;\;=\;\;
  2\,\langle\,a,\,b\,\rangle\, A
 \end{eqnarray*}
 due to Remark~\ref{re:CH} with $X\,:=\,a\,b^H\,+\,b\,a^H$.
 We obtain~\eqref{eq:Unsere_letzte_Gleichung}.\qed
\bibliographystyle{amsplain}
\end{document}